\documentclass{article}
\usepackage{amsmath,amsthm,amssymb,amscd}
\usepackage[all,cmtip,color]{xy}
\usepackage[german,english]{babel}
\usepackage{amsmath}
\usepackage{amssymb}
\usepackage{hyperref}
\usepackage{mathrsfs} 
\usepackage{amssymb}
\usepackage{todonotes}
\usepackage{amssymb}
\usepackage[all,cmtip]{xy}
\usepackage{hyperref}
\usepackage[margin = 2.5 cm]{geometry}
\usepackage{colonequals,enumerate}
\usepackage{comment}

\usepackage{authblk}
	
\usepackage[export]{adjustbox}
\usepackage{graphicx}

\DeclareMathOperator{\rig}{rig}

\DeclareMathOperator{\Hom}{\mathsf{Hom}}

\DeclareMathOperator{\Bun}{\mathsf{Bun}}

\DeclareMathOperator{\ev}{ev}
\DeclareMathOperator{\reg}{reg}

\DeclareMathOperator{\sm}{sm}
\DeclareMathOperator{\codim}{codim}

\DeclareMathOperator{\inv}{\mathsf{inv}}
\DeclareMathOperator{\hash}{\#}

\DeclareMathOperator{\A}{{\mathsf{A}}}
\DeclareMathOperator{\B}{{\mathsf{B}}}

\DeclareMathOperator{\Ho}{\mathsf{Ho}}
\DeclareMathOperator{\Nb}{\mathbf{N}}

\DeclareMathOperator{\Sh}{Sh}

\DeclareMathOperator{\Cb}{\mathbb{C}}
\DeclareMathOperator{\Fb}{\mathbb{F}}
\DeclareMathOperator{\Xc}{\mathcal{X}}

\newcommand{\Cc}{\mathcal{C}}
\newcommand{\Gb}{\mathbb{G}}

\newcommand{\Ec}{\mathcal{E}}

\newcommand{\Hc}{\mathcal{H}}

\DeclareMathOperator{\Pb}{\mathbb{P}}

\DeclareMathOperator{\colim}{\mathsf{colim}}
\renewcommand{\lim}{\mathsf{lim}}

\DeclareMathOperator{\ani}{ani}

\DeclareMathOperator{\Inn}{Int}

\newcommand{\C}{\mathsf{C}}

\newcommand\toover[1]{\mathrel{\smash{\overset{#1}{\to}}}}
\newcommand\varto[1]{\mathrel{\hbox to #1pt{\rightarrowfill}}}

\def\isoto{\stackrel{\sim}{\longrightarrow}}

\DeclareMathOperator{\id}{id}

\DeclareMathOperator{\Tt}{\mathcal{T}}

\DeclareMathOperator{\Lie}{Lie}
\DeclareMathOperator{\Ind}{\mathsf{Ind}}

\DeclareMathOperator{\Zb}{\mathbb{Z}}

\DeclareMathOperator{\Xb}{\mathbb{X}}
\DeclareMathOperator{\Mb}{\mathbb{M}}

\DeclareMathOperator{\F}{\mathcal{F}}

\DeclareMathOperator{\Gal}{Gal}
\DeclareMathOperator{\Aut}{\mathsf{Aut}}

\DeclareMathOperator{\G}{\mathbb{G}}

\DeclareMathOperator{\GL}{GL}
\DeclareMathOperator{\PGL}{PGL}
\DeclareMathOperator{\SL}{SL}
\DeclareMathOperator{\M}{{\mathsf{M}}}

\DeclareMathOperator{\D}{\mathsf{D}}

\DeclareMathOperator{\Tb}{\mathbb{T}}

\newcommand{\Mc}{\mathcal{M}}

\DeclareMathOperator{\Tr}{\mathsf{Tr}}
\DeclareMathOperator{\Fr}{Fr}

\DeclareMathOperator{\Gg}{\mathcal{G}}

\DeclareMathOperator{\Pc}{\mathcal{P}}

\DeclareMathOperator{\vol}{\mathsf{vol}}
\DeclareMathOperator{\Spec}{\mathsf{Spec}}

\DeclareMathOperator{\Hhom}{\underline{\Hom}}
\DeclareMathOperator{\Oo}{\mathcal{O}}

\renewcommand{\top}{\mathsf{top}}
\DeclareMathOperator{\ord}{ord}

\DeclareMathOperator{\Tc}{\mathcal{T}}
\DeclareMathOperator{\ad}{\mathsf{ad}}

\DeclareMathOperator{\Qb}{\mathbb{Q}}

\newcommand{\BA}{{\mathbb{A}}}
\newcommand{\BB}{{\mathbb{B}}}

\newcommand{\BF}{{\mathbb{F}}}
\newcommand{\BG}{{\mathbb{G}}}
\newcommand{\BH}{{\mathbb{H}}}

\newcommand{\BM}{{\mathbb{M}}}

\newcommand{\BQ}{{\mathbb{Q}}}

\newcommand{\BT}{{\mathbb{T}}}

\newcommand{\BW}{{\mathbb{W}}}
\newcommand{\BX}{{\mathbb{X}}}

\newcommand{\BZ}{{\mathbb{Z}}}

\newcommand{\CE}{{\mathcal E}}
\newcommand{\CF}{{\mathcal F}}

\newcommand{\CM}{{\mathcal M}}

\newcommand{\CO}{{\mathcal O}} 

\DeclareMathOperator{\Res}{\mathsf{Res}}
\DeclareMathOperator{\Nm}{\mathsf{Nm}}
\DeclareMathOperator{\charac}{\mathsf{char}}
\DeclareMathOperator{\AJ}{\mathsf{AJ}}
\DeclareMathOperator{\Pic}{\mathsf{Pic}}

\DeclareMathOperator{\Out}{Out}

\newcommand{\wpc}{\widetilde{\mathcal{P}} }
\newcommand{\wmc}{\widetilde{\Mc} }
\newcommand{\wac}{\widetilde{\A} }
\newcommand{\wm}{\widetilde{M} }
\newcommand{\g}{\mathfrak{g} }
\newcommand{\FA}{\A}
\newcommand{\cc}{\mathfrak{c} }
\DeclareMathOperator{\diag}{diag}

\DeclareMathOperator{\Ext}{Ext}

\newcommand{\defeq}{\colonequals}
\let\into\hookrightarrow

\theoremstyle{definition}
\newtheorem{definition}{Definition}[section]
\newtheorem{construction}[definition]{Construction}
\newtheorem{rmk}[definition]{Remark}

\theoremstyle{plain}
\newtheorem{theorem}[definition]{Theorem}
\newtheorem{proposition}[definition]{Proposition}
\newtheorem{corollary}[definition]{Corollary}

\newtheorem{lemma}[definition]{Lemma}

\newtheorem{example}[definition]{Example} 
\newtheorem{claim}[definition]{Claim}

\newtheorem{situation}[definition]{Situation}

\newtheorem{goal}[definition]{Goal}

\begin{document}
\title{Geometric stabilisation via $p$-adic integration\let\thefootnote\relax\footnotetext{M. G. was funded by a Marie Sk\l odowska-Curie fellowship: This project has received funding from the European Union's Horizon 2020 research and innovation programme under the Marie Sk\l odowska-Curie Grant Agreement No. 701679. D.W. was supported  by the Foundation Sciences Math\' ematiques de Paris, as well as a public grant overseen by the French National Research Agency (ANR) as part of the \emph{Investissements d'avenir} program (reference: ANR-10-LABX-0098) and also by ANR-15-CE40-0008 (D\'efig\'eo). P.Z. was supported by the Swiss National Science Foundation. \\
\includegraphics[height = 1cm,right]{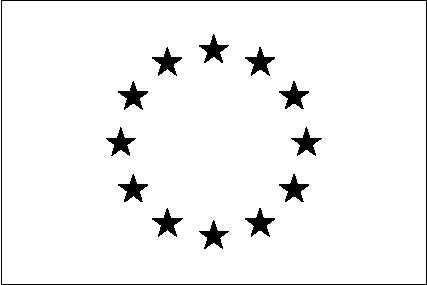} 
} 
}

\author[1]{Michael Groechenig\thanks{\url{michael.groechenig@utoronto.ca}}}
\author[2]{Dimitri Wyss\thanks{\url{dimitri.wyss@epfl.ch}}}
\author[3]{Paul Ziegler\thanks{\url{paul.ziegler@maths.ox.ac.uk}}}
\affil[1]{Department of Mathematics, University of Toronto}
\affil[2]{\'Ecole polytechnique f\'ed\'erale de Lausanne}
\affil[3]{Mathematical Institute, University of Oxford}

\renewcommand\Authands{ and }
\maketitle 
\abstract{In this article we give a new proof of Ng\^o's \textit{Geometric Stabilisation Theorem}, which implies the Fundamental Lemma. This is a statement which relates the cohomology of Hitchin fibres for a quasi-split reductive group scheme $G$ to the cohomology of Hitchin fibres for the endoscopy groups $H_{\kappa}$.
Our proof avoids the Decomposition and Support Theorem, instead the argument is based on results for $p$-adic integration on coarse moduli spaces of Deligne-Mumford stacks. Along the way we establish a description of the inertia stack of the (anisotropic) moduli stack of $G$-Higgs bundles in terms of endoscopic data, and extend duality for generic Hitchin fibres of Langlands dual group schemes to the quasi-split case.
}

\tableofcontents

\section{Introduction}

The proof of the Fundamental Lemma of Langlands-Shelstad given by Ng\^o \cite{MR2653248} contains a spectacular amount of mathematics. The original conjecture was formulated as an equality between $\kappa$-orbital integrals for a reductive group $G$ and stable orbital integrals for its $\kappa$-endoscopic group $H$ over a non-archimedean local field. In the equicharacteristic case, Goresky, Kottwitz and MacPherson \cite{GMK04} gave a geometric interpretation of the Fundamental Lemma in terms of the cohomology of affine Springer fibres. Subsequently Waldspurger \cite{Wa06} showed that it was sufficient to prove the function field case and further reduced the Fundamental Lemma to its Lie algebra version. For the latter, Ng\^o found a global reformulation in terms of the Hitchin fibration and finally deduced the Fundamental Lemma from his Geometric Stabilisation Theorems \cite[Theorem 6.4.1, Theorem 6.4.2]{MR2653248} (see also Laumon--Ng\^o \cite{MR2434884} where the case of the unitary group is treated).

In this paper we give a new proof of geometric stabilisation using $p$-adic integration along the Hitchin fibration. Our main motivation comes from the work of Hausel and Thaddeus \cite{MR1990670}, where they prove that the Hitchin fibrations for the Langlands dual groups $\SL_n$ and $\PGL_n$ are dual in the sense of abelian varieties, and conjecture a relation between the cohomology of these spaces (in \emph{loc. cit.} this is referred to as SYZ-duality). It was already observed by Hausel \cite[Section 5.4]{hausel2011global}, that their conjecture was closely related to the Geometric Stabilisation Theorem for $\SL_n$. In \cite{gwz} we established their conjecture via $p$-adic integration relying heavily on the aforementioned duality, which by work of Donagi-Pantev \cite{MR2957305} and Chen-Zhu \cite{chenzhu} also holds for general pair of Langlands dual groups $(G, \widehat{G})$. In the the present work we thus extend our methods from \cite{gwz} to general $(G, \widehat{G})$ and infer geometric stabilisation for $G$. This sheds new light on the algebro-geometric origins of Geometric Stabilisation, and thus the Fundamental Lemma. Furthermore, our proof does neither rely on the Decomposition Theorem nor on codimension estimates for the Hitchin system.


\subsubsection*{Statement of the main results}
We start by briefly introducing the Hitchin fibration, for a more detailed account we refer the reader to Section \ref{HiggsBundles} and \cite[Section 4]{MR2653248}. 

Let $X$ be a smooth, projective, geometrically connected curve of genus $g$ over a finite field $k$ of size $q$ and $D$ a line bundle of degree even $d$ on $X$ (see Section \ref{proofgst} for precise conditions on $d$). Let $G \to X$ be a quasi-split reductive group scheme on $X$ and $\g \to X$ its sheaf of Lie algebras. 
We assume the existence of a base point $\infty \in X(k)$, such that the group scheme $G_{\infty}$ on $\Spec k$ splits. Furthermore, we assume that $p = \mathrm{char}(k)$ is sufficiently large (see Subsection \ref{sub:overview} for more details).

 A $G$-Higgs bundle with coefficients in $D$ on $X$ is a pair $(E,\theta)$ where $E$ is a $G$-torsor on $X$, and $\theta$ a section of $\ad E \otimes D$, where $\ad E = E \times^{G} \g$ denotes the induced sheaf of Lie algebras of infinitesimal automorphisms of $E$. We denote by $\Mb_G = \Mb_G(X,D)$ the stack of $G$-Higgs bundles. The latter stack is a gerbe over the rigidification $\Mc_G$ by the Weil restriction of the centre $Z(X,G)$. Most of our constructions will be based on the rigidified stack of $G$-Higgs bundles $\Mc_G$.


The Chevalley map $\g \to \cc = \mathfrak{h}/W$ gives rise to a map from $\Mc_G$ to an affine space $\A_G$, called the Hitchin base. This map is referred to as the Hitchin fibration
\begin{equation}\label{hitfib}\chi: \Mc_G \to \FA_G.\end{equation}
Finally, there is a commutative group scheme $\Pc_G \to \FA_G$, related to the sheaf of regular centralisers. There is a natural action of $\Pc_G$ on $\Mc_G$, which over a dense open $\FA_G^{\Diamond} \subset \FA_G$ makes $\Mc_G$ into a $\Pc_G$-torsor. Geometric stabilisation phenomena occur on a slightly larger open subset, the anisotropic Hitchin base $\FA_G^{\Diamond} \subset \FA_G^{\ani} \subset \FA_G$ (see Definition \ref{defi:ani}). For technical reasons we further pass to an \'etale open covering $\widetilde{\FA}_G \to \FA_G^{\ani}$ (see Definition \ref{defi:tilde}). 

For every $a \in \widetilde{\FA}_G$ let $\widetilde{\Mc}_{G,a}$ and $\widetilde{\Pc}_{G,a}$ denote the respective fibres over $a$. The action of $\widetilde{\Pc}_{G,a}$ on the compactly supported $\ell$-adic cohomology $H^*_c(\widetilde{\Mc}_{G,a}, \overline{\Qb}_\ell)$ factors through the component group $\pi_0(\widetilde{\Pc}_{G,a})$ and hence we can define for every character $\kappa: \pi_0(\widetilde{\Pc}_{G,a}) \to \overline{\Qb}_\ell^\times$ the $\kappa$-point count
\begin{equation}\label{ladicpc}\#^\kappa \widetilde{\Mc}_{G,a}(k) = \sum_n (-1)^n \mathrm{tr}\left(\text{Frob}_k,H_c^n(\widetilde{\Mc}_{G,a}, \overline{\Qb}_\ell)^\kappa\right),\end{equation}
where $H_c^*(\widetilde{\Mc}_{G,a})^\kappa$ denotes the $\kappa$-isotypical component of $H_c^*(\widetilde{\Mc}_{G,a}, \overline{\Qb}_\ell)$. For $\kappa = \mathrm{1}$ the trivial character we write 
\[\#^{stab} \widetilde{\Mc}_{G,a}(k) = \#^{\mathrm{1}} \widetilde{\Mc}_{G,a}(k).\]

For an endoscopic group $H \to X$ of $G$ one has a closed immersion of Hitchin bases $\widetilde{\FA}_{H} \to \widetilde{\FA}_G$ and furthermore for every $a \in \widetilde{\FA}_{H}(k) \subset \widetilde{\FA}_{G}(k)$ a character $\kappa_a: \pi_0(\widetilde{\Pc}_{G,a}) \to \overline{\Qb}_\ell^\times$. The main result of this article is a new proof of Ng\^o's \cite[Theorem 6.4.2]{MR2653248}. 

\begin{theorem}[Corollary \ref{cor:geometric_endoscopy} \& \cite{MR2653248} Theorem 6.4.2]\label{gmst} Suppose that we are under the assumptions of Corollary \ref{cor:geometric_endoscopy}. For every $a \in \widetilde{\FA}_H(k)$ we have 
\[\#^{\kappa_a} \widetilde{\Mc}_{G,a}(k) = q^{r_H^G(D)}\#^{stab} \widetilde{\Mc}_{H,a}(k), \]
where $r_H^G(D)$ denotes $\frac{1}{2}(\dim \widetilde{\Mc}_{G} - \dim \widetilde{\Mc}_{H})$.
\end{theorem}

We prove this in Corollary \ref{cor:geometric_endoscopy}. As Ng\^o explains in \cite[Section 8]{MR2653248}, Theorem \ref{gmst} implies the Fundamental Lemma. We note that in \emph{loc. cit.} this result is stated for the stacks $\Mb_G$. However, the following remark applies:

 \begin{rmk}\label{rmk:cohomology}
 Since $\Mb_{G,a}$ is a gerbe over $\Mc_{G,a}$ banded by a finite \'etale group scheme, the $\Qb_\ell$-adic cohomology groups of $\Mb_{G,a}$ and $\Mc_{G,a}$ are canonically isomorphic. In particular, the two stacks have the same number of $k$-rational points over a finite field $k$.
 \end{rmk}

In our $p$-adic approach it is not the endoscopic group $H$ that appears naturally but rather its dual group $\widehat{H}$. Thus our proof of Theorem \ref{gmst} relies on a geometric version of Walspurger's non-standard Fundamental Lemma for dual groups \cite{Wa08}, which we can also prove with our methods.

\begin{theorem}[Corollary \ref{nonstd} \& \cite{MR2653248} Theorem 8.8.2 ]\label{walem} For every $a \in \widetilde{\FA}_G(k) \cong \widetilde{\FA}_{\widehat{G}}(k)$ we have
\[ \#^{stab} \widetilde{\Mc}_{G,a}(k)= \#^{stab} \widetilde{\Mc}_{\widehat{G},a}(k).\]
\end{theorem}

\subsubsection*{Structure of the proof}

In order to explain the proofs of Theorem \ref{gmst} and \ref{walem} we first summarize the $p$-adic integration theory of Section \ref{padicint}. This is a generalisation of \cite[Section 4]{gwz}, which is motivated by \cite{Ba99,DL2002}, and also \cite{Ya06}. The origin of the theory of $p$-adic integration is the following result of Weil (\cite[Theorem 2.2.5]{weil2012adeles}). Consider a local field $F$ with ring of integers $\Oo_F$ and residue field $k = \Fb_q$. For a scheme $X/\Oo_F$ we denote by $e\colon X(\Oo_F) \to X(k)$ the specialisation map, induced by the map of affine schemes $\Spec k \to \Spec \Oo_F$.

\begin{theorem}[Weil] \label{weilthm}
Let $X/\Oo_F$ be a smooth scheme of relative dimension $d$. There exists a canonical measure $\mu_{can}$ on the set $X(\Oo_F)$, such that for a $k$-rational point $x \in X(k)$ we have
$$\vol(e^{-1}(x)) = \frac{1}{q^d}.$$
In particular, one has $\vol(X(\Oo_F)) = \frac{|X(k)|}{q^d}$.
\end{theorem}

Our article relies on a refinement of Weil's formula to varieties with tame quotient singularities which is the content of Section \ref{padicint}.
Let $\Mc / \Oo_F$ be a smooth and tame Deligne-Mumford stack satisfying the assumptions of Situation \ref{situation:locally_quotient}. Denote by $M$ the coarse moduli space of $\Mc$ and by $U \subset M$ the maximal open over which the quotient $\Mc \to M$ is an isomorphism. The key object for $p$-adic integration on $\Mc$ is the $F$-analytic manifold 
\[M(\Oo_F)^\natural = M(\Oo_F) \cap U(F),\]
as it admits a real-valued measure $\vol_{\Mc}$. Furthermore we construct a specialisation map
 \[e: M(\Oo_F)^\natural \to [I_{\widehat{\mu}}\Mc(k)],  \]
where $I_{\widehat{\mu}}\Mc(k)$ is a twisted version of the usual inertia stack, taking into account possibly non-constant automorphism groups, and the brackets $[ \cdot ]$ denote the set of isomorphism classes in a groupoid. The main result of Section \ref{padicint} is the following.
\begin{theorem}[Theorem \ref{thm:volume}]\label{canmi} For every $x \in I_{\widehat{\mu}}\Mc(k)$ we have
\[ \vol_{\Mc}\left(e^{-1}(x)\right) =  \frac{q^{-w(x)}}{|\Aut_{I_{\widehat \mu}\Mc(k)}(x)|},   \]
where $w: I_{\widehat{\mu}}\Mc(k) \to \Qb$ is the \textit{weight function}, see Definition \ref{defi:weight}. In particular, we see that 
$$\vol\left(M(\Oo_F)^{\natural}\right) = \#^wI_{\widehat{\mu}}\Mc(k),$$
where $\#^w$ is a weighted count taking both automorphisms and $w$ into account. 
\end{theorem}

We remark that the theorem above is well-known in the context of motivic integration (see Denef--Loeser and Yasuda \cite{DL2002,Ya06}), and that $p$-adic volumes of varieties with quotient singularities are also studied in Yasuda's \cite{Yasuda:2014aa}. However we do not know how to formally deduce Theorem \ref{canmi} from these results. The arguments of \emph{loc. cit.} served as a source of inspiration.

From now on on we let $F=k((t))$. By pulling back along $\Spec(\Oo) \to \Spec(k)$, the base curve $X$ as well as the corresponding Hitchin fibration $\chi: \Mb_G \to \FA_G $ are now defined over $\Oo = k[[t]]$ and the special fibre coincides with \eqref{hitfib}. Furthermore our $p$-adic integration theory applies to $\widetilde{\Mc}_G$, the rigidification of $\widetilde{\Mb}_G$ with respect to the generic automorphism group $Z(X,G)$ (see \cite[Theorem 5.1.5]{abramovich2003twisted} for the concept of rigidification). The measure $\mu_{G}$ on $\widetilde{M}_G(\Oo)^\natural$ induces for every $b \in \widetilde{\FA}^\Diamond_G(F)$ a measure on the fibre $\widetilde{M}_{G,b}(F)$ and we have
\begin{equation}\label{stdcomp}  \frac{\mu_{can,G}\left(\widetilde{M}_{G,b}(F)\right)}{|H^0(X,Z(\widehat{G}))| } = \frac{\mu_{can,\widehat{G}}\left(\widetilde{M}_{\widehat{G},b}(F)\right)}{|H^0(X,Z(G))|}.\end{equation}
This is essentially a consequence of the duality result of \cite{MR2957305,chenzhu} for pairs of Langlands dual grups, which we extend to quasi-split reductive group schemes $G \to X$ in Theorem \ref{thm:duality_quasi-split}. The duality implies that the neutral connected components of $\widetilde{M}_{G,b}$ and $\widetilde{M}_{\widehat{G},b}$ are dual abelian varieties and thus have the same volume (see Proposition \ref{abvol}). 

By integrating the volume of the Hitchin fibres over all points $b \in \widetilde{\FA}^\Diamond_G(F) \cap \widetilde{\FA}_G(\Oo)$ which restrict to a fixed point $a \in \widetilde{\FA}_G(k)$, we obtain from \eqref{stdcomp} and Theorem \ref{canmi} the following special case of Theorem \ref{mainidentity}:
\begin{equation}\label{incount} \frac{\#^w I_{\widehat{\mu}}\widetilde{\Mc}_{G,a}(k)}{{|H^0(X,Z(\widehat{G}))| }} = \frac{\#^w I_{\widehat{\mu}}\widetilde{\Mc}_{\widehat{G},a}(k)}{{|H^0(X,Z(G))|}},  \end{equation}
Let us stress here, that formula \eqref{incount} is derived by only considering generic fibres (of the $p$-adic fibration), which is the key point in the way we use $p$-adic integration. In Subsection \ref{ExampleSection} we also give a toy example for this argument.

The key points of the argument are contained in Subsections \ref{set-up}-\ref{character}, and the main computation is given in Subsection \ref{character}. We recommend that a reader familiar with the theory of Higgs bundles, starts reading these subsections, before delving into the more technical parts of the paper. 

Section \ref{secinst} is devoted to the study of the groupoid $I_{\widehat{\mu}}\widetilde{\Mc}_{G}(k)$. Under some extra assumptions we prove in Theorem \ref{IMTheoremTwisted} an equivalence  

\begin{equation}\label{tinstack}  I_{\hat\mu}\widetilde{\mathcal{M}}_G(k) \cong \bigsqcup_{\CE} \widetilde{\mathcal{M}}^{G-\infty,G-\rig}_{H_\CE}(k), \end{equation}
where the disjoint union runs over coendoscopic data $\CE$ for $G$, which are essentially endoscopic data for $\widehat{G}$, and the coendsoscopic groups $H_\CE$ are dual to the corresponding endoscopic groups for $\widehat{G}$. For the superscripts $G-\infty,G-\rig$ see Constructions \ref{TildeMuCons} and \ref{grig}. This picture is reminiscent of Frenkel--Witten's \cite{MR2417848}.

Combining \eqref{incount} and \eqref{tinstack} we obtain an equality between two sums of point counts of Hitchin fibres, ranging over a set of coendoscopic data for $G$ and $\widehat{G}$ respectively (see \ref{mainidentity}). This is the prototypical example of an equality we can obtain through $p$-adic integration.

In order to prove Theorems  \ref{walem} and \ref{gmst} we need twisted versions of \eqref{stdcomp}, which we develop in Section \ref{proofgst}. Similar to the proof of the refined topological mirror symmetry conjecture \cite[Theorem 7.24]{gwz}, the idea is to consider for an $a \in \widetilde{\FA}_G(k)$ and $t \in \pi_0(\widetilde{\Pc}_{G,a})$ the stack $\widetilde{\Mc}_{G,a}^t=\widetilde{\Mc}_{G,a} \times^{\widetilde{\Pc}_{G,a}} T_t $, where $T_t$ is a $\widetilde{\Pc}_{G,a}$-torsor representing $t$ via the isomorphisms $H^1(k,\widetilde{\Pc}_{G,a}) \cong H^1(k,\pi_0(\widetilde{\Pc}_{G,a})) \cong \pi_0(\widetilde{\Pc}_{G,a})$. To compensate this twist on the $\widehat{G}$-side we introduce a natural function on $\chi_t:I_{\widehat{\mu}}\widetilde{\Mc}_{\widehat{G},a}(k) \to \Cb$ and show that for every $b \in \widetilde{\FA}^\Diamond_G(F) \cap \widetilde{\FA}_G(\Oo)$ restricting to $a$ over $\Spec(k)$ the function $\chi_s \circ e$ on $\widetilde{\Mc}_{G,b}(F)$ can be interpreted as the Hasse invariant of a certain $\Gb_m$-gerbe on $\widetilde{\Mc}_{G,b}$. Using the non-degeneracy of the Tate-duality pairing for abelian varieties we are able to derive in Theorem \ref{mainidentity} the twisted versions of \eqref{stdcomp} and \eqref{incount} needed to prove first Theorem \ref{walem} and then \ref{gmst}. This part relies on the results of Section \ref{sec:hasse} which gives a \emph{stack-theoretic} interpretation of the \emph{Hasse invariant} (generalising \cite[Section 5]{gwz}).

In this article we consider the Fundamental Lemma for Lie algebras as studied in \cite{MR2653248} with the intention to showcase the versatility of $p$-adic integration. We remark that since the appearance of \emph{loc. cit.} more general versions were studied in the literature, see for example Chaudouard--Laumon \cite{MR2735371,Chaudouard:2009uq}.

\subsubsection*{Acknowledgments} We thank Pierre-Henri Chaudouard, H\'el\`ene Esnault, Tamas Hausel, Jochen Heinloth, Michael McBreen, Ng\^o B\`ao Ch\^au, Yiannis Sakellaridis and Arul Shankar for pleasant conversations and their interest in our project. We also thank the anonymous referee for helpful comments and suggestions. In particular for pointing out a gap in the previous version of the proof of \ref{abvol} and for explaining how to resolve it.

\section{p-adic integration}\label{padicint}
We fix a local field $F$ with ring of integers $\CO_F$ and residue field $k=k_F$ of order $q=p^r$. We also fix an algebraic closure $\bar F$ of $F$ and consider the intermediate field extensions
\begin{equation*}
  F \subset F^\text{un} \subset F^{tr} \subset \bar F
\end{equation*}
given by maximal unramified (respectively tamely ramified) extension of $F$ in $\bar F$.

In this section we introduce our formalism for $p$-adic integration on a smooth and tame Deligne-Mumford stack $\Mc$ with coarse moduli space $M$, satisfying the assumptions in \ref{situation:locally_quotient}. The main result is Theorem \ref{thm:volume} which computes the volume of a fibre of the specialisation map $e$ (Construction \ref{defi:e}) with respect to the canonical measure on $M(\Oo_F)$ introduced in Definition \ref{defcanmes}.

\subsection{Integrating $p$-adic differential forms}

 We write $|\cdot|$ for the non-archimedean norm on $F$ and $\mu_F^n$ for the Haar measure on $F^n$ with the usual normalisation $\mu_F^n(\Oo_F^n)=1$. Analytic Manifolds and differential forms over $F$ are essentially defined the same way as over the real numbers, as explained in \cite{MR1743467}. 
Given an $n$-dimensional manifold $X$ over $F$ and a global section $\omega$ of $(\Omega^n_X)^{\otimes r}$ we can define a measure $d\mu_\omega$ as follows: Given a compact open chart $U \hookrightarrow F^n$ of $X$ and an analytic function $f: U \rightarrow F$, such that $\omega_{|U} = f(x)(d x_1 \wedge d x_2 \wedge\dots \wedge d x_n)^{\otimes r}$ we set
\[ \mu_{\omega}(U)  = \int_U |f|^{1/r} d\mu_{F^n}.\]
This extends to a measure on $X$ as in \cite[3.2]{Yasuda:2014aa}.

Starting with any geometrically reduced algebraic space $X$ of finite type over $\Oo_F$ one has the following construction, which can be found for example in \cite[Section 4]{Yasuda:2014aa}. Traditionally, the literature contains the superfluous assumption that $X$ be a scheme. The case of algebraic spaces is dealt with in exactly the same way.

First write $X_F=X\times_{\Spec(\Oo_F)} \Spec(F)$ and $X_{k}=X\times_{\Spec(\Oo_F)} \Spec(k)$. Let $X_F^{sm}$ be the smooth locus of $X_F$ and set
\[ X^\circ= X(\Oo_F) \cap X_F^{sm}(F),\]
where we think of $X(\Oo_F)$ as a subset of $X(F) = X_F(F)$. Then $X^\circ$ has naturally the structure of an analytic manifold over $F$. Thus we can integrate any section $\omega \in H^0(X_F^{sm},(\Omega^{top}_{X/F})^{\otimes r})$ on $X^\circ$. In this way we obtain a measure $\mu_\omega$ on $X^\circ$ which we extend by zero to all of $X(\Oo_F)$. The following two results will be essential for dealing with these measures.

\begin{proposition}[{\cite[Lemma 4.3, Theorem 4.8]{Yasuda:2014aa}\label{intools}}]

 \begin{enumerate}[(a)]
\item\label{measurezero} For a closed algebraic subspace $Y \subset X$ of positive codimension one has $\mu_{\omega}(Y(\Oo_F)) = 0$.
\item\label{chovg} Let $f:Y \rightarrow X$ be a morphism of geometrically reduced algebraic spaces of finite type over $\Oo_F$. Assume that $Y$ admits a generically stabiliser-free action by a finite \'etale group scheme $\Gamma$ over $\Oo_F$, that the morphism $f$ is $\Gamma$-invariant and that the induced morphism $Y/\Gamma \rightarrow X$ is birational. Then for any open $\Gamma$-invariant subset $A \subset Y(\Oo_F)$ and any section $\omega$ of $(\Omega_X^{\top})^{\otimes r}$ we have
\[ \frac{1}{|\Gamma(F)|} \int_{A} |f^*\omega|^{1/r} = \int_{f(A)} |\omega|^{1/r}.\]
\end{enumerate}
\end{proposition}

\begin{proof}
Statement (a) is stated in exactly the same form in \cite[Lemma 4.3]{Yasuda:2014aa} for schemes. The proof for algebraic spaces is the same. Statement (b) is a slightly more general version of \cite[Theorem 4.8]{Yasuda:2014aa}, the proof follows the exact same strategy: the morphism $f$ induces a map of $F$-analytic varieties $f^{\circ}\colon X^{\circ} \to Y^{\circ}$ which is a finite covering map of degree $|\Gamma(F)|$ of an open subset of $Y^{\circ}$. Therefore, we have
$\int_{A} |(f^\circ)^*\omega|^{1/r} = \int_{f^{\circ}(A)} |\omega|^{1/r}$ for every Borel measurable subset $A = f^{\circ,-1}(B) \subset X^{\circ}$.
\end{proof}

We finish this subsection with a comparison of $p$-adic volumes of dual abelian varieties, which lies at the heart of the comparison of $p$-adic volumes of moduli stacks of Higgs bundles. This generalises the key lemma \cite[Lemma 6.15]{gwz}, and in fact provides an alternative argument.

\begin{proposition}\label{abvol}
Let $A$ and $B$ be dual abelian varieties over $F$ and $\phi: A \to B$ an isogeny of degree coprime to $p$. Then for any $\omega \in H^0(B,\Omega_{B/F}^\text{top})$ we have
\[\int_{A(F)} |\phi^* \omega| = \int_{B(F)} |\omega|.\]
\end{proposition}

\begin{proof} Let $\mathcal{A}$ and $\mathcal{B}$ be the N\'eron models of $A$ and $B$. The form $\omega$ extends to a rational section of $\Omega_{\mathcal{B}/\Oo_F}^\text{top}$. By translation-invariance of this rational section, its pole-order along a connected component of the special fibre $\mathcal{B} \times_{\Oo_F} k_F$ is constant. The statement is invariant under scaling $\omega$ by elements of $F$, so we may assume that $\omega$ extends to a generating section of the trivial line bundle $\Omega_{\mathcal{B}/\Oo_F}^\text{top}$. We write $\mathcal{A}^0$ for the neutral component, $\Phi_{\mathcal{A}}$ for the component group scheme of the special fibre of $\mathcal{A}_{k}$ and similarly $\mathcal{B}^0$ and $\Phi_\mathcal{B}$. Then by \cite[7.3.6]{NeronModels} the isogeny $\phi$ extends to an isogeny $\phi:\mathcal{A} \to \mathcal{B}$, which is in particular \'etale. Hence $\phi^* \omega$ is again a generating section of $\Omega_{\mathcal{A}/\Oo_F}^\text{top}$ and by Weil's Theorem \ref{weilthm} the proposition is equivalent to showing $|\mathcal{A}(k)| = | \mathcal{B}(k)|$.

First by Lang's theorem we have $H^1(k,  \mathcal{A}^0)=0$, thus $|\mathcal{A}(k)| = |\mathcal{A}^0(k)||\Phi_{\mathcal{A}}(k)|$ and similarly for $\mathcal{B}$. Since $\phi$ induces an isogeny $\phi^0:\mathcal{A}^0 \to \mathcal{B}^0$ we have an exact sequence 
\[ 0 \rightarrow \ker(\phi^0)(k) \rightarrow  \mathcal{A}^0(k) \rightarrow   \mathcal{B}^0(k) \rightarrow H^1(k,\ker(\phi^0)) \rightarrow H^1(k,  \mathcal{A}^0) = 0.  \]
Since $\ker(\phi^0)$ is abelian we have $| \ker(\phi^0)(k)| = |H^1(k,\ker(\phi^0))|$ and thus by exactness also $|\mathcal{A}^0(k)| = |\mathcal{B}^0(k)|$. Finally in \cite[Proposition 4.3]{MR2961846}, Lorenzini shows $|\Phi_{\mathcal{A}}(k)|=|\Phi_{\mathcal{B}}(k)|$ using Grothendieck's paring $\Phi_{\mathcal{A}} \times \Phi_{\mathcal{B}} \to \BQ/\BZ$.

\end{proof} 

\subsection{Examples of $p$-adic integration} \label{ExampleSection}
In this subsection we give a toy example for our main argument using $p$-adic integration which will appear in Subsection \ref{PAdicComparison}. This subsection will not be used in the rest of the article and only serves to illustrate our argument in a simplified situation.

We consider the following situation: 
\begin{situation}
  Let $\M$ and $\A$ be two smooth $\Oo_F$-schemes together with a dominant morphism $\pi\colon \M \to \A$. We assume that $\pi$ is generically smooth and fix an open dense subset $\A^\lozenge \subset \A$ for which the restriction of $\pi$ to $\M^{\lozenge} \defeq \M \times_{\A} \A^\lozenge$ is smooth.
\end{situation}

By our assumptions the sheaf $\Omega^\top_{\M/\Oo_F}$ is a line bundle. By integrating local generating sections $\omega$ of this line bundle we obtain as above a measure $\mu_{\M}$ on $\M(\Oo_F)$. Since any other choice of $\omega$ differs from the given one by a section of $\Oo_{\M/\Oo_F}^*$ whose values on $\M(\Oo_F)$ will have $p$-adic norm one, this measure is independent of the choice of $\omega$. In the same way we obtain a measure $\mu_{\A}$ on $\A(\Oo_F)$.
 
We denote by $e\colon \M(\Oo_F) \to \M(k)$ the specialisation map. For a set of points $S \subset \M(k)$, by Weil's Theorem \ref{weilthm}, the cardinality of $S$ can be computed using the measure $\mu_{\M}$ as follows:
\begin{equation} \label{PointCount1}
  |S|=q^{\dim \M} \mu_{\M}(e^{-1}(S))
\end{equation}
The idea is now to compute the measure $\mu_{\M}(e^{-1}(S))$ fibrewise along the morphism $\pi$ using a Fubini-type statement:

The short exact sequence 
\begin{equation*}
  0 \to \pi^* \Omega^1_{\A^\lozenge/\Oo_F} \to \Omega^1_{\M^\lozenge/\Oo_F} \to \Omega^1_{\M^\lozenge/\A^\lozenge} \to 0
\end{equation*}
induces an isomorphism
\begin{align} \label{WedgeIso}
  \pi^*\Omega^\top_{\A^\lozenge/\Oo_F} \otimes \Omega^\top_{\M^\lozenge/\A^\lozenge} &\isoto \Omega^\top_{\M^\lozenge/\Oo_F} \\
  \eta \otimes \phi \mapsto \eta\wedge \phi. \nonumber
\end{align}
We consider the subset $\A(\Oo_F)^\flat \defeq \A(\Oo_F) \cap \A^\lozenge(F)$ of $\A(\Oo_F)$. This is an $F$-analytically open subset of $\A(F)$ whose complement has measure zero by \cite[Lemma 4.3]{Yasuda:2014aa}. Given $a \in \A(\Oo_F)$ we will write $a_F$ for its image in $\A(F)$.

Working locally on $\M$ and $\A$, as above we choose a local generating section $\omega$ of $\Omega^\text{top}_{\M/\Oo_F}$ and $\eta$ of $\Omega^\text{top}_{\A/\Oo_F}$. Then by the isomorphism \eqref{WedgeIso}, locally on $\M$ there exists a unique section $\phi$ of $\Omega^\top_{\M^\lozenge/\A^\lozenge}$ such that $\pi^*\eta \wedge \phi=\omega$. Then for any $a \in \A(\Oo_F)^\flat$ we can restrict $\phi$ to a top-degree differential form on the fibre $\M_{a_F}$ over $F$ and hence get a measure $\mu_{\phi}$ on $\M_a(\Oo_F)$. As before, since $\pi^*\eta$ and $\omega$ are well defined up to sections of $\Oo_{\M}^*$, so is $\phi$ and hence the measure $\mu_\phi$ is independent of these choices. Thus for every $a \in \A(\Oo_F)^\flat$ we obtain a canonical measure $\mu_a$ on $\M_a(\Oo_F)$. For $a \in \A(\Oo_F) \setminus \A(\Oo_F)^\flat$ we let $\mu_a$ be the zero measure on $\M_a(\Oo_F)$.

The differential forms appearing in the construction above may only be Zariski locally defined. However, in our applications below all these forms will be assumed to exist globally. Then the measure $\mu_a$ on the fibre is given by integrating the single differential form obtained by restricting $\phi$ to the fibre. 

Then by combining \eqref{PointCount1} with the Fubini theorem for $p$-adic integration (c.f. \cite[Theorem 7.6.1]{MR1743467}) we obtain the following formula:
\begin{equation}
  \label{PointCount2}
  |S| = q^{\dim \M} \int_{a \in \A(\Oo_F)^\flat} \mu_a(e^{-1}(S) \cap \M_a(\Oo_F)) d\mu_{\A}= q^{\dim \M} \int_{a \in \A(\Oo_F)} \mu_a(e^{-1}(S) \cap \M_a(\Oo_F)) d\mu_{\A}
\end{equation}

\begin{example} Consider the example
\begin{align*} \pi: \BA^2_{\Oo_F} &\rightarrow \BA^1_{\Oo_F} \\
(x,y) &\mapsto xy.
\end{align*}

We can use \eqref{PointCount2} to compute the number of points of the special fibre $|\pi^{-1}(0)(k)| = 2q-1$ as a $p$-adic double integral. We have $(\BA^1_{\Oo_F})^\lozenge = \BG_{m,\Oo_F}$ and for any point $a\in \BA^1(\Oo_F) \cap \BG_{m}(F)= \Oo_F \setminus \{0\}$, projection onto the $x$-variable defines an isomorphism of $F$-analytic manifolds
\begin{equation}\label{profib} \pi^{-1}(a)(\Oo_F) \cong \{ x \in  \Oo_F \ |\ |x| \geq |a|\}.    \end{equation}
To obtain the measure on the fibre $\pi^{-1}(a)$ we fix the standard volume forms $dx \wedge dy$ and $da$ on $\BA^2$ and $\BA^1$ and find
\[ \frac{dx \wedge dy}{da} = \frac{dx \wedge dy}{ydx + xdy} = \frac{dx}{x}, \]
under the isomorphism \eqref{profib}. Then we can compute the right hand side of \eqref{PointCount2} directly as

\begin{align*} q^2 \int_{a \in \mathfrak{m} \setminus \{0\} } |da| \int_{1\geq |x| \geq |a|} \frac{|dx|}{|x|} &=  q^2 \int_{a \in \mathfrak{m} \setminus \{0\} } (\nu(a)+1)(1-q^{-1}) da \\
&= q^2(1-q^{-1})^2(2q^{-1}+3q^{-2}+4q^{-3}+ \dots) = 2q-1.
\end{align*}
Here $\nu(a) \in \BZ$ denotes the valuation of $a$.
\end{example}

As an application of \eqref{PointCount2} we obtain the following:
\begin{theorem}
 Consider smooth $\Oo_F$-schemes $\M_i$ for $i=1,2$, a smooth $\Oo_F$-scheme $\A$ and proper surjective morphisms $\pi_i\colon \M_i \to \A$. Consider in addition smooth group schemes $\M_i^\text{reg}$ over $\A$ together with open immersions $\M_i^\text{reg} \into \M_i$ and an open dense subscheme $\A^\lozenge \subset \A$ such that the following conditions are satisfied:
 \begin{enumerate}[(i)]
 \item $\codim(\M_i \setminus \M_i^\text{reg} \subset \M_i) \geq 2$
 \item There exists an isogeny $\rho\colon (\M_1^\text{reg})^\circ \to (\M_2^\text{reg})^\circ$ between the relative connected components of the identity which has degree prime to $p$.
 \item The preimages $\M^\lozenge_i \defeq \M_i \times_{\A} \A^\lozenge$ of $A^\lozenge$ in $\M_i$ are contained in $\M_i^\text{reg}$. With respect to the resulting relative group scheme structure the schemes $\M_i^\lozenge$ are dual abelian schemes over $\A^\lozenge$.
 \end{enumerate}

Then for every point $a \in \A(k)$ we have 
\begin{equation*}
  |\M_{1,a}(k)|=|\M_{2,a}(k)|.
\end{equation*}
\end{theorem}
\begin{proof}
We choose an isogeny $\rho$ as in (ii).

  Since $\M_i^\text{reg}$ is a smooth group scheme, the relative sheaf of differentials $\Omega^1_{\M_i^\text{reg}/\A}$ is trivial (in fact it is trivialized by translation-invariant sections). Hence the same holds for the line bundle $\Omega^\top_{\M_i^\text{reg}/ \A}$. Let $\phi_2$ be a nowhere vanishing global section of $\Omega^\top_{\M_2^\text{reg}/ \A}$. Then $\phi_1\defeq \rho^*\phi_2$ is a nowhere vanishing global section of $\Omega^\top_{\M_1^\text{reg}/ \A}$ by the assumption on $\rho$. 

The claim is local for the Zariski topology on $\A$. Hence after replacing $\A$ by a suitable open subset we may assume that $\Omega^\top_{\A/\Oo_F}$ is trivial. We fix a global generating section $\eta$ of $\Omega^\top_{\A/\Oo_F}$. Then $\pi_i^* \eta \wedge \phi_i$ is a global generating section of $\Omega^\top_{\M_i^\text{reg}/\A}$. By assumption (i) and the Hartogs extension property this section uniquely extends to a global generating section $\omega_i$ of $\Omega^\top_{\M_i/\Oo_F}$. Thus by the above construction, for $b \in \A(\Oo_F)^\flat$, the measures $\mu_b$ on $M_{i,b}(\Oo_F)$ can be obtained by integrating the forms $\phi_i$.

By the assumptions we have $d\defeq \dim M_1=\dim M_2$. For $a\in \A(k)$ let $\A(\Oo_F)^\flat_a \subset \A(\Oo_F)^\flat$ be the set of those points which specialize to $a$. Then by applying \eqref{PointCount2} to the set $S=\M_{i,a}(k)$ we get
\begin{equation*}
  |\M_{i,a}(k)|=q^d \int_{b \in \A(\Oo_F)_a^\flat} \mu_b(M_{i,b}(\Oo_F)) d\mu_{\A} =q^d \int_{b \in \A(\Oo_F)_a^\flat} \mu_b(M_{i,b}(F)) d\mu_{\A},
\end{equation*}
where for the last step we have used the properness of the morphisms $\pi_i$. Thus it suffices to show that
\begin{equation*}
  \mu_b(\M_{1,b}(F)) =\mu_b(\M_{2,b}(F))
\end{equation*}
for all $b \in \A(\Oo_F)^\flat$. Using assumption (iii) and the fact that by the above the measure $\mu_b$ on $M_{i,b}(F)$ is given by integrating $\phi_i$, this is exactly the statement of Proposition \ref{abvol}.
\end{proof}
\subsection{The canonical measure for coarse moduli spaces of Deligne-Mumford stacks} \label{canmedm}

This subsection is devoted to a generalisation of Weil's Theorem \ref{weilthm} to certain algebraic spaces with quotient singularities.

\begin{situation}\label{situation:locally_quotient}
 We consider a smooth and tame Deligne-Mumford stack $\Mc/\Oo_F$. 
\begin{enumerate}[(a)]
\item Let $M/\Oo_F$ be a coarse moduli space of $\Mc$, and $Q\colon \Mc \to M$ the canonical map (see Keel-Mori \cite{MR1432041}). We denote by $V \subset M$ the maximal open for which $Q^{-1}(V) \to V$ is an isomorphism. 
\item Assume that there exists a finite covering of $\Mc$ by Zariski-open substacks $\Mc = \bigcup_{i \in I} \Mc_i$ and for every $i \in I$ we have an equivalence $\Mc_i = [U_i/\Gamma_i]$, where $\Gamma_i$ is a finite \'etale group $\Oo_F$-scheme, $U_i$ a smooth $\Oo_F$-scheme, the action of $\Gamma_i$ on $U_i$ is generically free and every orbit is contained in an affine subset. We say that $\Mc$ is Zariski-locally a finite \'etale quotient stack.
\end{enumerate}
\end{situation}

In the equicharacteristic case it follows from a result of Kresch \cite[Theorem 4.4 \& Proposition 5.2]{MR2483938} that a smooth and tame DM-stack (over a field) $\Mc_k/k$ with quasi-projective coarse moduli space is Zariski-locally a finite \'etale quotient stack and that moreover the $\Gamma_i$ may be taken to be constant.

\begin{definition} In the notation of Situation \ref{situation:locally_quotient} we define the $p$-adic manifold of generic $\Oo_F$-points $M(\Oo_F)^\natural$ by 
\[ M(\Oo_F)^\natural = M(\Oo_F) \cap V(F).\]
\end{definition}

The following lemma serves as a definition of the measure $\vol_{\Mc}$.

\begin{lemma}\label{defcanmes}
Let $\Mc/\Oo_F$ be a smooth tame DM-stack over $\Oo_F$ with coarse moduli space $M/\Oo_F$. We assume that $\Mc$ is Zariski-locally a finite \'etale quotient as in Situation \ref{situation:locally_quotient}. Then there exists a unique Borel measure $\vol_{\Mc}$ on $M(\Oo_F)$ with the following properties:
\begin{enumerate}[(a)]
\item $\vol_{\Mc}(M(\Oo_F)) = \vol_{\Mc}(M(\Oo_F)^\natural)$,
\item For an open subscheme $W \subset M$, a generating section $\omega \in (\Omega_{\Mc}^{\top})^{\otimes r}(Q^{-1}(W))$ (that is, $\Oo_W \cdot{} \; \omega = (\Omega_{\Mc}^{\top})^{\otimes r}$)  induces an analytic form on $M(\Oo_F)^\natural$ by restriction, which we still call $\omega$. For every Borel measurable subset $A \subset W(\Oo_F)^{\natural}$ we have
$$\vol_{\Mc}(A) = \int_{A} |\omega|^{\frac{1}{r}}.$$
\end{enumerate}
\end{lemma}

\begin{proof} The proof is an elementary exercise using the following two observations:
\begin{enumerate}[(1)]
\item the Borel measure $A \mapsto \int_{A} |\omega|^{\frac{1}{r}}$ is independent of the chosen generator $\omega$ and of $r$, 
\item there exists a positive integer $r$ and a finite Zariski-open covering $\Mc = \bigcup_{i \in \Nb} \Mc_i$, such that there exists a generating section $\omega_i \in (\Omega^{\top}_{\Mc})^{\otimes r}(\Mc_i)$ for all $i \in I$.
\end{enumerate}
The first statement can be shown as follows: Assume that for $j= 1,2$ we have a generating section $\omega_j$ of $(\Omega^{\top}_{\Mc})^{\otimes r_j}(U_i)$  (also known as $r_j$-gauge forms).  Let $r$ be a common multiple of $r_1$ and $r_2$. We write $r = d_j r_j$. This shows that there exists an invertible function $f \in \Oo_{\Mc}(U_i)$, such that $\omega_2^{\otimes d_2} = f \omega_1^{\otimes d_1}$. Invertibility of $f$ implies $f(x) \in \Oo_F^{\times}$ for $x \in U_i(\Oo_F)$. Thus, we have $|f(x)| = 1$ and therefore $|\omega_1(x)|^{\frac{1}{r_1}} = |\omega_2(x)|^{\frac{1}{r_2}}$ for all $x \in U(\Oo_F)$. 

By assumption we are in Situation \ref{situation:locally_quotient} and therefore there exists a finite Zariski-open covering of $\Mc$ by quotient stacks $\mathcal{M}_i = [U_i/\Gamma_i]$, where $\Gamma_i$ is a finite \'etale group $\Oo_F$-scheme. We denote the open subset $Q(\mathcal{M}_i)$ by $W_i$ (recall that $Q$ is a universal homeomorphism and hence open). In particular, there is a finite \'etale morphism $\pi_i\colon U_i \to \mathcal{M}_i$. Without loss of generality we may assume the existence of generating top degree forms $\eta_i \in \Omega_{U_i}^{\top}(U_i)$ (that is, $\Oo_{U_i}\cdot{ } \; \eta_i = \Omega_{U_i}^{\top}$). The norm $\Nm_{U_i/\mathcal{M}_i}(\eta_i)$ of $\eta_i$ along $\pi_i$ defines a generating section of $(\Omega_{\mathcal{M}_i}^{\top})^{\otimes (\deg \pi_i)}$. We define $r = \mathrm{scm}(\deg \pi_i|i \in I)$ and $\omega_i = \Nm_{U_i/\mathcal{M}_i}(\eta_i)^{\otimes \frac{r}{\deg \pi_i}}$.
\end{proof}

\subsection{Twisted inertia stacks and the specialisation map}


In this subsection we introduce the \emph{twisted inertia stack} $I_{\widehat{\mu}}\Mc$ and study a \emph{specialisation map} taking values in the set of isomorphism classes of $k$-rational points of this stack. In this context, \emph{twisting} refers to \emph{Tate twists}.

Let $k$ be a perfect field of characteristic $p$. We denote by $\Nb'$ the set of positive integers coprime to $p$, and define $\widehat{\mu}$ to be the profinite \'etale group scheme given by the inverse limit $\varprojlim_r \mu_r$
indexed by positive integers $r \in \Nb'$ ordered by divisibility. The connecting morphisms $\mu_{dr} \to \mu_r$ are given by the $d$-th power map. This inverse limit defines an affine group scheme of infinite type over $k$.

\begin{definition}\label{defi:twisted_inertia}
For a stack $\Xc$ over $\Spec(k)$, the \emph{twisted inertia stack} $I_{\widehat{\mu}}\Xc$ is defined to be the $\Xc$-stack 
$$I_{\widehat{\mu}}\Xc = \colim_n\Hom(B\mu_n,\Xc) \to \Xc.$$
\end{definition}

For a finite field, rational points of twisted inertia stacks have the following explicit description. 

\begin{lemma}\label{lemma:twisted_inertia} Let $\Xc/\Fb_q$ be a stack whose diagonal morphism is of finite presentation. For a geometric point $x \in \Xc(\Fb_q)$ we denote by $\varphi\colon \Aut_{\Xc(\bar{\Fb}_q)}(x) \to \Aut_{\Xc(\bar{\Fb}_q)}(x)$ the automorphism induced by the Frobenius element $\varphi \in \Gal(\bar{\Fb}_q / \Fb_q)$.
\begin{enumerate}[(a)]
\item We have an equivalence of groupoids 
$$ I_{\widehat{\mu}}\Xc(\Fb_q) \simeq \{(x,\alpha) \mid x \in \Xc(\Fb_q)\text{, } \alpha \in \Hom_{\rm cts}(\widehat{\mu}(\bar{\Fb}_q),\Aut_{\Xc(\bar{\Fb}_q)}(x)) \text{ and }\alpha \circ \varphi = \varphi \circ \alpha\}$$
with the obvious groupoid structure on the right hand side.
\item Let $\xi$ be a profinite generator of $\widehat{\mu}(\bar \Fb_q)$. Under the equivalence of (a), sending $(x,\alpha)$ to $(x,\alpha(\xi))$ gives an equivalence of groupoids
$$[\xi]\colon I_{\widehat{\mu}}\Xc(\Fb_q) \simeq \{(x,\alpha)\in I\Xc(\bar{\Fb}_q)|x \in \Xc(\Fb_q)\text{ and } {}\varphi^*\alpha = \alpha^q\}$$
with the obvious groupoid structure on the right hand side.
\item In particular for $(x,\alpha) \in  I_{\widehat{\mu}}\Xc(\Fb_q)$ as in (b) we have
\begin{equation}\label{twaut}
  \Aut_{ I_{\widehat{\mu}}\Xc(\Fb_q)}(x,\alpha) = \{ \beta \in \Aut_{\Xc(\Fb_q)}(x) \ |\ \alpha \circ \beta = \beta \circ \alpha\}.
 \end{equation}
\end{enumerate}
\end{lemma}

\begin{proof}
Assertion (b) is a reformulation of (a), since the Galois action on $\widehat{\mu}$ is given by $\xi \mapsto \xi^q$. Therefore it suffices to prove (a). By virtue of Definition \ref{defi:twisted_inertia}, we have $I_{\widehat{\mu}}\Xc(\Fb_q) = \varinjlim_r \Hom(\mu_r,I\Xc)(\Fb_q)$ (morphisms of group $\Xc$-schemes). For every positive integer $r$ we have an equivalence of groupoids 
$$\Hom(\mu_r,I\Xc)(\Fb_q) = [\alpha \in \Hom(\mu_r(\bar{\Fb}),I\Xc(\bar{\Fb}_q))| \varphi \circ \alpha = \alpha \circ \varphi]$$
by virtue of Galois descent (and using that the diagonal of $\Xc$ is of finite presentation).
\end{proof}


\begin{construction} \label{LTors}
Let $L$ a finite totally ramified extension of $F$ of degree $N$. Choose a uniformiser $\underline{\pi}$ of $L$. Then for any $\CO_F$-algebra $R$ and any section $\zeta \in \mu_N(R)$ there is a unique automorphism of $\CO_L \otimes_F R=R[\underline{\pi}]$ fixing $R$ which sends $\underline{\pi}$ to $\zeta \underline{\pi}$. This defines an action of the group scheme $\mu_N$ on $\Spec(\CO_L)$ relative to $\Spec(\CO_F)$. Furthermore, since any other choice of uniformiser $\underline{\pi}$ differs from the given one by a unit in $\CO_F$ this action is independent of this choice. This action endows $\Spec(L)$ with the structure of a $\mu_N$-torsor over $\Spec(F)$.
\end{construction}

For the notion of the normalisation of a stack appearing in the following we refer the reader to \cite[Appendix A]{ascher2017moduli}.

\begin{proposition} \label{CMIso}
  Let $\CM$ be a tame DM-stack over $\Oo_F$ satisfying the following conditions:
  \begin{enumerate}[(i)]
 \item The diagonal morphism $\CM \to \CM \times_{\Oo_F} \CM$ is finite.
  \item The morphism $\CM \to \Spec(\CO_F)$ is the coarse moduli space of $\CM$.
  \item The generic fibre $\CM_F$ of $\CM$ is equal to $\Spec(F)$.
  \end{enumerate}
 Let $\widetilde \CM$ be the normalisation of $\CM$. For some totally ramified finite field extension $F \subset L$ of degree $N$ there exists an isomorphism $[\Spec(\CO_L)/\mu_N] \cong \widetilde\CM$. For a given $L$ this isomorphism is unique up to isomorphism. Moreover, the induced morphism $[\Spec(k_F)/\mu_N] = [\Spec(k_L)/\mu_N] \cong \widetilde \CM_{k_F}$ in the special fibre is independent of $L$ up to isomorphism.
\end{proposition}
\begin{proof}
First we observe that property (i) implies that the inertia stack $I \CM$ is finite over $\CM$, since $I\CM$ is by definition equal to the base change of the diagonal along itself.

We claim that the normalisation $\widetilde \CM$ again satisfies the conditions of the proposition. For condition (iii) this is immediate. For condition (ii) let $\widetilde M$ be the coarse moduli space of $\widetilde \CM$. By the Keel--Mori Theorem, the stack $\CM$ is proper and quasi-finite over its coarse moduli space $\Spec(\CO_F)$ (see \cite{MR1432041} and \cite[Theorem 1.1]{conradKM}), and analogously for $\widetilde \CM$. We observe that $\widetilde{\CM}\to \CM$ is finite (tame stacks are locally of finite type, and therefore this follows from \cite[Tag 0BXR]{stacks-project} and descent theory). This implies that $\widetilde \CM$, and hence $\widetilde M$, is quasi-finite over $\Spec(\CO_F)$. By \cite[Theorem 4.16]{MR3237451} the fact that $\widetilde \CM$ is normal implies the same for $\widetilde M$. Using the fact that $\widetilde \CM \to \CM$ is birational one concludes that $\widetilde M \to \Spec(\CO_F)$ is an isomorphism. 

To see that condition (i) holds we argue as follows. As noted above, the morphism $\widetilde{\CM} \to \CM$ is finite. Furthermore, the composition of morphisms with finite diagonal has again finite diagonal, by virtue of \cite[Tag 050K]{stacks-project}.

Thus we may replace $\CM$ by $\widetilde \CM$ and can assume that $\CM$ is normal.

Since the diagonal morphism of $\CM$ is finite we may apply \cite[Theorem 2.7]{MR1844577}, which gives the existence of a surjective finite morphism $X \to \CM$ from a scheme $X$. Such an $X$ is proper and quasi-finite and hence finite over $\Spec(\CO_F)$. Hence after taking some connected component of $X$ we may assume that $X=\Spec(\CO_{L'})$ for some finite field extension $F \subset L'$. After enlarging $L'$ we assume that $L'$ is Galois over $F$.

Consider a uniformiser $\underline{\pi}$ of $L'$ and the subfield $L\defeq F[\underline{\pi}]$ of $L'$. The field extension $L \into L'$ is Galois and unramified. By the valuative criterion for properness of morphisms of stacks (c.f. \cite[0CLY]{stacks-project}) every morphism $\Spec(\CO_{L'}) \to \CM$ is uniquely determined by its restriction to the fibre over $\Spec(F)$. Using this and condition $(iii)$ it follows that the finite surjection $\Spec(\CO_{L'}) \to \CM$ factors through a finite surjection $\Spec(\CO_L)=[\Spec(\CO_{L'})/\Gal(L'/L)] \to \CM$. Since $\CM$ and $\Spec \CO_L$ are normal and $1$-dimensional, they are regular. The ``miracle criterion for flatness'' (\cite[Theorem 23.1]{crt}) implies that this finite surjection $\Spec(\CO_L) \to \CM$ is flat. By an analogous argument to the above, this surjection in turn factors through a finite flat surjection $[\Spec(\CO_L)/\mu_N] \to \CM$. Since this morphism is an isomorphism generically it is an isomorphism everywhere.

To see the uniqueness claim for a given $L$, one uses the fact that again by the valuative criterion for properness the stack $\CM$ has no non-trivial automorphisms over $\Spec(\CO_F)$. 

Finally, if for two possibly different totally ramified field extensions $L_1$ and $L_2$ of $F$ we have isomorphisms $[\Spec(\CO_{L_i})/\mu_N]\cong \CM$ we need to show that the resulting morphisms $[\Spec(k_F)/\mu_N] \to \CM_{k_F}$ are isomorphic. For this, chose a sufficiently large Galois extension $L'$ of $F$ together with embeddings $L_i \into L'$. We obtain morphisms $\Spec(\CO_{L'}) \to \Spec(\CO_{L_i}) \to \CM$. By invoking the valuative criterion another time, we see that these morphism differ by an element of $\Gal(L'/F)$. Since these elements induce the identity on $k_F$, we get what we want.
\end{proof}

\begin{construction}\label{defi:e} 
Let $\CM$ be a tame stack in the sense of \cite{MR2427954} over $\Spec(\CO_F)$ with coarse moduli space $M$. We write $[I_{\widehat{\mu}}\Mc(k)]$ for the set of isomorphism classes of the groupoid $I_{\widehat{\mu}}\Mc(k)$ and define a map
\[e: M(\Oo_F)^\natural \to [I_{\widehat{\mu}}\Mc(k)],  \]
as follows: 

For $x \in M(\CO_F)^\natural$ consider the pullback $\CM_x$ along $x$:
\begin{equation*}
  \xymatrix{
    \CM_x \ar[r] \ar[d] & \CM \ar[d] \\
    \Spec(\CO_F) \ar[r] & M
}
\end{equation*}
By the definition of $\M(\CO_F)^\natural$, the stack $\CM_x$ satisfies condition $(iii)$ of Theorem \ref{CMIso}. Since by \cite[3.3]{MR2427954} the formation of coarse moduli spaces of tame stacks commutes with base change, it also satisfies condition $(ii)$. Hence Proposition \ref{CMIso} gives us a morphism $[\Spec(k)/\mu_N] \to \widetilde \CM_{k} \to\CM_{k}$. This induces a point $e(x) \in I_{\widehat \mu}\CM(k)$ which is well-defined up to isomorphism.
\end{construction}

\subsection{Torsors over local fields}\label{sub:torsors}

Let $\Gamma$ be a finite \'etale group $\Oo_F$-scheme, of order coprime to $p$ and $\underline{\pi} \in F$ a fixed uniformizer.
For details about non-abelian Galois cohomology we refer to \cite[Section I.5]{Se07}. We denote by $H^1(F,\Gamma)$ the pointed set of isomorphism classes of (left) $\Gamma$-torsors. If $\Gamma$ is a constant group scheme over $\Oo_F$, there is a bijection of pointed sets
\[ H^1(F,\Gamma) \cong \Hom( \Gal^{tr}(F),\Gamma)/\Gamma,  \]
where $ \Gal^{tr}(F)$ denotes the tamely ramified Galois group of $F$ and $\Gamma$ acts on this set through conjugation.
 
The group $ \Gal^{tr}(F)$ has an explicit description as the profinite group on two generators $\beta, \gamma$ subject to the relation $\beta\gamma \beta^{-1} = \gamma^q$ (see \cite{Hasse} and \cite{Iwasawa}). Here $\beta$ is a lift of the Frobenius and $\gamma$ a generator of the tamely totally ramified Galois group of $F$. Consequently we have the presentation
\begin{equation*}H^1(F,\Gamma) \cong\{ (x_\beta,x_\gamma) \in \Gamma^2 \mid  \beta\gamma \beta^{-1} = \gamma^q\}/\Gamma, \text{ if $\Gamma/\Oo_F$ is constant.} \end{equation*}

In the non-constant case we have a similar description. Let $\phi \in \Gal(\bar{\Fb}_q/\Fb_q) = \pi_1^{\text{\'et}}(\Spec \Oo_F)$ denote the Frobenius automorphism. We we abbreviate $\Gamma^{un} = \Gamma(F^{un})$ and  write $\phi\colon \Gamma^{un} \to \Gamma^{un}$ for the induced bijection.
Then
\begin{equation}\label{ntor} H^1(F,\Gamma) = \{(x_{\beta},x_{\gamma}) \in( \Gamma^{un})^2 \ | \ x_{\beta}\phi(x_{\gamma})x_{\beta}^{-1} = x_{\gamma}^q \} / \sim{},\end{equation}
where the equivalence relation identifies $(x_{\beta},x_{\gamma}) \sim{} (y^{-1}x_{\beta}\phi(y),y^{-1}x_{\gamma}y)$ for all $y \in \Gamma(F^{un})$. We denote the class of $(x_\beta,x_\gamma)$ in $H^1(F,\Gamma)$ by $[(x_\beta,x_\gamma)]$.

Let us spell out how to associate to a cocycle $(x_\beta,x_\gamma)$ as in \eqref{ntor} a $\Gamma$-torsor: First we define a new action of $\Gal^{tr}(F)$ on $\Gamma^{un}$ by setting for all $y \in \Gamma^{un}$
\begin{equation}\label{newact} \beta \cdot y = x_\beta {}^\phi y \text{ and } \gamma \cdot y = x_\gamma y. \end{equation}
This defines an action of $\Gal^{tr}(F)$ on the algebra $\prod_{y \in \Gamma^{un}} F^{tr}$ and the invariant part gives the coordinate ring for a $\Gamma$-torsor over $F$. In fact if $\{y_i\}_{i\in I} \subset \Gamma^{un}$ is a set of representatives of  $\Gal^{tr}(F)$-orbits in $\Gamma^{ur}$ with respect to the action \eqref{newact}, then we have an isomorphism
\begin{equation}\label{invexp}   \left(  \prod_{y \in \Gamma^{un}} F^{tr}  \right)^{\Gal^{tr}(F)}   \cong \prod_{i\in I} \left( F^{tr} \right)^{G_i},    \end{equation}
where $G_i \subset \Gal^{tr}(F)$ is the stabilizer of $y_i$ with respect to the new action. In particular we see that any $\Gamma$-torsor is of the form $\Spec(\left( \prod_i L_i   \right)$ where $L_i/F$ is a finite separable extension.

We will often use the following fact.

\begin{lemma}\label{lemma:madealemmaoutofit} Let $Q = \Spec\left( \prod_{i\in I} L_i   \right)$ be a $\Gamma$-torsor. 
\begin{enumerate}
\item[(i)] We have $[Q] = [x_\beta,1]$ for some $x_\beta \in \Gamma^\text{un}$ if and only if $L_i/F$ is unramified for all $i\in I$. In this case we call $Q$ unramified.
\item[(ii)] We have $[Q] = [1,x_\gamma]$ for some $x_\gamma \in \Gamma^\text{un}$ if and only if $L_i/F$ is totally ramified for at least one $i\in I$. If $L_i,L_j$ are both totally ramified, then $L_i \cong L_j$ and $\Gamma(F)$ acts simply transitive on the connected components of $Q$ corresponding to totally ramified extensions. In this case we call $Q$ strongly ramified.
\item[(iii)] There exists an unramified right $\Gamma$-torsor $P$, such that the twist $Q_P = P \times^\Gamma Q$ is a strongly ramified $\Gamma_P$-torsor, where $\Gamma_P$ denotes the inner form of $\Gamma$ defined by $P$. 
\end{enumerate}
\end{lemma}

\begin{proof} Let $G_i \subset \Gal^{tr}(F)$ denote the stabilizer of $y_i$ as in \eqref{invexp}. Then for (i) we notice that $L_i=\left( F^{tr} \right)^{G_i}$ is unramified if and only if $\gamma \in G_i$. This is true for all (or equivalently, for one) $i$ if and only if $x_\gamma = 1$.

For (ii) suppose first $[Q] = [1,x_\gamma] $. Then we see that the identity $1 \in \Gamma(F)$ is fixed by $\beta$ and hence the extension corresponding to the orbit of $1$ is totally ramified. Conversly if $L_i / F$ is totally ramified we know that $\beta \cdot y_i = y_i$, this implies $x_\beta = y {}^\phi y^{-1}$ for some $y\in \Gamma^{un}$. Using the equivalence relation between cochains we see that  $[Q] = [1,x_\gamma] $ for a suitable $x_\gamma$. Now let $[Q] = [1,x_\gamma]$ and assume $L_i,L_j$ are both totally ramified. Then $y_i,y_j$ are both fixed by $\beta$ which implies that $y_i,y_j \in \Gamma(F)$. Since $\gamma$ acts by multiplication by $x_\gamma$ on the left, we get $G_i = G_j$. This implies the second part of (ii).

 For (iii) let $[Q] =[x_\beta,x_\gamma]$ and $P$ be the unramified right-torsor corresponding to $(x_\beta^{-1},1)$. A direct cocycle computation shows that $Q_P$ is represented by the cocycle $(1,x_\gamma)$, hence $Q_P$ is strongly ramified.

\end{proof}

\begin{definition}
  Let $Q$ be a strongly ramified $\Gamma$-torsor and $\Spec(L) \subset Q$ a totally ramified connected component. We denote by $I_Q \subset \Gamma$ the stabilizer of $\Spec(L)$. 
\end{definition}
By Lemma \ref{lemma:madealemmaoutofit} the subgroup scheme $I_Q$ of $\Gamma$ is well-defined up to conjugation by $\Gamma(F)$ and $\Spec(L)$ is a $I_Q$-torsor.

\begin{lemma}\label{canincc} For a strongly ramified $\Gamma$-torsor $Q$ we have a canonical isomorphism $\mu_N \cong I_Q$, with $N= [L/F]$.
\end{lemma}
\begin{proof} 
By Construction \ref{LTors} we have a canonical structure of $L$ as a $\mu_N$-torsor over $F$. This torsor structure is compatible with any automorphism of $\Spec(L)$ over $\Spec(F)$. Hence we obtain a canonical  morphism $I_Q \to \mu_N$ which is necessarily an isomorphism.
\end{proof}

\begin{construction}\label{defi:XGamma}
 Let $Q$ be a $\Gamma$-torsor over $F$. We denote by $Q_{\Oo}$ the normalisation of $\Spec \Oo_F$ in $Q$. It follows from the valuative criterion for properness that the action of $\Gamma$ on $Q$ extends uniquely to an action of $\Gamma$ on $Q_{\Oo}$. Hence we get a smooth tame Deligne-Mumford $\Oo_F$-stack
$$\Xc_{\Gamma,Q} = [Q_{\Oo}/\Gamma].$$
\end{construction}

The stack $\Xc_{\Gamma,Q}$ contains a dense open subspace which is isomorphic to $\Spec F = [Q/\Gamma] \subset [Q_{\Oo}/\Gamma]$. We denote the open immersion by
$$j\colon \Spec F \hookrightarrow \Xc_{\Gamma,Q}.$$

Since normalisation commutes with smooth base change, for any unramified right $\Gamma$-torsor $P$ there is a canonical equivalence
\[   \Xc_{\Gamma,Q} \cong \Xc_{\Gamma_P,Q_P}.     \]

In particular by Lemma \ref{lemma:madealemmaoutofit} we can take $P$, such that $Q_P$ is stongly ramified and hence we can identify the special fibre of $\Xc_{\Gamma_P,Q_P}$ with $B_{k}I_{Q_P}$. Together with Lemma \ref{canincc} we finally obtain a closed immersion 
\begin{equation}\label{caninc} B_{k}\mu_N \hookrightarrow   \Xc_{\Gamma,Q}.  \end{equation}
This closed immersion is not unique, however see Remark \ref{rmk:noncan}.

\subsection{The specialisation map and the volume of its fibres}\label{sub:e}

\begin{construction}\label{constrspe}
Consider a tame finite \'etale quotient stack $\CM=[U/\Gamma]$, where $U$ is a smooth $\Oo_F$-scheme. In this case we can render the specialisation map $e\colon M(\CO_F)^\natural$ of Construction \ref{defi:e} more explicit.

For $x \in M(\Oo_F)^\natural$ the pullback $\Spec(F) \times_{M} U$ is a $\Gamma$-torsor $Q$ over $F$. 
We obtain the following commutative diagram:
\[
\xymatrix{
Q_{\Oo} \ar[d] \ar@/^10pt/@{-->}[rr] & \ar@{_(->}[l] Q \ar[r] \ar[d] & U \ar[d] \\
\Spec \Oo_F  \ar@/_10pt/[rr] &  \ar@{_(->}[l]  \Spec F \ar[r] & U/\Gamma
}
\]
As before $Q_{\Oo}$ is the normalisation of $\Spec \Oo_F$ in $Q$. If we write $Q_{\Oo} = \Spec(\prod_i \Oo_{L_i})$ the dashed map is obtained by applying the valuative criterion of properness to the finite morphism $U \to U/\Gamma$ and the discrete valuation rings $\Oo_{L_i}$. The dashed morphism is in particular $\Gamma$-equivariant and taking the stack-quotient by $\Gamma$ gives a morphism $\Xc_{\Gamma,Q} \to \Mc$. Through \eqref{caninc} we obtain a morphism $B_k\mu_N \to \Mc$. By comparing this construction with Construction \ref{defi:e} one sees that the corresponding point in $[I_{\hat{\mu}}\Mc(k)]$ is equal to $e(x)$. 
\end{construction}

\begin{construction}\label{totramlift}

Let us describe the fibres of $e$ more explicitly in the case when $\Mc=[U/\Gamma]$ is a finite \'etale quotient stack, essentially following \cite[Section 2]{DL2002}. We fix a profinite generator $\xi \in \widehat{\mu}(\bar k)$ and let $(x,\alpha) \in I_{\widehat{\mu}}\Mc(k)$ be as in Lemma \ref{lemma:twisted_inertia} (b), i.e. $x \in \Mc(k)$ and $\alpha \in \Aut_x(\bar{k})$. After possibly twisting $U$ and $\Gamma$ by a suitable unramified torsor (Lemma \ref{lemma:madealemmaoutofit}) we may further assume that $x$ lies in the image of the quotient map $U(k) \to \Mc(k)$. Then up to conjugation in $\Gamma$ we can identify $\alpha$ with an element in $\Gamma(\bar{k})$ and such an identification gives an inclusion $\mu_N \hookrightarrow \Gamma$ where $N = \ord(\alpha)$.

Now let $\underline{\pi} \in F$ be a uniformiser and $L = F(\underline{\pi}^{\frac{1}{N}})$ be the tamely totally ramified extension of $F$ obtained by adjoining an $N$-th root of $\underline{\pi}$. By Construction \ref{LTors} the scheme $\Spec(L)$ is naturally a $\mu_N$-torsor over $\Spec(F)$. The pushforward of $\Spec(L)$ to a $\Gamma$-torsor can be represented by the cocycle $(1,\alpha)$. We denote by $U(\Oo_L)^{\mu_N}_x$ the set of $\mu_N$-equivariant morphisms $\Spec(\Oo_L) \to U$ which specialize to $x$ on the special fibre. Then using Constrution \ref{constrspe} it follows that the natural quotient morphism $U(\Oo_L)^{\mu_N}_x \to M(\Oo_F)^\natural$ defines a surjection
\begin{equation}\label{liftingsq} U(\Oo_L)^{\mu_N}_x \to e^{-1}(x,\alpha).\end{equation}
\end{construction}

\begin{definition}\label{defi:weight}
Let $k$ be a field.
\begin{enumerate}[(a)]
\item Let $(\chi_1,\dots,\chi_r) \in (\Qb/\Zb)^r$. We define $w(\chi_1,\dots,\chi_r) = \sum_{i=1}^r c_i$, where $c_i\in \Qb$ denotes the unique rational representative of $\chi_i \in \Qb/\Zb$, such that $0 < c_i \leq 1$.
\item To a $k$-vector space and an algebraic action $a\colon \widehat{\mu} \to \mu_N \to \Aut(V)$ we can associate a character decomposition $V_{\bar{k}} \cong \bigoplus_{i=1}^rV_{\bar{k}}(\chi_r)$, where $\chi_i \in \widehat{\mu}^{\vee} \cong \Qb/\Zb$ (as discrete groups). We denote by $w(a) = w(\chi_1,\dots,\chi_r)$ and refer to it as the weight of $(V,a)$.
\item Let $(x,a) \in I_{\widehat{\mu}}\Mc(k)$ where $\Mc$ is a smooth DM-stack over $k$. We denote by $w(x,a)$ the weight of the induced action $\widehat{\mu}$ on $T_x\Mc$.
\end{enumerate}
\end{definition}
Usually the weight is defined for points in the (untwisted) inertia stack $I\Mc$ by fixing a primitive root of unity, see for example \cite[Definition 2.2]{gwz}. It is not hard to check, that Definition \ref{defi:weight} agrees with the usual one through Lemma \ref{lemma:twisted_inertia}.

\begin{theorem}\label{thm:volume}
Let $\Mc/\Oo_F$ be a DM-stack as in Situation \ref{situation:locally_quotient}(b). For $(x,\alpha) \in I_{\widehat{\mu}}\Mc(k)$ we have
$$\vol_{\Mc}\left(e^{-1}(x,\alpha)\right) = \frac{q^{-w(x,\alpha)}}{|\Aut_{I_{\widehat \mu}\Mc(k)}(x,\alpha)|}.$$
\end{theorem}

\begin{proof} Without loss of generality we can assume $\Mc =[U/\Gamma]$. By twisting $U$ and $\Gamma$ with an element of $H^1(\Oo_F,\Gamma)$ if necessary, we may further assume, that $x\in\Mc(k)$ lies in the image of the quotient map $\rho:U(k) \to [U/\Gamma](k)$, see Lemma \ref{lemma:madealemmaoutofit}. Then \eqref{liftingsq} implies, that every $\psi \in e^{-1}(x,\alpha)$ fits into a commutative square
\[ \xymatrix{ \Spec(\Oo_L) \ar[d] \ar[r]^{\psi_L} & U \ar[d] \\
\Spec(\Oo_F) \ar[r]^{\psi} & U/\Gamma,
}
\]
where $L$ is as in Construction \ref{totramlift} and $\psi_L$ is $\mu_N$-equivariant with respect to the fixed inclusion $\mu_N \hookrightarrow \Gamma$, which sends the primitive root $\xi$ to $\alpha$. From this we see that the quotient $s:[U/\mu_N] \to [U/\Gamma]$ induces a surjection $s^{-1}(e^{-1}(x,\alpha)) \to e^{-1}(x,\alpha)$. Denote by $U_L$ the twist  and $\Gamma_L$ the inner form defined by the $\mu_N$-torsor $L$. Then the diagram
\[ \xymatrix{ U_L \ar[rd] \ar[d] &  \\
U_L/\mu_N = U/\mu_N \ar[r]^s & U_L / \Gamma_L = U/\Gamma ,
}
\]
commutes and every arrow is a generically \'etale morphism of smooth $F$-varieties. By applying Proposition \ref{intools} (\ref{chovg}) to $U_L \to U_L/\mu_N$ and to $U_L \to U_L/\Gamma_L$, we see
\[ \vol_{\Mc}(e^{-1}(x,\alpha)) = \frac{|\mu_N(F)|}{|\Gamma_L(F)|}\vol_{[U/\mu_N]}(s^{-1}(e^{-1}(x,\alpha))). \]

Now essentially by definition we have a commutative diagram

\[ \xymatrix{  U/\mu_N(\Oo_F)^\natural \ar[r]^{e_{\mu_N}} \ar[d]^s  &  [I_{\widehat{\mu}}[U/\mu_N](k)]       \ar[d]^{ I_{\widehat{\mu}}s}  \\
U/\Gamma(\Oo_F)^\natural \ar[r]^e & [I_{\widehat{\mu}}[U/\Gamma](k)],           }\]
from which we deduce the decomposition 
\[s^{-1}(e^{-1}(x,\alpha)) = \bigsqcup_{(y,\xi) \in I_{\widehat{\mu}}s^{-1}(x,\alpha)} e_{\mu_N}^{-1}(y,\xi).\]

The volume of $e_{\mu_N}^{-1}(y,\xi)$ can be computed as follows. First we use \cite[Lemma 4.19]{gwz} (the proof for non-constant $\mu_N$ is the same) to reduce to the case where $U=\BA^n$, $y=0$ is the origin and the action of $\mu_N$ is linear. In this case we have 
\[ \vol_{[U/\mu_N]}(e_{\mu_N}^{-1}(y,\alpha)) = \frac{q^{-w(\alpha)}}{|\mu_N(F)|},\]
by Lemma \ref{diagcase} below. In particular it is independent $(y,\xi)$, so it remains to determine the cardinality of $I_{\widehat{\mu}}s^{-1}(x,\alpha)$. 

Since $x$ is in the image of $\rho$, we can associate to $(x,\alpha)$ an element $x' \in U(k)$ invariant under $\alpha$, which is well-defined up to isomorphisms in $ I_{\widehat{\mu}}[U/\Gamma](k)$ i.e. up to the action of the centralizer $C(\alpha)(k)$. The image $y$ of $x'$ in $[U/\mu_N](k)$ gives rise to a point $(y,\xi)  \in I_{\widehat{\mu}}s^{-1}(x,\alpha)$ and conversely every $(y,\xi)  \in I_{\widehat{\mu}}s^{-1}(x,\alpha)$ arises in this way. In particular we have 
\[  |I_{\widehat{\mu}}s^{-1}(x,\alpha)| = |C(\alpha)(k)\cdot x'|,  \]
where $x'$ is any lift of $(x,\alpha)$. Putting everything together we get
 \[ \vol_{\Mc}(e^{-1}(x,\alpha)) = \frac{ |C(\alpha)(k)\cdot x'|q^{-w(x,\alpha)}}{|\Gamma_L(F)|}. \]

Finally we compute $ |\Aut_{I_{\widehat \mu}[U/\Gamma](k)}(x,\alpha)|$. Let $x' \in U(k)$ denote again a lift of $x$. By  \eqref{twaut} and the orbit formula for finite groups we have

\[ |\Aut_{I_{\widehat \mu}[U/\Gamma](k)}(x,\alpha)| = |\{ \beta \in \Aut_{[U/\Gamma](k)}(x) \ |\ \alpha \circ \beta = \beta \circ \alpha\}| =\frac{|C(\alpha)(k)|}{|C(\alpha)(k)\cdot x'|}. \]
By defintion $\Gamma_L$ is the inner form of $\Gamma$ on which the Frobenius acts as on $\Gamma$ and the generator of the totally ramified part by conjugation with $\alpha$. From this we see $|\Gamma_L(F)| = |C(\alpha)(k)|$, which finishes the proof.
\end{proof}

\begin{lemma}\label{diagcase} Let $\CM = [\BA^n/\mu_N]$, where $\mu_N$ acts linearly and non-trivially on $\BA^n$ over $\CO_F$. Let $\xi \in \mu_N(\overline{F})= \mu_N(\overline{\BF}_q)$ be a primitive $N$-th root and $0 \in \BA^n$ the origin. Then $(0,\xi)$ defines an element in $I_{\hat{\mu}}\CM(\BF_q)$ and we have
\[  \vol_{\Mc}(e^{-1}(0,\xi)) = \frac{q^{-w(0,\xi)}}{|\mu_N(F)|}. \] 
\end{lemma}

\begin{proof} Let us assume first that the action of $\mu_N$ is diagonal, in which case the proof is essentially taken from \cite[Section 2]{DL2002}. Let $A_\xi = \diag(\xi^{c_1},\dots,\xi^{c_n})$, where $1 \leq c_1,\dots,c_n \leq N$, be the matrix through which $\xi$ acts on $\BA^n$. Notice that in this notation we have $w(\xi) = \frac{1}{N}\sum_i c_i$. Then we define a morphism
\begin{equation}\label{lambdasur} \lambda: \BA^n \to \BA^n/\mu_N\end{equation}
on the level of coordinate rings as $h(x_1,\dots,x_n) \mapsto h(\underline{\pi}^\frac{c_1}{N}x_1,\dots,\underline{\pi}^\frac{c_n}{N}x_n)$ for $h\in \CO[x_1,\dots,x_n]^{\mu_N}$. Notice that the $\mu_N$-invariance will imply that $\lambda$ is in fact defined over $\CO_F$. One can check that $\lambda$ is generically $|\mu_N(F)|$-to-$1$ and (see \cite[(2.3.4)]{DL2002})
\begin{equation*} \lambda(\CO_F^n) \cap (\BA^n/\mu_N)(\CO_F)^{\natural} = e^{-1}(0,\xi).\end{equation*}
The $N$-th power of the canonical form $dx_1\wedge \dots \wedge dx_n$ certainly descends to $\CM$ and thus defines an orbifold form $\omega_{orb} \in H^0(\CM,(\Omega^{top}_\CM)^{\otimes N})$. One sees that
\[ \lambda^* \omega_{orb} = \underline{\pi}^{\sum_i c_i} (dx_1\wedge \dots dx_n)^{\otimes N},\]
which, together with Lemma \ref{intools}, gives
\[\vol_{\Mc}(e^{-1}(0,\xi)) = \int_{e^{-1}(0,\xi)} |\omega_{orb}|^{\frac{1}{N}} = \frac{1}{|\mu_N(F)|}\int_{\CO_F^n} |\underline{\pi}^{\sum_i c_i}|^{\frac{1}{N}}dx_1\cdots dx_n = \frac{q^{-w(\xi)}}{|\mu_N(F)|}.\]

If the action of $\mu_N$ is not diagonal we may can still diagonalize it over some finite unramified extension $L/F$. In concrete terms let $A_\xi$ be the matrix through which $\xi$ acts on $\BA^n$ and $B \in \GL_n(\CO_L)$ such that $\widetilde{A}_\xi = B A_\xi B^{-1} $ is diagonal. Similar as before we define $\lambda: \BA^n \to \BA^n/\mu_N$ on the level of coordinate rings as
\[  h(x)  \mapsto h(B^{-1} \tilde{\lambda} B x), \]
for $h \in F[x_1,\dots,x_n]^{\mu_N}$. Here $\tilde{\lambda}$ is the diagonal matrix $\diag(\underline{\pi}^\frac{c_1}{N},\dots,\underline{\pi}^\frac{c_n}{N})$ and $c_1,\dots,c_n$ are defined using the eigenvalues of $\widetilde{A}_\xi$. Again $\mu_N$-invariance implies that $\lambda$ is defined over $\CO_L$ and in fact it is even defined over $\CO_	F$. 

To see this we need to show ${}^\phi (B^{-1} \tilde{\lambda} B) = B^{-1} \tilde{\lambda} B$, where $\phi \in \Gal(L/F)$ denotes the Frobenius morphism. Since $\mu_N$ acts algebraically on $\BA^n$ we have $ {}^\phi A_\xi = A^q_\xi$ and hence 
\[ ^\phi\left(B^{-1}\widetilde{A}_\xi B \right) = {}^\phi A_\xi = A^q_\xi = B^{-1} \widetilde{A}_\xi^q B.\]
On the other hand since $\widetilde{A}_\xi$ is a diagonal matrix whose entries are roots of unity we find
\[{}^\phi\left(B^{-1}\widetilde{A}_\xi B \right) = {}^\phi B^{-1} {}^\phi {\widetilde{A}_\xi} {}^\phi B = {}^\phi B^{-1} \widetilde{A}_\xi^q {}^\phi B.   \]

It follows that $\widetilde{A}_\xi^q$ commutes with $^\phi B	B^{-1}$. Since $q$ and $N$ are coprime also $\widetilde{A}_\xi$ commutes with $^\phi B	B^{-1}$. Thus $^\phi B	B^{-1}$ preserves the eigenspaces of $\widetilde{A}_\xi$ and it follows from the definition of $\tilde{\lambda}$ that $^\phi B	B^{-1}$ finally commutes with $\tilde{\lambda}$ which implies that $\lambda$ is an $\CO_F$-morphism. We can now use the same argument as in the diagonal case to finish the proof.
\end{proof}

\section{The Hasse invariant and inertia stacks}\label{sec:hasse}

In this paper we will consider not only volumes but also integrals of functions which are associated to canonical gerbes on the moduli space of Higgs bundles (over non-archimedean local fields $F$). The origin of these functions is a classical construction going back to Hasse. 

The Hasse invariant assigns to a $\G_m$-gerbe on a non-archimedean local field $F$ an element of $\Qb/\Zb$. The purpose of this section is to give a stack-theoretic interpretation of Hasse invariants. From this we will deduce the following assertion which plays an important role in the proof of Geometric Stabilisation (see Section \ref{proofgst}).

\begin{goal}
Let $\Mc/\Oo_F$ be a smooth and tame Deligne-Mumford stack (satisfying \ref{situation:locally_quotient}(b)) and $\alpha \in Br(\Mc)$ be a gerbe on $\Mc$. Then there exists a function $f_{\alpha}\colon I_{\widehat{\mu}}\Mc(k) \to \Qb/\Zb$, such that $\inv(x_F^*\alpha) = f_{\alpha}(e(x))$ for every $x \in M(\Oo_F)^{\natural}$ (see Proposition \ref{prop:hasse2}).
\end{goal}

\subsection{Recollection on gerbes}

Let $\Xc$ be a Deligne-Mumford stack, and $A$ a commutative group scheme. An $A$-gerbe on $\Xc$ is by definition a morphism of stacks $\Gg \to \Xc$, for which that there exists a surjective \'etale morphism of finite presentation $f\colon Y \to \Xc$, such that the base change $\Gg \times_{\Xc} Y \to Y$ is isomorphic to the morphism $Y \times BA \to Y$. The following result is well-known and follows directly from the definition and descent theory.

\begin{lemma}[Giraud]
The set of equivalence classes of $A$-gerbes on $\Xc$ is in natural bijection with the \v Cech cohomology group $\check H^2_{\text{\'et}}(\Xc,A)$. 
\end{lemma}

We will be mostly concerned with $\G_m$-gerbes on $\Xc$, which by the lemma above are classified by elements of $\check H^2_{\text{\'et}}(\Xc,\G_m)$. Henceforth we apply the convention that the term \emph{gerbe} refers to a $\G_m$-gerbe, unless stated otherwise. Furthermore we remark that if $\Xc$ is a scheme which is quasi-projective over an affine scheme, we have $\check H^2_{\text{\'et}}(\Xc,\G_m) = H^2_{\text{\'et}}(\Xc,\G_m)$ (see \cite[Theorem 2.17]{MR559531} for this isomorphism). 

For a non-archimedean local field $F$, the cohomology group $H^2_{\text{\'et}}(F,\G_m)$ is well-understood. This amounts to the computation of the Brauer group of $F$ which plays an important role in this article.

\begin{proposition}[Hasse invariant]\label{prop:Hasse}
There exists an isomorphism $\inv\colon H^2_{\text{\'et}}(F,\G_m) \xrightarrow{\simeq} \Qb/\Zb$, called the \emph{Hasse invariant}.
\end{proposition}

\begin{proof}
For the convenience of the reader we recall the construction of Hasse's invariant (in a modern formulation) which follows the presentation in \cite{MR559531}. We denote by $D$ the affine scheme $\Spec \Oo_F$, and $D^{\circ} = \Spec F$. There is an open immersion $i\colon D^{\circ} \hookrightarrow D$ with closed complement denoted by $j\colon \bullet = \Spec k\hookrightarrow D$. We have a short exact sequence of sheaves on the small \'etale site $(D)_{\text{\'et}}$
$$0 \to \Oo_{D}^{\times} \to i_*\Oo_{D^{\circ}}^{\times} \to j_*\Zb_{\bullet} \to 0.$$
Therefore we have an exact sequence
$$H^2_{\text{\'et}}(D,\Oo_D^{\times}) \to H^2_{\text{\'et}}(F,\G_m)= H^2_{\text{\'et}}(D^{\circ},\Oo_{D^{\circ}}^{\times}) \to H^2_{\text{\'et}}(D,j_*\Zb_{\bullet}) \to H^3_{\text{\'et}}(D,\Oo^{\times}_D).$$ The cohomology group $H^2_{\text{\'et}}(D,j_*\Zb_{\bullet})$ can be identified with $H^2_{\text{\'et}}(\bullet,\Zb)$ (see \cite[Example 2.22]{MR559531} for a similar argument). Using the Bockstein isomorphism $H^2_{\text{\'et}}(\bullet,\Zb) \simeq H^1_{\text{\'et}}(\bullet,\Qb/\Zb)$, and the fact that $\bullet = \Spec k$, we obtain $H^1_{\text{\'et}}(\bullet,\Qb/\Zb) \simeq \Qb/\Zb$. Furthermore we have a splitting of the morphism $F^{\times} \to \Zb$ given by $1 \mapsto \underline{\pi} \in F$ (where $\underline{\pi}$ is a uniformiser). Hence the morphism $H^2_{\text{\'et}}(D,j_*\Zb_{\bullet}) \to H^3_{\text{\'et}}(D,\Oo^{\times}_D)$ is the zero map.

The cohomology group on the left is by definition of $D$ equal to $H^2_{\text{\'et}}(\Oo_F,\G_m) = Br(\Oo_F)$. Since $\Oo_F$ is a henselian local ring with a finite residue field, its Brauer group vanishes. This shows that we have an isomorphism $H^2_{\text{\'et}}(D^{\circ},\Oo_{D^{\circ}}^{\times}) \xrightarrow{\simeq} H^2_{\text{\'et}}(D,j_*\Zb_{\bullet}) \simeq \Qb/\Zb$.
\end{proof}

\subsection{A stacky interpretation}

We begin this subsection by formally computing the Brauer group of $B\widehat{\mu}$.

\begin{lemma}\label{lemma:Brauer_stack}
Let $k$ be a field and $N$ a positive integer which is invertible in $k$. Then we have an equivalence $f\colon H^2_{\text{\'et}}(B_k\mu_N,\G_m) \simeq \Zb/N\Zb$. Furthermore, for $N' = dN$ (and $d$ coprime to $p$ as well), we have a commutative diagram
\[
\xymatrix{
H^2_{\text{\'et}}(B_k\mu_N,\G_m) \ar[r] \ar[d] & \Zb/N\Zb \ar[d]^{d} \\
H^2_{\text{\'et}}(B_k\mu_{dN},\G_m)  \ar[r] & \Zb/dN\Zb.
}
\]
\end{lemma}

\begin{proof}
We have an equivalence $H^2_{\text{\'et}}(B_k\mu_N,\G_m) \simeq H^1_{\text{\'et}}(\Spec k,\Hhom(\mu_N,\G_m)) \simeq H^1_{\text{\'et}}(\Spec k,\Zb/N\Zb) \simeq \Zb/N\Zb$: 

The morphism $H^2_{\text{\'et}}(B_k\mu_N,\G_m) \to H^1_{\text{\'et}}(\Spec k,\Hhom(\mu_N,\G_m))$ can be understood in more concrete terms, using either the language of gerbes, or cocycles. We begin with the first. An element $\alpha$ of $H^2_{\text{\'et}}(B_k\mu_N,\G_m)$ represents a $\G_m$-gerbe $\Gg \to B_k\mu_N$. We claim $H^2_{\text{\'et}}(B_{\bar{k}}\mu_N,\G_m) = 0$. 
\begin{claim}
$H^2_{\text{\'et}}(B_{\bar{k}}\mu_N,\G_m) = 0$
\end{claim}

\begin{proof}
The cohomology group $H^2_{\text{\'et}}(B_{\bar{k}}\mu_N,\G_m)$ classifies central extensions 
$$0 \to \G_m \to E \to \mu_N \to 0$$
of the finite \'etale (and constant) group scheme $\mu_N$ by $\G_m$. Since $\mu_N$ is cyclic, a central extension $E$ as above is automatically abelian (use the short exact sequence (*) in \cite{prasad} to deduce this). Using \cite[Exp. 9 Prop. 8.2]{MR0274459} we conclude that $E$ is a diagonalisable group scheme. 

Cartier duality $\Hhom(-,\G_m)$ yields a short exact sequence
$$0 \to \Zb/n\Zb \to \Hhom(E,\G_m) \to \Zb \to 0$$
of abelian group schemes. Since $\Hhom(E,\G_m)$ has a rational point over $\bar{k}$ this sequence splits. By Cartier dualising a second time we obtain that the original sequence splits (equivalently we could refer to the anti-equivalence of the category of diagonalisable group schemes with the category of abelian groups). 
\end{proof}

This shows that the gerbe $\Gg$ splits when pulled back to $B_{\bar{k}}\mu_N$. We choose such a splitting, which can be either represented by a section $s\colon B_{\bar k}\mu_N \to \Gg$ or an isomorphism $\Gg \times_{B_k\mu_N} B_{\bar{k}}\mu_N \simeq B_{\bar{k}}\mu_N \times B\G_m$.  The Frobenius automorphism $\varphi \in \Gal(\bar{k}/k)$ yields a second splitting $\phi(s)$. Two splittings of the same gerbe differ by a $\G_m$-torsor, and we may therefore view $\varphi(s)/s$ as an element of $\Pic(B_{\bar{k}}\mu_N) = \mu_N^{\vee} = \Zb/N\Zb = H^1_{\text{\'et}}(k,\mu_N^{\vee})$.

The description in terms of cocyles is similar: Since $H^2_{\text{\'et}}(B_{\bar{k}}\mu_N,\G_m) = 0$ there exists a finite field extension $k'/k$, such that $\alpha$ splits when pulled back to $B_{k'}\mu_N$. Therefore, there exists a $1$-cochain $(\psi_{ij})$ for the sheaf $\G_m$ on the small \'etale site of $B_{k'}\mu_N$, such that we have $d(\psi_{ij}) = (\phi_{ijk})_{k'}$. As before we denote by $\varphi\in \Gal(k'/k)$ the Frobenius automorphism. Since $\phi_{ijk}$ is a $2$-cochain defined over $k$ we have
$\varphi(\phi_{ijk})_{k'} = (\phi_{ijk})_{k'}$. Therefore, $\varphi(\psi_{ij})/\psi_{ij}$ is a $1$-cochain representing a $\G_m$-torsor on $B_{k}\mu_N$. It corresponds to an element of $H^1(B_k\mu_N,\G_m) = H^1(k,\Zb/N\Zb)$.

Naturality with respect to the morphism $\mu_N \to \mu_{dN}$ is built into the cocycle computation above.
\end{proof}


\begin{definition}
Let $k$ be a finite field, $\Mc/\Spec k$ an algebraic stack and $\alpha$ a $\G_m$-gerbe on $\Mc$. We define a function 
$$f_{\alpha}\colon I_{\widehat{\mu}}\Mc(k) \to \Qb/\Zb$$
as follows. For a $k$-rational point $y\colon B_k\mu_N \to \Mc$ of $I_{\widehat{\mu}}\Mc$ we let $f_{\alpha}(y)$ be the invariant of $y^*\alpha \in H^2_{\text{\'et}}(B_k\mu_N,\G_m) \simeq \Zb/N\Zb = \frac{1}{N}\Zb/\Zb \subset \Qb/\Zb$ defined in Lemma \ref{lemma:Brauer_stack}. 
\end{definition}

We claim that this definition yields another way to compute the Hasse invariant. The proof of the following proposition will be given below.
\begin{proposition}\label{prop:hasse2}
Let $F$ be a local field and $\Mc/\Oo_F$ a smooth tame $DM$-stack with coarse moduli space $M/\Oo_F$ as in Situation \ref{situation:locally_quotient}(b). We let $\alpha$ be a $\G_m$-gerbe on $\Mc$. For every $x \in M(\Oo_F)^{\natural}$ we have 
$$\inv(x_F^*\alpha) = f_{\alpha}\left(e(x)\right)$$
where $x_F\colon \Spec F \to \Mc$ denotes the induced map, and $e\colon M(\Oo_F)^{\natural} \to I_{\widehat{\mu}}\Mc(k)$ is the specialisation map of Construction \ref{defi:e}.
\end{proposition}

In the following remark and definition we use the notation introduced in Construction \ref{defi:XGamma}.

\begin{rmk}\label{rmk:noncan}
We assume that $Q$ is strongly ramified (see Lemma \ref{lemma:madealemmaoutofit}). There is a (non-canonical) closed immersion of stacks $i\colon B_k \mu_N \hookrightarrow \Xc_{\Gamma,Q}$, such that the induced map of stabiliser groups is given by the canonical inclusion $\mu_N \subset \Gamma$. The induced morphism $Br( \Xc_{\Gamma,Q}) \to Br(B_{k}\mu_N)$ is independent of choices. 
\end{rmk}

\begin{definition}
Assume that $Q$ is strongly ramified. We define the map $f_{\Gamma,Q}\colon H^2_{\text{\'et}}(\Xc_{\Gamma,Q},\G_m) \to \Qb/\Zb$ as the composition
\[f_{\Gamma,Q}\colon H^2_{\text{\'et}}(\Xc_{\Gamma,Q},\G_m) \xrightarrow{i^*} H^2_{\text{\'et}}(B_kI_Q,\G_m) = H^2_{\text{\'et}}(B_k\mu_N,\G_m) = \frac{1}{N}\Zb/\Zb \subset \Qb/\Zb.\]
\end{definition}

\begin{lemma}\label{lemma:pre_hasse}
For a strongly ramified $\Gamma$-torsor $Q$ we have a commutative diagram 
\[
\xymatrix{
H^2_{\text{\'et}}(\Xc_{\Gamma,Q},\G_m) \ar[r]^-{j^*} \ar[rd]_{f_{\Gamma,Q}} & H^2_{\text{\'et}}(F,\G_m) \ar[d]^{\inv} \\
 & \Qb/\Zb.
}
\]
\end{lemma}

\begin{proof}
At first we unravel the definition of the Hasse invariant. A priori, the Brauer group of $F$ is computed by the Galois cohomology group $H^2(\Gal(\bar{F}/F),\bar{F}^{\times})$, however every gerbe on $F$ splits on an unramified cover, and we can therefore identify this Galois cohomology group with $H^2(\Gal(F^{un}/F),(F^{un})^{\times})$. We use the canonical identification $\widehat{\Zb} \simeq \Gal(F^{un}/F) \simeq \Gal(\bar{k}/k)$.

Recall from Lemma \ref{lemma:Brauer_stack} that the Hasse invariant is defined by composition of the isomorphisms 
$$H^2(\Gal(F^{un}/F) ,(F^{un})^{\times}) \simeq H^2(\Gal(\bar{k}_F/k),\Zb)) \simeq H^1_{\text{\'et}}(\widehat{\Zb},\Qb/\Zb) \simeq \Qb/\Zb.$$
The morphism on the left is induced by the natural map $(F^{un})^{\times} \to \Zb$ which sends $\underline{\pi} \mapsto 1$. We observe  that this map has a section $\Zb \to (F^{un})^{\times}$ sending $1$ to $\underline{\pi}$.

Let $\alpha \in Br(F)$. We write $\phi=(\phi_{ijk})$ for the $2$-cocycle on the site of unramified covering of $F$ (that is, the site equivalent to the small \'etale site of $\Oo_F$) of the constant sheaf $\Zb$ corresponding to $\alpha$ under the above isomorphism. By the paragraph above, the gerbe $\alpha$ is represented by the $2$-cocycle $\underline{\pi}^{\phi_{ijk}}$.

By assumption, $\ord(\alpha)$ divides $N=|I_Q|$ and is coprime to $p$. Therefore, $(N\cdot{}\phi_{ijk})$ is a coboundary. That is, there exists a $1$-cochain $\psi=(\psi_{ij})$, such that $d(\psi) = N \cdot \phi$.

Let $L/F$ be a totally ramified field extension which splits the torsor $Q$, and let $\underline{\pi}^{1/N}$ be a uniformiser for $L$. We then have a $1$-cochain $(\underline{\pi}^{\frac{\psi_{ij}}{N}})$ with values in $L^{\times}$, such that 

\begin{equation}\label{eqn:3}d(\underline{\pi}^{\frac{\psi_{ij}}{N}}) = \underline{\pi}^{\phi_{ij}}.\end{equation}

We now define a $2$-cochain taking values in the sheaf $\Zb/N\Zb = \Hhom(\mu_n,\G_m)$. It is given by 
$$c_{ij}\colon \xi \mapsto \xi^{\psi_{ij}}$$
for $\xi \in \mu_N$. We claim that $d(c_{ij}) = 1$, indeed for $\xi \in \mu_N$ we have a field automorphism $\sigma_{\xi}$ of $L/F$ (and all unramified base changes thereof) which sends $\underline{\pi}^{\frac{1}{N}} \mapsto \xi \cdot{} \underline{\pi}^{\frac{1}{N}}$. Applying $\sigma_{\xi}$ to \eqref{eqn:3} we obtain
$$d(\xi^{\psi_{ij}} \underline{\pi}^{\frac{\psi_{ij}}{N}}) = \underline{\pi}^{\phi_{ij}}.$$
This shows $d(\xi^{\psi_{ij}}) = 1$ as we wanted. Hence $c_{ij}= (\xi \mapsto \xi^{\psi_{ij}})$ is a $1$-cocycle. 
We claim that this is the $1$-cocycle which represents $f_{\Gamma,Q}(\alpha)$. 
\begin{claim}
The $1$-cocycle $(c_{ij})$ represents $f_{\Gamma,Q}(\alpha) \in H^1(k,\mu_N^{\vee})$.
\end{claim}

\begin{proof}
Recall that $\alpha$ is represented by the $2$-cocycle $(\underline{\pi}^{\phi_{ijk}})$ on the small site of unramified \'etale schemes over $F$ (equivalently, the small \'etale site of $\Oo_F$). We pull back to the $\Gamma$-torsor $Q \to \Spec F$ where we may choose an $N$-th root of $\underline{\pi}$, which we denote by $\underline{\pi}^{\frac{1}{N}}$. The pulled back gerbe $\alpha_Q$ is still represented by the cocycle $(\underline{\pi}^{\phi_{ijk}})$ but this time it is a coboundary, as we have 
$$(\underline{\pi}^{\phi_{ijk}}) = \left((\underline{\pi}^{\frac{1}{N}})^{N\cdot{}\phi_{ijk}}\right) = \left( (\pi^{\frac{1}{N}})^{d(\psi_{ij})} \right).$$
In the theory of gerbes, the $1$-cochain $\left( (\pi^{\frac{1}{N}})^{(\psi_{ij})} \right)$ represents a splitting of the pulled back gerbe $\alpha_Q$.
We have an action of $\mu_N$ on $Q$ which sends $\underline{\pi}^{\frac{1}{N}}$ to $\xi \mapsto \xi \cdot{} \underline{\pi}^{\frac{1}{N}}$. With respect to this action, the splitting of $\alpha_Q$ represented by $\left( (\pi^{\frac{1}{N}})^{(\psi_{ij})} \right)$ is sent to 
$$\left(\xi \mapsto  (\xi^{\psi_{ij}}\pi^{\frac{1}{N}})^{(\psi_{ij})} \right) = \left( c_{ij}(\xi)\cdot{} \underline{\pi}^{\psi_{ij}} \right) .$$
This shows that $(c_{ij})$ is a cocycle representing $f_{\Gamma,Q}(\alpha)$.
\end{proof}

By identifying $\Zb/N\Zb$ with $\mu_N^{\vee}$ we obtain the $1$-cocycle $(\psi_{ij})$ in $\Zb/N\Zb$. It is easy to see that the boundary map of the Bockstein sequence
$$0 \to \Zb \to \Zb \to \Zb/N\Zb \to 0$$
sends $(\phi_{ijk})$ to the cocycle $(\psi_{ij})$ in $\Zb/N\Zb$. By means of the commutative diagram with exact rows
\[
\xymatrix{
0 \ar[r] & \Zb \ar[r] \ar[d] & \Zb \ar[r] \ar[d] & \Zb/N\Zb \ar[r] \ar[d] & 0 \\
0 \ar[r] & \Zb \ar[r] & \Qb \ar[r] & \Qb/\Zb \ar[r] & 0
}
\]
and the definition of the Hasse invariant as image of $(\psi_{ijk}) \in H^2(k,\Zb) \simeq H^1(k,\Qb/\Zb)$ we conclude $f_{\Gamma,Q}(\alpha) = \inv(\alpha)$.
%
\end{proof}

We now turn to the proof of the proposition above.

\begin{proof}[Proof of Proposition \ref{prop:hasse2}]
For $x \in M(\Oo_F)^{\natural}$ we want to show the identity $\inv(x_F^*\alpha) = f_{\alpha}(e(x))$. In Construction \ref{constrspe} we showed that $x\colon \Spec \Oo_F \to M$ fits into a commutative diagram
\[
\xymatrix{
& \Xc_{\Gamma,Q} \ar[r] \ar[d] & \Mc \ar[d] \\
\Spec F \ar[r] \ar[ru] & \Spec \Oo_F \ar[r]^x & M.
}
\]
The induced morphism $B_k\mu_N \to \Mc_k$ corresponds to $e(x) \in I_{\widehat{\mu}}\Mc(k)$.
This shows that $x_F^*\alpha$ extends to a $\G_m$-gerbe on $\Xc_{\Gamma,Q}$. By virtue of Lemma \ref{lemma:pre_hasse} we have $\inv(x_F^*\alpha) = f_{\Gamma,Q}(x_F^*\alpha)$. This shows what we wanted. 
\end{proof}

We end this section with an example computation of a Hasse invariant using Proposition \ref{prop:hasse2}. This observation will play an important role later on.

\begin{definition}\label{defi:unramified_gerbe}
Let $\Mc/\Oo_F$ be a proper DM-stack over the ring of integers of a local field $F$, and $L \in \Pic(\Mc)$ a $\G_m$-torsor of finite order. We define the gerbe $\alpha_L \in H^2_{\text{\'et}}(\Mc,\G_m)$ to be the the element induced by $H^1(\pi_1^{\text{\'et}}(\Oo_F),H^1_{\text{\'et}}(\Mc_{\Oo_{F^{\text{un}}}},\G_m))$ corresponding to the morphism $\widehat{\Zb} \to H^1_{\text{\'et}}(\Mc,\G_m)$ given by $1 \mapsto L$.
\end{definition}

A $\G_m$-torsor on $\B_k\mu_N$ corresponds to a character $\mu_N \to \G_m$, that is, an element of $\mu_N^{\vee} = \Zb/N\Zb \subset \Qb/\Zb$. We denote the corresponding isomorphism by $H^1_{\text{\'et}}(B_k\mu_N,\G_m) \simeq \mu_N^{\vee} = \Zb/N\Zb$ by $\chi \leftrightarrow L_{\chi}$.

\begin{lemma}\label{lemma:unramified_gerbe}
Using the notation of Definition \ref{defi:unramified_gerbe} we have $f(\alpha_{L_{\chi}}) = \chi \in \Zb/N\Zb \subset \Qb/\Zb$ where $f$ denotes the invariant of Lemma \ref{lemma:Brauer_stack}.
\end{lemma}

\begin{proof}
We have seen in the proof of Lemma \ref{lemma:Brauer_stack} that 
$$H^2_{\text{\'et}}(B_k\mu_N,\G_m) \simeq H^1_{\text{\'et}}(k,\Hhom(\mu_n,\G_m))\simeq H^1_{\text{\'et}}(k,\Zb/N\Zb).$$
By direct inspection we see that $f(\alpha_{\chi})$ corresponds to the $\Zb/N\Zb$-torsor given by the continuous map $\widehat{\Zb} \to \Zb/N\Zb$ which sends the topological generator $1 \mapsto \chi$.
\end{proof}

Proposition \ref{prop:hasse2} immediately implies the following assertion. Recall that a $\G_m$-torsor $L$ on a stack $\Mc$ induces a function $\chi_L\colon I_{\widehat{\mu}}\Mc(k) \to \Qb/\Zb$. Indeed, $y \in I_{\widehat{\mu}}\Mc(k)$ corresponds to a morphism $y\colon B_k\mu_N \to \Mc$. Pullback of $L$ along $y$ yields $y^*L \in \Pic(B_k\mu_N)$ and the latter corresponds to a character $\mu_N \to \G_m$, that is, an element of $\Zb/N\Zb \subset \Qb/\Zb$.

\begin{corollary}\label{cor:hasse}
Let $\Mc$ be a proper tame DM-stack as in Situation \ref{situation:locally_quotient}(b), and $L \in \Pic(\Mc)$. For $x \in M(\Oo_F)^{\natural}$ we have $\inv(x_F^*\alpha_L) = \chi_L(e(x))$.
\end{corollary}

Henceforth we view this as an identity of complex valued functions, by virtue of the embedding 
$$\Qb/\Zb \hookrightarrow \Cb^{\times},$$
given by $\lambda \mapsto e^{2\pi i \lambda}$.

\section{Preliminaries on Higgs bundles} \label{HiggsBundles}

In this section we recall basic facts on moduli stacks of $G$-Higgs bundles where $G/X$ is a quasi-split reductive group scheme on $X$.

\subsection{Quasi-split reductive group schemes}

\begin{definition} \label{PinnDef}
  Let $S$ be a scheme and $G$ a reductive group scheme over $S$. 
  \begin{enumerate}[(i)]
 \item  A \emph{pinning} of $G$  over $S$ is a triple $(T,B,s)$, such that $T$ is a maximal torus of $G$ over $S$, $B$ a Borel subgroup of $G$ over $S$ containing $T$ and $s$ a section of $\Lie B$ over $S$, such that there exists an \'etale covering $S'$ of $S$ satisfying the following: Over $S'$ the groups $T$, $B$ and $G$ become constant and hence $B$ admits root subgroups $U_\alpha$ for each root $\alpha$ in the set $\Phi^+$ of roots $T$ appearing in $\Lie B$. Then over $S'$ the section $s$ decomposes as $s=\sum_{\alpha \in \Phi^+} s_\alpha$ for nowhere vanishing sections $s_\alpha$ of $\Lie U_\alpha$ over $S'$.
 \item A pinning $(T,B,s)$ is \emph{split} if the torus $T$ is split.
  \item Let $H\subset G$ be a closed reductive subgroup scheme of $G$. A pinning $(T',B',s')$ of $H$ over $S$ is \emph{compatible} with a pinning $(T,B,s)$ of $G$ over $S$ if $T'=T$, $B'\subset B$ and if over a suitable \'etale covering $S'$ of $S$ over which all the appearing groups become constant and the sections $s$ and $s'$ decompose into summands $s_\alpha$ and $s'_\alpha$ we have $s_\alpha=s'_\alpha$ for those roots $\alpha$ which appear in $\Lie B'$.
  \end{enumerate}
\end{definition}

The following fact follows from \cite[XXIV.3.10]{SGA3III}:
\begin{proposition}
  For a reductive group scheme $G$ over $S$, the sheaf of pinnings of $G$ is a torsor under $G^\text{ad}$.  
\end{proposition}
 This implies the following:
\begin{lemma} \label{PinnSplit}
  Let $S$ be a scheme, $G$ a reductive group scheme over $S$. Consider the exact sequence
  \begin{equation*}
    1 \to G^\text{ad} \to \Aut(G) \to \Out(G) \to 1
  \end{equation*}
of sheaves on $S$. For a pinning of $G$ over $S$, the subsheaf of $\Aut(G)$ consisting of those automorphisms which fix the pinning maps isomorphically to $\Out(G)$. Hence every pinning of $G$ gives a splitting of the above sequence.
\end{lemma}
Using Lemma \ref{PinnSplit}, whenever we are working with a reductive group scheme $G$ and a fixed pinning of $G$, we will consider $\Out(G)$ as a subsheaf of $\Aut(G)$.

\begin{lemma} \label{CompPinn}
    Let $S$ be a scheme, $G$ a reductive group scheme over $S$ and $H \subset G$ a closed reductive subgroup scheme of $G$. Let $(T,B,s)$ be a pinning of $G$, such that $T\subset H$. Then there exists a unique pinning $(T',B',s')$ of $H$ which is compatible with $(T,B,s)$. 
\end{lemma}
\begin{proof}
 We have to take $T'=T$ and $B'=B\cap H$ and it remains to consider $s'$. Over an \'etale covering $S'$ of $S$ as in Definition \ref{PinnDef} existence and uniqueness of $s'$ follow directly from the definition. Hence the same holds over $S$ by descent.
\end{proof}


\begin{definition}
   Let $S$ be a scheme and $\G$ a reductive group scheme over $S$ with a pinning $(T,B,s)$. An outer form of $\Gb$ on $S$ is given by a group scheme $G \to S$ together with an $\Out(\Gb)$-torsor $\rho_G$ on $S$ as well as an isomorphism of group schemes $G \simeq \rho_G \times^{\Out(\Gb)}_S \Gb$. This isomorphism induces a pinning of $G$.
\end{definition}

\begin{definition}
  A reductive group scheme $G$ over a scheme $S$ is \emph{quasi-split} if it admits a pinning.
\end{definition}

\begin{lemma}
  Let $S$ be a scheme and $G$ a quasi-split reductive group scheme over $S$ with a chosen pinning. There exists a split reductive group scheme $\Gb$ over $S$ together with a split pinning and an $\Out(\Gb)$-torsor $\rho$ such that there exists an isomorphism $G \cong \Gb \times^{\Out(\Gb)} \rho$ respecting the pinnings on these group schemes.
\end{lemma}
\begin{proof}
  By the theory of reductive group schemes (c.f. e.g. \cite[Theorem 6.1.16]{MR3362641}), there exists a split reductive group scheme $\Gb$ on $\Spec(\Zb)$ together with a split pinning which is isomorphic to $G$ compatibly with the pinnings \'etale-locally on $S$. Then we may take $\rho$ to be the sheaf which sends $S' \to S$ to the set of isomorphisms $G_S \cong \Gb_S$ which respect the pinnings. By Lemma \ref{PinnSplit} this is an $\Out(\Gb)$-torsor.
\end{proof}
Now let $k$ be a field and $\Gb$ a split reductive group scheme $\Gb$ over $k$. We fix a split pinning of $\Gb$.

The group $\Out(\Gb)$ is not necessarily finite. Indeed, for $\Gb$ a torus of rank $r \geq 2$ the group $\Out(\Gb) = \GL_r(\Zb)$ is infinite. However, the following lemma shows that an $\Out(\Gb)$-torsor $\rho$ over a a smooth projective curve over $k$ always admits a reduction to a finite group scheme:

\begin{lemma} \label{Isotriviality}
 Let $X$ be a smooth projective curve over $k$ and $\rho$ an $\Out(\Gb)$-torsor on $X$. There exists a finite \'etale covering $X'$ of $X$ over which the torsor $\rho$ becomes trivial.
\end{lemma}
\begin{proof}
  Since $X$ is normal, by \cite[X.5.16]{SGA3II}, there exists a finite \'etale covering $X'$ of $X$ over which the maximal torus $T$ of $G$ becomes split. Then Zariski-locally on $X'$, the pinning of $G$ gives a splitting of $G$ in the sense of \cite[XXIII.1.1]{SGA3III}. Hence by the unicity of split reductive group schemes (c.f. \cite[XXIII.4.1]{SGA3III}) we see that Zariski-locally on $X'$ there exists an isomorphism $G\cong \Gb$ respecting the pinnings. That is the torsor $\rho$ splits Zariski-locally on $X'$. But since $\rho_{X'}$ is a disjoint union of proper curves which are \'etale over $X'$, this implies that $\rho$ splits over $X'$.
\end{proof}

We now discuss Langlands dual groups:

\begin{definition}
Let $\widehat{\Gb}/k$ be the reductive group classified by the root datum dual to the one of $\Gb/k$. There exists a natural isomorphism $\Out(\widehat{\Gb}) = \Out(\Gb)$. For a scheme $S$ over $k$ and an outer form $G$ of $\Gb$ over $S$ given by a $\Out(\Gb)$-torsor $\rho$ over $S$ we call the quasi-split outer form $\widehat G$ of $\widehat \Gb$ over $S$ given by $\rho$ considered as an $\Out(\widehat\Gb)$-torsor the Langlands dual of $G$. 
\end{definition}




\subsection{An overview of the Hitchin system}\label{sub:overview}

In order to fix terminology we give a brief overview of the basic definitions and constructions. We consider a relatively general framework at first, but remind the reader that in this article the ring $R$ below will either be a finite field, or $\Oo_F$, the ring of integers of a local field $F$.

\begin{situation}\label{situation:Higgs_general}
\begin{enumerate}[(a)]
\item Let $R$ be a commutative ring and $X_R/R$ a flat family of smooth projective curves with geometrically connected fibres. 
\item Let $D$ be a line bundle on $X_R$, of even degree $d$.
\item We fix a split reductive group scheme $\Gb$ on $\Spec(R)$ together with a split pinning $(\Tb,\mathbb{B},s)$ (see Definition \ref{PinnDef}) and consider a quasi-split form $G \to X_R$ of $\Gb$ on $X_R$. That is, we fix an $\Out(\Gb)$-torsor $\rho$ on $X_R$, such that $G = \Gb \times^{\Out(\Gb)} \rho$. The induced maximal torus of $G$ is denoted by $T/X$. 
\item We denote the (sheaves of) Lie algebras of $\Gb$, $G$, $\Tb$ and $T$ by $\mathbf{g}$, $\mathfrak{g}$, $\mathbf{t}$ and $\mathfrak{t}$ respectively. 
\item We assume that there exists a finite subgroup scheme $\Theta \hookrightarrow \Out(\Gb)$ such that the order $|\Theta \ltimes \mathbb{W}|$ is invertible on $X_R$ and such that there exists a $\Theta$-reduction of $\rho$.
\end{enumerate}
\end{situation}

\begin{rmk}
  By Lemma \ref{Isotriviality}, there always exists a reduction of $\rho$ to a finite subgroup $\Theta \into \Out(\Gb)$ in case $R$ is a field.
\end{rmk}
For the rest of this section we put ourselves into Situation \ref{situation:Higgs_general}. Then we define $G$-Higgs bundles on $X/R$ as follows. Recall that the adjoint bundle of a $G$-torsor $E$ is defined to be $\ad(E) = (E \times_{X_R}^{G} \mathfrak{g})$.

\begin{definition}\label{defi:higgs}
\begin{enumerate}[(a)]
\item  A $G$-Higgs bundle (with coefficients in $D$) is a pair $(E,\theta)$ where $E$ is a $G$-torsor on $X_R$ and $\theta$ is a global section of $\ad(E) \otimes D$.
\item We denote by $\mathbb{M}_G = \mathbb{M}_G(X,D)$ the stack which associates to an affine $R$-scheme $S$ the groupoid of $G$-Higgs bundles on $X_S/S = X \times_R S/S$.
\item The generic stabiliser group of $\mathbb{M}_G$ is the group scheme $Z(X,G)$ obtained as the Weil restriction of the group scheme $Z(G)$ over $X$ along the proper morphism $X \to \Spec(R)$ (the representability of this group scheme follows from \cite[Theorem 1.5]{MR2239345}). The rigidification (\`a la \cite[Theorem 5.1.5]{abramovich2003twisted}) of $\mathbb{M}_G$ by this group scheme will be denoted by $\Mc_G$. 
\end{enumerate}
\end{definition}

There exists an elegant reformulation of Definition \ref{defi:higgs}(a), which appears in \cite[4.2.2]{MR2653248}. 

\begin{construction}\label{construction:ngo_higgs}
We denote by $\mathfrak{g}_D$ the principal Lie algebra bundle on $X_R$ obtained by twisting $\mathfrak{g}/X$ of Definition \ref{defi:higgs}(d) with the line bundle $D$. That is, we define $\mathfrak{g}_D = \mathfrak{g} \times_X^{\G_m} D$. The adjoint action of $\Gb$ on $\mathbf{g}$ induces a relative action of $G$ on $\mathfrak{g}_D$. The datum of a $G$-Higgs bundle $(E,\theta)$ with coefficients in $D$ is equivalent to a section of the morphism of $X$-stacks
\[
\xymatrix{
[\mathfrak{g}_D/G] \ar[r] & X \ar@{=}[d] \\
& X. \ar[lu]
}
\] 
Indeed, by the definition of quotient stacks, a section $X \to [\mathfrak{g}_D/G]$ corresponds to a commutative diagram
\[
\xymatrix{
\mathfrak{g}_D \ar[r] & X \ar@{=}[d] \\
E \ar[r] \ar[u]_{\tilde{\theta}} & X 
}
\]
where $E$ is a $G$-torsor, and $\tilde{\theta}$ is a $G$-equivariant map $E \to \mathfrak{g}_D$. By $G$-equivariance, the morphism $\tilde{\theta}$ gives rise to a global section $\theta\colon X \to E\times^G \mathfrak{g}_D$. We leave the verification of the opposite direction that a $G$-Higgs bundle $(E,\theta)$ gives rise to a commutative diagram as above, to the reader.
\end{construction}

This abstract viewpoint is convenient in defining the Hitchin morphism. The reference for point (a) below is \cite[1.1]{MR2653248}.

\begin{definition}\label{defi:Hitchin_map}
\begin{enumerate}[(a)]
\item We denote by $\mathbf{c}$ the Chevalley base of $\mathbf{g}$. That is, $\mathbf{c} = \mathbf{g}/G$. We write $\mathfrak{c}/X$ for the induced affine space bundle on $X$. 
\item The relative Chevalley base $\mathfrak{c}$ is endowed with a $\G_m$-action. We use the notation $\mathfrak{c}_D = \mathfrak{c} \times^{\G_m}_X D$.
\item Let $\A = \A_G(X,D)$ be the affine scheme corresponding to the vector space $H^0(X,\mathfrak{c}_D)$.
\item Let $S$ be an affine $R$-scheme. The Hitchin map $\chi \colon \mathbb{M}_G \to \A$ sends an $S$-family of Higgs bundle $(E,\theta)$ to the global section of $\mathfrak{c}_D$ defined by the commutative diagram
\[
\xymatrix{
\mathfrak{c}_D  & X \ar[l] \ar[ld] \\
[\mathfrak{g}_D/G]. \ar[u] & 
}
\] 
\end{enumerate}
\end{definition}

The Hitchin base $\A$ corresponds to the vector space $H^0(X,\mathfrak{c}_D)$. If $G/X$ is split the latter can be identified with a direct sum $\bigoplus_{i=1}^m H^0(X,D^{\otimes e_i})$, where $e_1,\dots, e_m$ are positive integers (see \cite[4.13]{MR2653248}). 

In order to study the Hitchin map we introduce the cameral cover. 

\begin{definition}\label{defi:cameral_cover}
Let $S$ be an affine $R$-scheme, and $a\in \A(S)$ be an $S$-valued point of the Hitchin base. The cameral cover $\Cc_a$ is defined to be the finite $W$-equivariant morphism $\Cc_a \to X \times_R S$ which belongs to a cartesian diagram
\[
\xymatrix{
\Cc_a \ar[r] \ar[d] & \mathfrak{t}_D \ar[d] \\
X \times_R S \ar[r]^a &  \mathfrak{c}_D.
}
\]
We denote the universal cameral cover by $\Cc \to X \times_R \A$.
\end{definition}

We recall in Subsection \ref{sub:prym} how cameral covers can be used to understand a natural family of Picard $\A$-stacks, the so-called abstract Prym, acting on the Hitchin system. Using work of Donagi--Gaitsgory (\cite{MR1903115}), these strict Picard stacks turn out to be moduli stacks of $T$-torsors on $\Cc$ equipped with a certain extra structure. We recall this picture in Definition \ref{defi:donagi-gaitsgory}.

\subsection{The abstract Prym}\label{sub:prym}

The moduli stack $\mathbb{M}_G$ is acted on by a strict Picard stack $\Pb$, that is, an abelian group object in stacks. 

\begin{definition}
A stack $\Pb$ together with a strict symmetric monoidal structure $\otimes \colon \Pb \times \Pb \to \Pb$ is called a \emph{strict Picard stack}, if the induced monoid structure on the sheafification of isomorphism classes
$\pi_0^{sh}(\Pb)$ is a group structure.
\end{definition}

\begin{definition}
\begin{enumerate}[(a)]
\item We denote by $J/[\mathfrak{c}/\G_m]$ the sheaf of regular centralisers as defined in \cite[2.1]{MR2653248}.
\item For $a\in \A(S)=H^0(X_S,\mathfrak{c}_D)$ we write by abuse of notation $J/X_S$ for the induced commutative group scheme $a^* J$.
\item For $a\in \A(S)$ we denote the stack of $J$-torsors by $\Pb=\Pb_{G,X_S,D}$.
\end{enumerate}
\end{definition}

The action of the Prym $\Pb$ on $\mathbb{M}$ relative to the Hitchin base $\A$ is defined in \cite[4.3.2]{MR2653248}. We recall the main points of the definition below.

\begin{construction}\label{const:action}
Let $S$ be an affine $R$-scheme and $(E,\theta)$ a Higgs bundle, giving rise to an $S$-point $a \in \A(S)$ of the Hitchin base. There exists a natural morphism $J \to \Aut(E,\theta)$. Given a $J$-torsor $M$ we define $M \otimes E = E \times_X^{J} M$. Since $J$ acts on $E$ through automorphisms preserving the Higgs field, the twist $M \otimes E$ is endowed with a canonical section of $\ad(E \otimes M) \otimes D$. We denote the resulting Higgs bundle by $M \otimes (E,\theta)$, or $(M \otimes E,\theta)$.
\end{construction}

There exists an open subset $\mathbb{M}^{\reg}$, which is non-empty if $D$ has even degree, and is a $\Pb$-torsor in this case:

\begin{definition}
A Higgs bundle $(E,\theta)$ is regular, if the classifying section $X \to [\mathfrak{g}_D/G]$ factors through the open dense substack $[\mathfrak{g}^{\reg}/G]$. Similarly, we say that an $S$-family of Higgs bundles is regular, if its classifying section $X \times_R S \to [\g_D / G]$ factors through $[\g^{\reg}_D/G]$. The resulting open substack will be denoted by $\mathbb{M}^{\reg}$.
\end{definition}

Details for the following proposition can be found in \cite[4.2.4 and Proposition 4.3.3]{MR2653248}.

\begin{proposition}\label{prop:kostant}
\begin{enumerate}[(i)]
\item To any square root $D'$ of $D$ (that is, $D'^{\otimes 2} \simeq D$) we may canonically associate a section $[\epsilon]^{D'}\colon \A \to \mathbb{M}^{\reg}$, the so-called \emph{Kostant section}. In particular, the open subset $\mathbb{M}^{\reg}$ is non-empty. 
\item The action of $\Pb$ on $\mathbb{M}^{\reg}$ relative to $\A$ defines a torsor structure on $\mathbb{M}^{\reg}$. In particular, every square root $D'$ of $D$ as in (a) yields a trivialisation of torsors $\epsilon^{D'}\colon \Pb \xrightarrow{\sim} \mathbb{M}^{\reg}$.
\end{enumerate}
\end{proposition}

Furthermore, there exists an open subset $\A^{\Diamond} \subset \A$, such that $\Pc^{\Diamond} = \Pc \times_{\A} \A^{\Diamond}$ is a so-called Beilinson $1$-motive.

\begin{definition}\label{beil}
A strict Picard stack $\Pb\to S$ on a scheme $S$ is called a Beilinson $1$-motive, if there exists a surjective \'etale covering $S' \to S$, such that $\Pb \times_S S'$ is equivalent to a product 
$$C \times_{S'} A \times_{S'} B_{S'}K,$$
where $C$ and $K$ are finite \'etale group schemes over $S'$, and $A \to S'$ is a (connected) abelian $S'$-scheme.
\end{definition} 

The $2$-category of Beilinson $1$-motives has a natural duality which simultaneously extends the classical construction of the dual abelian variety and Cartier duality for finite \'etale group schemes. It is for this reason that we care about determining where the Prym variety $\Pb$ is a Beilinson $1$-motive.

\begin{definition}\label{lemma:Agood}
We denote by $\A^{\Diamond}$ the open subset corresponding to geometric points $a\colon \Spec \bar{k} \to \A$, such that the cameral cover $\Cc_a$ is smooth. We denote by $\mathbb{M}^\Diamond$, $\Pb^\Diamond$, etc. the pullbacks of $\mathbb{M}$, $\Pb$, etc. to $\A^\Diamond$.
\end{definition}

For a divisor of high degree the set $\A^\Diamond$ is non-empty. It is shown in \cite[Proposition 4.7.1]{MR2653248} that this is the case for $\deg D > 2g$.

It follows from \cite[Proposition 4.7.7]{MR2653248} that $\Pb^{\Diamond}$ is a Beilinson $1$-motive:

\begin{lemma}\label{lemma:B1}
We have $\mathbb{M}^{\Diamond} \subset \mathbb{M}^{\reg}$. Furthermore, $\Pb^{\Diamond}$ is a Beilinson $1$-motive.
\end{lemma}

We now discuss Donagi--Gaitsgory's description of $J$-torsors \cite{MR1903115}, following the exposition of Chen--Zhu \cite[Section 3]{chenzhu}. Although \emph{loc. cit.} is formulated for the split case $G = \Gb \times X$, we remark that their description can be immediately extended to the quasi-split case, as their description is \'etale-local on $X$. We refer the reader to the beginning of section 3 of \emph{loc. cit.} for a more detailed exposition of these technicalities.

\begin{situation}\label{situation:donagi-gaitsgory}
Let $S$ be an affine $R$-scheme, and $a \in \A(S)$. We denote by $\Cc_a \to X \times_R S$ the corresponding cameral cover, and by $J$ the sheaf of regular centralisers. Recall that we have a finite \'etale group scheme $W \to X$ and \'etale morphisms $\Phi \to X$ and $\Phi^{\vee} \to X$ whose sections are roots, respectively coroots of $G$. 
\end{situation}

\begin{definition}\label{defi:donagi-gaitsgory}
\begin{enumerate}[(a)]
\item Consider the maximal torus $T \subset G$ which is part of the pinning of $G$. There exists a natural action of the finite \'etale group scheme $W \to X$ on $T$, which we denote by 
$$\mathsf{act}_1\colon W \times_X T \to T.$$ 
\'Etale-locally on $X$, the group scheme $G$ splits and this action is equivalent to the standard action of $\mathbb{W}$ on $\Tb$. 
\item We refer to the canonical action of $W$ on $\Cc_a$ by $\mathsf{act}_2\colon W \times_X \Cc_a \to \Cc_a$.
\item Let $E$ be a $T$-torsor on $\Cc_a$. For an \'etale morphism $U \to X \times_R S$ and $w \in W(U)$ we write
$$w(E) = (\mathsf{act}_2(w^{-1})^*E) \times^{T,w}_{X \times_R S} T.$$
This defines an action of $W$ on the stack $B_{\Cc_a}T$ of $T$-torsors on the small \'etale site of $\Cc_a$. 
\item The stack which sends $S \to \Spec R$ to the groupoid of $W$-equivariant $T$-torsors on $\Cc_a$ will be denoted by $\Bun^W_T(\Cc_a)$.
\item For a $W$-equivariant $T$-torsor we denote by $\mathsf{can_{\alpha,\check \alpha}}\colon (E|\Cc^{s_\alpha}_a) \times^T \G_m \times^{\check \alpha,\G_m} T \simeq \Cc_a^{s_\alpha} \times T$ the isomorphism induced from the $W$-equivariant structure.
\item A $+$-structure on $E \in \Bun^W_T(\Cc_a)(S)$ is given by the following datum: for every \'etale $U \to X \times_R S$ and sections $\alpha \in \Phi(U)$, an isomorphism $c_{\alpha}\colon (E|_{\Cc_a^{s_\alpha}}) \times^T \G_m \simeq \Cc_a^{s_\alpha} \times \G_m$, such that for an \'etale $U' \to U$ and a section $\check \alpha \in \Phi^{\vee}(U')$ we have
$$c_{\alpha} \times^{\check \alpha,\G_m} \id_T = \mathsf{can}_{\alpha,\check \alpha}\colon (E|\Cc^{s_\alpha}_a) \times^T \G_m \times^{\check \alpha,\G_m} T \xrightarrow{\simeq} \Cc_a^{s_\alpha} \times T.$$
The stack of $W$-equivariant $T$-torsors on $\Cc_a$ with a $+$-structure is denoted by $\Bun_T^{W,+}(\Cc_a)$.
\end{enumerate}
\end{definition}

With this long definition at hand we can finally state Donagi--Gaitsgory's description of the Prym variety $\Pb_a$.

\begin{theorem}[Donagi--Gaitsgory]\label{thm:dg}
There exists a canonical equivalence of strict Picard stacks 
\[\Bun_T^{W,+}(\Cc_a) \simeq \Pb_a.\]
\end{theorem}

\begin{construction}[Abel-Jacobi map]\label{const:AJ}
For a scheme $S$ over $\Spec(R)$ we denote by $\Cc_a^{sm}$ the maximal open subset of $\Cc_a$ for which $\Cc_a^{sm} \to S$ is smooth.
For an $S$-point $(x,\lambda) \in \Cc_a^{sm} \times \Xb_*(T)$, we denote by $\Oo(\lambda x)$ the induced $T$-torsor $\Oo_{\Cc}(x) \times^{\lambda,\G_m} T$. Using the same methods as in \cite[3.4]{chenzhu} one associates to $\Oo(\lambda x)$ a $W$-equivariant $T$-torsor on $\Cc_a$ with a $+$-structure.

We denote by $\Cc_a^{sm} \times_X \Xb_*({T}) \to \Pb_a \simeq \Bun_T^{W,+}(\Cc_a)$ the corresponding morphism and refer to it as the \emph{Abel-Jacobi map}.
\end{construction}

Until the end of this subsection we work over the open subset $\A^{\Diamond} \subset \A$. Pulling back a line bundle on $\Pb_G^\Diamond$ along the Abel-Jacobi map yields a line bundle on $\Cc^{\Diamond} \times_X \Xb_*(T)$. We have $\Xb_*(T) = \Xb^*(\widehat{T})$. The datum of a line bundle $L$ on $\Cc^{\Diamond} \times_X \Xb^*(\widehat{T})$ is equivalent to a $\widehat{T}$-torsor on $\Cc^{\Diamond}$: indeed there is a unique $\widehat{T}$-torsor $Q$ (up to a unique isomorphism) on $\Cc^{\Diamond}$, such that for every \'etale local section $\lambda$ of $\Xb^*(\widehat{T})$ we have an isomorphism $\lambda(Q) \simeq L|_{\Cc^{\Diamond} \times \{\lambda\}}$.

As explained in \cite[Section 3]{chenzhu}, the resulting $\widehat{T}$-torsor can be obtained with a strongly $W$-equivariant structure and a $+$-structure. This leads to a morphism of strict Picard stacks
$$\AJ_{G}^*\colon (\Pb_G^{\Diamond})^{\vee} \to \Pb_{\widehat{G}}^\Diamond.$$
According to the following theorem this is an equivalence of strict Picard stacks. This has been shown in characteristic zero and split reductive group schemes by Donagi--Pantev \cite{MR2957305} and by Chen--Zhu \cite[Section 3]{chenzhu} for split reductive group schemes in positive characteristic. In the appendix we explain how to prove this for quasi-split reductive group schemes.

\begin{theorem}[Donagi--Pantev, Chen--Zhu, Corollary \ref{cor:duality}]\label{thm:duality_quasi-split}
Let $X/k$ be a smooth projective curve over a field $k$, and $G/X$ a quasi-split reductive group scheme, that is, an outer form of a split reductive group $\Gb/k$. If $\charac(k)$ is $0$ or, $|\mathbb{W}|$ is not divisible by $\charac(k)$, the morphism
$\AJ_{G}^*\colon (\Pb_G^{\Diamond})^{\vee} \to \Pb_{\widehat{G}}^\Diamond$ is an equivalence of strict Picard stacks.
\end{theorem}

In \emph{loc. cit.} the reader will find the superfluous assumption of $k$ being algebraically closed. This can be relaxed by virtue of faithfully flat descent.

\begin{definition}\label{defi:Pc}
We denote the rigidification (see \cite[Theorem 5.1.5]{abramovich2003twisted}) of $\Pb$ with respect to the abelian group $Z(X,G)=H^0(X,Z(G))$ by $\Pc$. By definition, $\Pc \to \A$ is an algebraic space with a relative group object structure.
\end{definition}

By virtue of Lemma \ref{lemma:B1} we see that the fibres of $\Pc$ are extensions of abelian finite \'etale group schemes by abelian varieties. Similarly we see that the Deligne-Mumford stack $\Mc$ is acted on by $\Pc$ (relatively to $\A$), and obtain an analogue for $\Mc^{\reg}$ of Proposition \ref{prop:kostant}.

In \cite[Proposition 4.18.1]{MR2653248} Ng\^o defines isogenies between Prym varieties $\Pc_{G_1} \to \Pc_{G_2}$, where $G_1$ and $G_2$ are quasi-split reductive group $X$-schemes which correspond to isogenous root data (see \cite[D\'efinition 1.12.1]{MR2653248}). We recall his construction for the special case of $G$ and its Langlands dual $\widehat{G}$ (see \cite[Exemple 1.12.2]{MR2653248}). We refer the reader to \cite[Proposition 4.18.1]{MR2653248} for details on the construction below.

\begin{construction}\label{const:ngo_isogeny}
We denote by $J$ and $\widehat{J}$ the sheaves of regular centralisers associated to $G$ and $\widehat{G}$. There exists an isomorphism of neutral connected components $J^{\circ} \simeq \widehat{J}^{\circ}$ (this follows from \cite[Proposition 2.4.7]{MR2653248}, as the group schemes $J^1$ and $\widehat{J}^1$ are equivalent).

There exists a natural number $N$, for which the $N$-th power maps $[N]\colon J \to J$ and $[N]\colon \widehat{J} \to \widehat{J}$ factorise through the neutral connected component (simply take $N$ to be a number which is a multiple of the order of $\pi_0(J)$ and $\pi_0(\widehat{J})$). Using the aforementioned isomorphism $J^{\circ} \simeq \widehat{J}^{\circ}$ we obtain a map $\phi_N\colon J \to \widehat{J}$ which sits in a commutative diagram
\[
\xymatrix{
J \ar[r]^{[N]} \ar[d]_{\phi_N} & J^{\circ} \ar@{^(->}[d] \\
\widehat{J} & \widehat{J}^{\circ}. \ar@{_(->}[l]
}
\]
This gives rise to an isogeny  $\rho: \Pc_G \to \Pc_{\widehat{G}}$. Furthermore, if the base ring $R$ is a field $k$ of characteristic $p$, we can assume that $N$ is coprime to $p$. This follows from \cite[Corollaire 2.3.2]{MR2653248}: This statement implies that $|\pi_0(J)|$ is a divisor of $|Z(X,G)|$. The order of the group $Z(X,G)$ is coprime to $p$: Over $\rho$ we have a canonical isomorphism $G_\rho \cong \Gb_\rho$. Under this isomorphism the group $Z(X,G)_\rho$ embeds into $\Tb_\rho$. The latter being a torus, we deduce that the order of $Z(X,G)$ is coprime to $p$.

\end{construction}

%

\subsection{The anisotropic locus and $\widetilde{\Mc}$}\label{sub:ani}

We briefly recall the definition of the anisotropic locus and of the \'etale-open $\widetilde{\A} \to \A$. We adopt the terminology and conventions of Situation \ref{situation:Higgs_general}.

\begin{definition}\label{defi:ani}
Let $\bar{k}$ be an algebraically closed field, $R \to \bar{k}$ a homomorphism, and $a \in \A(\bar{k})$ a geometric point of the Hitchin base. We say that $a$ is \emph{anisotropic} if the group of connected components of the Prym variety $\pi_0(\Pc_a)$ is a finite group, and furthermore $a$ belongs to the open subset $\A^{\heartsuit}$ defined in \cite[4.5]{MR2653248} (amounting to a reduced cameral cover). According to \cite[6.1]{MR2653248} there exists an open subset $\A^{\ani}$, such that a geometric point of the Hitchin base is anisotropic, if and only if it belongs to $\A^{\ani}$. We denote the base change $\Mc \times_{\A} \A^{\ani}$ by $\Mc^{\ani}$. 
\end{definition}

Despite the ostensible dependance of $\A^{\ani}$ on the group $G$, this open subset is largely independent of the group $G$. It follows from \cite[4.10.5]{MR2653248} that $a \in \A^{\ani}$ is equivalent to $a \in \A^{\heartsuit}$ and finiteness of the group $\widehat{\Tb}^{W_a}$. This implies that $\A^{\ani}_{G_1} = \A^{\ani}_{G_2}$ if $G_1$ and $G_2$ have \emph{isogenous} root data (in \cite[1.12.4]{MR2653248} the groups $G_1$ and $G_2$ are said to be \emph{appari\'es}). In particular this applies to the case of Langlands dual reductive groups:

\begin{rmk}\label{rmk:ani}
With respect to the isomorphism $\A_G \simeq \A_{\widehat{G}}$ one has an identification of the open subsets $\A^{\ani}_G$ and $\A^{\ani}_{\widehat{G}}$.
\end{rmk}

It follows from \cite[Proposition 4.14.1]{MR2653248} that $\Mc^{\ani}$ is a smooth Deligne-Mumford stack if $\text{deg}(D) \geq 2g-2$.

\begin{rmk}
We remind the reader that our $\mathbb{M}_G$ is Ng\^o's $\Mc_G$, and our $\Mc_G$ refers to $\mathbb{M}_G$ rigidified at the generic stabiliser $Z(X,G)$. 
\end{rmk}

The aforementioned \'etale-open $\widetilde{\A}$ whose definition we recall below, depends on the choice of a base point $\infty\colon \Spec R \to X$.

\begin{definition}\label{defi:tilde}
We fix $\infty \in X(R)$ and denote by $\mathfrak{c}_{D,\infty}$ the pullback of $\mathfrak{c}_D=\mathbf{c} \times^{\G_m}_X D$ along $\infty$, and by $\ev_{\infty}\colon \A \to \mathfrak{c}_{D,\infty}$ the associated evaluation map. The \'etale morphism $\widetilde{\A} \to \A$ is defined by the cartesian square
\[
\xymatrix{
\widetilde{\A} \ar[r] \ar[d] & \A \ar[d]^{\ev_{\infty}} \\
\mathfrak{t}_{D,\infty}^{rs}\ar[r] & \mathfrak{c}_{D,\infty}.
}
\]
\end{definition}

Our main result is concerned with the cohomology of the Deligne-Mumford stack $\widetilde{\Mc}_G^{\ani}$ which we define next. The definition is taken from \cite[5.3.1]{MR2653248}.

\begin{definition} \label{TildeDef}
We denote by $\widetilde{\Mc}_G$ the base change $\Mc_G \times_{\A} \widetilde{\A}$, and by $\widetilde{\Mc}_G^{\ani}$ the base change $\Mc_G^{\ani} \times_{\A} \widetilde{\A}$.
\end{definition}

The following construction is a variant of  \cite[Corollaire 4.11.3]{MR2653248}.

\begin{construction}\label{cons:auto_torus}
Let $S$ be an affine scheme over $\Spec(R)$ and $(E,\phi,\tilde\infty) \in \widetilde{\mathbb{M}}^{\ani}_G(S)$. In the following we construct a canonical injective homomorphism $\Aut(E,\theta) \to T_{\infty}(S)$. Often we will use a chosen point $\infty_\rho \in \rho_\infty(R)$ to identify $G_\infty$ with $\BG$ and $T_\infty$ with $\BT$. In such a situation we will consider this homomorphism as a homomorphism $\Aut(E,\phi) \to \BT$.

We have a Higgs bundle $(E,\phi)$ over $S$ and a morphism $\tilde \infty\colon S \to \mathfrak{t}^\text{rs}_{D,\infty}$ making the following diagram commute:
\begin{equation} \label{SomeDiag}
  \xymatrix{
    X \times S \ar[r] & [\mathfrak{g}_D/G] \ar[r] & \mathfrak{g}_D/G \\
    S \ar[u]^\infty \ar[r]^{\tilde\infty} \ar[r] & \mathfrak{t}^\text{rs}_{D,\infty} \ar[ru] 
  }
\end{equation}
Consider an automorphism $\gamma$ of $(E,\phi)$. By \cite[4.11.4]{MR2653248}, since our Higgs bundle lies over $\A^{\ani}$, the automorphism $\gamma$ has finite order $r$ coprime to $p$. 

In the above picture, the automorphism $\gamma$ is is given by a global section of the pullback to $X \times S$ of the group scheme $I$ over $[\mathfrak{g}_D/G]$ from \cite[2.1]{MR2653248}. Over the regular semisimple locus, this group scheme can be identified with the pullback from $\mathfrak{g}_D/G$ of the group scheme $J$ from \emph{loc.cit.} Finally, the pullback of $J$ along $\mathfrak{t}^\text{rs}_\infty \to \mathfrak{g}_D/G$ can be canonically identified with $T_\infty$. Hence we see that the pullback of $\gamma$ along $\infty\colon S \to X \times S$ is given by a section $\gamma_\infty \in T_\infty[r](S)$. This defines the desired homomorphism. Finally the fact that $\gamma$ is of finite order implies that it is uniquely determined by its pullback along $\infty$.
\end{construction}

\begin{construction}\label{cons:AJ}
Recall from Construction \ref{const:AJ} that we have an Abel-Jacobi map $\AJ_G\colon \Cc^{\sm} \times_X \Xb_*(T) \to \Pc_G$. 

The space $\widetilde \A$ comes with a morphism $\widetilde \A \to \mathfrak{t}^\text{rs}_{D,\infty}$. Together with the morphism $\widetilde \A \to \A \toover{\infty} X \times \A$ this induces a canonical morphism $\widetilde{\A} \to \Cc^{\sm}$. 

Now we assume in addition that we have a point $\infty_\rho \in \rho(R)$ above $\infty$ and use this point to identify $T_\infty$ with $\Tb$. Then we get a morphism
\begin{equation*}
   \widetilde{\A} \times \Xb_*(\Tb) =\widetilde{\A} \times \Xb_*(T_\infty) \to \Cc^{\sm} \times_X \Xb_*(T) \to  \Pc_G
\end{equation*}
and hence a morphism
\begin{equation}\label{ajconst}
   \widetilde{\A} \times \Xb_*(\Tb) \to  \widetilde{\A} \times_{\A} \Pc_G =\widetilde{\Pc}_G.
\end{equation}
\end{construction}

The following result is a combination of \cite[Lemme 4.10.2]{MR2653248} and \cite[Corollaire 6.7]{MR2218781}. 

\begin{proposition}[Ng\^o]\label{prop:surjection}
The induced map $\Xb_*(\Tb) \times \widetilde{\A}^{\ani} \to \pi_0(\widetilde{\Pc}_G^{\ani})$ is a surjection. In particular for every $a \in \widetilde{\A}^{\ani}(k)$ the finite group scheme $\pi_0(\widetilde{\Pc}_{G,a}^{\ani})$ is constant.
\end{proposition}

We remark that Ng\^o's construction of the map $\Xb_*(\Tb) \times \widetilde{\A}^{\ani} \to \pi_0(\widetilde{\Pc}_G^{\ani})$ is explained in the proof of \cite[Proposition 6.8]{MR2218781}. Up to replacing $\infty$ by $u_S$ in \emph{loc. cit.} the two constructions agree.

\section{Endoscopy and inertia stacks of moduli spaces of Higgs bundles}\label{secinst}
Throughout this section we put ourselves in Situation \ref{situation:Higgs_general}, and assume additionally that the base ring $R$ is either  a finite field $k$ or the algebraic closure of a finite field $\bar{k}$. 

\subsection{Quotients by twisted group actions}

We will need the following description of the quotient stacks given by twisted group actions:
\begin{construction} \label{TwistedQuotient}
  Consider a base algebraic space $S$, smooth group algebraic spaces $A$ and $\BB$ over $S$ as well as an algebraic space $\BX$ over $S$ with actions of $A$ and $\BB$ over $S$. Let furthermore $A$ act on $\BB$ through group automorphisms in such a way that the action $\BB \times_S \BX \to \BX$ is $A$-equivariant. For an $A$-torsor $\rho$ over $S$, by twisting we obtain an algebraic space $X\defeq \BX \times^A \rho$ and a group algebraic space $B \defeq \BB \times^A \rho$ over $S$ with an action $B \times_S X \to X$. 

By definition, $X$ is the quotient of $\BX \times \rho$ by the antidiagonal action of $A$ and the action of $B$ on $X$ is induced by the action of $\BB$ on the first factor of $\BX \times \rho$. These actions of $A$ and $\BB$ on $\BX \times \rho$ fit together to an action of $\BB \rtimes A$ on $\BX \times \rho$. Then the quotient morphism $\BX \times \rho \to X$ induces an isomorphism of quotient stacks
\begin{equation*}
  [(\BX \times \rho)/ (\BB \rtimes A)] \cong [X/B].
\end{equation*}
\end{construction}

%
%
%

\subsection{Construction of certain homomorphisms}



For a homomorphism $\kappa\colon \hat \mu \to \BT$ over $\Spec(k)$ we consider the centraliser $(\BG \ltimes \Out(\BG))_\kappa$ of $\kappa$ in $\BG \ltimes \Out(\BG)$ and write its connected-\'etale sequence as
\begin{equation*}
  1 \to \BH_\kappa \to (\BG \rtimes \Out(\BG))_\kappa \to \pi_0(\kappa) \to 1,
\end{equation*}
where $\BH_\kappa$ is the connected centraliser of $\kappa$ in $\BG$. By Lemma \ref{CompPinn} there exists a unique pinning $(T,B\cap \BH_\kappa,s')$ of $\BH_\kappa$ compatible with the pinning of $\BG$. We will always endow $\BH_\kappa$ with this pinning. The group scheme $\pi_0(\kappa)$ comes with natural homomorphisms $o_\BG\colon \pi_0(\kappa)\to \Out(\BG)$ and $o_\BH\colon \pi_0(\kappa) \to \Out(\BH_\kappa)$. We let $\pi_0(\kappa)$ act on $\BH_\kappa$ and $\BG$ through these homomorphisms.


\begin{construction} \label{ikCons}
  Let $\kappa \colon \hat \mu \to  \BT$ be a homomorphism over $\Spec(k)$. We construct a canonical homomorphism $\iota_\kappa\colon \BH_\kappa \rtimes \pi_0(\kappa)  \to \BG \rtimes \pi_0(\kappa)$ satisfying the following conditions:
  \begin{enumerate}[(i)]
  \item The restriction of $\iota_\kappa$ to the normal subgroup $\BH_\kappa$ is the inclusion of $\BH_\kappa$ into the normal subgroup $\BG$.
  \item The morphism between the quotients $\pi_0(\kappa)$ induced by $\iota_\kappa$ is the identity. 
  \item If we let $\pi_0(\kappa) \ltimes \BH_\kappa$ act on $\BH_\kappa$ via the above action of $\pi_0(\kappa)$ on $\BH_\kappa$ and the action of $\BH_\kappa$ on itself via conjugation, and analogously for $\BG$, the inclusion $\BH_\kappa \into \BG$ is equivariant with respect to $\iota_\kappa$.
  \end{enumerate}

Let $\tilde \pi_0(\kappa) \subset (\BG \rtimes \Out(\BG))_\kappa$ be the subgroup sheaf consisting of those elements whose action on $\BH_\kappa$ preserves the pinning of $\BH_\kappa$. Since the stabiliser in $\BH_\kappa$ of the pinning is $Z(\BH_\kappa)$, the subgroup $\tilde\pi_0(\kappa)$ fits into an exact sequence 
\begin{equation*}
  1 \to Z(\BH_\kappa) \to \tilde \pi_0(\kappa) \to \pi_0(\kappa) \to 1.
\end{equation*}
We denote again by $o_\BH$ the composition $\tilde\pi_0(\kappa) \to \pi_0(\kappa) \toover{o_\BH} \Out(\BH_\kappa)$ and similarly for $o_\BG$. As for $\pi_0(\kappa)$ we let $\tilde\pi_0(\kappa)$ act on $\BH_\kappa$ and $\BG$ through these homomorphisms.

We first construct a morphism $\tilde\iota_\kappa\colon \BH \rtimes \tilde\pi_0(\kappa) \to \BG \rtimes \tilde\pi_0(\kappa)$: Let $\phi \in \tilde\pi_0(\kappa)$. Then $\phi=g_\phi o_G(\phi)$ for some $g_\phi \in \BG$. From the definition of $\tilde\pi_0(\kappa)$ it follows that the automorphism $o_\BH(\phi)$ of $\BH_\kappa$ is equal to the restriction of $\Inn(g_\phi)\circ o_\BG(\phi)$ to $\BH_\kappa$. Using this one checks that 
\begin{align*}
\tilde\iota_\kappa\colon \BH_\kappa \rtimes \tilde\pi_0(\kappa) &\to \BG \rtimes \tilde\pi_0(\kappa),\\ h \phi &\mapsto h g_\phi \phi  
\end{align*}
 is a group homomorphism.

One checks that $\tilde\iota_\kappa$ maps $\BH_\kappa \rtimes Z(\BH_\kappa)$ to $\BG \rtimes Z(\BH_\kappa)$ and hence induces a homomorphism $$\iota_\kappa\colon \BH_\kappa \rtimes \pi_0(\kappa)  \to \BG \rtimes \pi_0(\kappa).$$ By a direct verification this homomorphism has the required properties.
\end{construction}

\begin{proposition} \label{ikProps}
  In the situation of Construction \ref{ikCons}, the homomorphism $\iota_\kappa$ factors through $(\BG\rtimes \pi_0(\kappa))_\kappa$. By composition with $(\BG\rtimes \pi_0(\kappa))_\kappa \to (\BG \rtimes \Out(\BG))_\kappa$ it gives an isomorphism $\BH_\kappa \rtimes \pi_0(\kappa) \to (\BG \rtimes \Out(\BG))_\kappa$ which restricted to $\pi_0(\kappa)$ gives a splitting of the quotient homomorphism $(\BG \rtimes \Out(\BG))_\kappa \to \pi_0(\kappa)$. 
\end{proposition}
\begin{proof}
  The first claim follows directly from the construction of $\iota_\kappa$.

Since $(\BG \rtimes \Out(\BG))_\kappa$ lies in an exact sequence
\begin{equation*}
  1 \to \BH_\kappa \to (\BG \rtimes \Out(\BG))_\kappa \to \pi_0(\kappa) \to 1,
\end{equation*}
to prove the remaining claim it suffices to show that the homomorphism $$\BH_\kappa \rtimes \pi_0(\kappa) \to (\BG\rtimes \Out(\BG))_\kappa$$ restricts to the identity on $\BH_\kappa$ and induces the identity on the quotients $\pi_0(\kappa)$. The claim for $\BH_\kappa$ follows directly from the construction of $\iota_\kappa$. To verify the claim for $\pi_0(\kappa)$ consider an element $\bar \phi \in \pi_0(\kappa)$. Then for a lift $\phi \in \tilde \pi_0(\kappa)$, the identity $\phi=g_\phi o_\BG(\phi)$ from Construction \ref{ikCons} shows that the image of $\iota_\kappa(\phi)$ under the projection $(\BG\rtimes \pi_0(\kappa))_\kappa \to (\BG \rtimes \Out(\BG))_\kappa \to \pi_0(\kappa)$ is equal to $\phi$.
\end{proof}

\subsection{The twisted inertia stack}

\begin{definition} \label{DEDef1}
A coendoscopic datum for $G$ over $k$ is a triple $\CE=(\kappa,\rho_\kappa,\rho_\kappa \to \rho)$ consisting of a homomorphism $\kappa \colon \hat\mu \to \BT$ over $\Spec(k)$, a $\pi_0(\kappa)$-torsor $\rho_\kappa$ over $X$ and a $o_\BG$-equivariant morphism $\rho_\kappa \to \rho$ (recall that $\rho$ is the $\Out(\BG)$-torsor which twists $\BG$ to $G$). We call $\kappa$ the \emph{type} of $\CE$.

For such a datum we let $H_\CE \defeq \BH_\kappa \times^{\pi_0(\kappa)} \rho_\kappa$, where $\pi_0(\kappa)$ acts on $\BH_\kappa$ through $o_\BH$. Equivalently, the group $H_\CE$ is obtained by twisting $\BH_\kappa$ with the $\Out(\BH_\kappa)$-torsor $\rho_\kappa \times^{\pi_0(\kappa)} \Out(\BH_\kappa)$ and hence is a quasi-split group scheme. We will denote the various objects from Section \ref{HiggsBundles} associated to $H_\CE$ with a subscript $\CE$.

 In particular, the maximal torus $\BT \subset \BH_\kappa$ induces by twisting a maximal torus $T_\CE \defeq \BT \times^{\pi_0(\kappa)} \rho_\kappa \subset H_\CE$, where as before $\pi_0(\kappa)$ acts on $\BT$ through $o_\BH$. Similarly the homomorphism $\kappa\colon \hat \mu \to \BT$ induces by twisting a homomorphism $\gamma_\CE \colon \hat\mu \to Z(H_\CE)$ over $X$. We denote by $\mathfrak{h}_\CE$ be the Lie algebra of $H_\CE$.
\end{definition}

\begin{rmk}
The Langlands dual $\widehat{H}_{\CE}$ is called an \emph{endoscopy group} for $\widehat{G}$.
\end{rmk}

\begin{construction} \label{muCons}
  Let $\CE=(\kappa,\rho_\kappa,\rho_\kappa \to \rho)$ be a coendoscopic datum for $G$ over $k$. We construct a canonical morphism 
  \begin{equation*}
    \mu_\CE \colon [\mathfrak{h}_{\CE,D}/H_\CE] \to [\mathfrak{g}_D/G]
  \end{equation*}
as follows: Construction \ref{TwistedQuotient} gives canonical isomorphisms
\begin{align*}
  [\mathfrak{h}_{\CE,D}/H_\CE] & \cong [(\mathbf{h}_\CE \times \rho_\kappa \times D) / (\BH_\kappa \rtimes (\pi_0(\kappa) \times \BG_m))] \\
\text{ and }  [\mathfrak{g}_D /G] & \cong [(\mathbf{g} \times \rho_\kappa \times D) / (\BG \rtimes (\pi_0(\kappa) \times \BG_m))].
\end{align*}
Now the product of the homomorphism $\iota_\kappa$ from Construction \ref{ikCons} and the identity on $\BG_m$ with respect to which the inclusion $\mathbf{h}_\CE \times \rho_\kappa \times D \into \mathbf{g} \times \rho_\kappa \times D$ is equivariant induces the desired morphism.

By passing to coarse moduli spaces we obtain a morphism 
\begin{equation*}
  \nu_\CE \colon \mathfrak{c}_\CE=\mathfrak{h}_{\CE,D}/H_\CE \to \mathfrak{c}=\mathfrak{g}_D/G.
\end{equation*}
One can check that this morphism coincides with the one constructed in \cite[1.9]{MR2653248}. In fact the construction we have given here is a direct extension of the construction given in \emph{loc. cit.}

Composition with $\nu_\CE$ defines a morphism $\nu_\CE\colon \A_{\CE} \to \A$. Similarly, composition with $\mu_\CE$ defines a compatible morphism $\BM_{H_\CE} \to \BM_G$. Since the homomorphism $\gamma_\CE$ maps to the centre of $H_\CE$, for every $H_\CE$-Higgs bundle $F$ it naturally induces a homomorphism $\hat \mu \to \Aut(F)$. If $E$ is the image in $\BM_G$ of such an $F$, we get a homomorphism $\hat\mu \to \Aut(E)$ by functoriality. Using these homomorphisms we lift the morphism $\BM_{H_\CE} \to \BM_G$ to a morphism $\mu_\CE\colon \BM_{H_\CE} \to I_{\hat\mu} \BM_G$.

\end{construction}
\begin{lemma}\label{AniComp}
  Let $\CE$ be a coendoscopic datum for $G$ over $k$. The morphism $\nu_\CE \colon \A_\CE \to \A$ satisfies $\nu_\CE^{-1}(\A^{\ani})=\A_\CE^{\ani}$.

\end{lemma}
\begin{proof}
By Remark \ref{rmk:ani} the anisotropic locus $\A^{\ani}$ with respect to $H_\CE$ coincides with the one with respect to the endoscopic group $\widehat H_\CE$ for $\widehat G$. So we may consider the latter locus. We denote the associated abstract Prym by $\Pc_{\widehat H_\CE}$.

 Since the anisotropic loci are open subschemes of the Hitchin bases it suffices to consider $\bar k$-valued points. So let $a_{\CE} \in \A^{\ani}_{\CE}(\bar k)$ with image $a$ in $\A(\bar k)$. The point $a_{\CE}$ lies in $\A^{\ani}_{\CE}$ if and only if the group $\pi_0(\Pc_{\widehat H_\CE,a_\CE})$ is finite, and analogously for $a$ and the group $\pi_0(\mathcal{P}_a)$. So to prove the claim it suffices show that $\pi_0(\Pc_{\widehat H_\CE,a_\CE})$ is finite if and only if $\pi_0(\mathcal{P}_a)$ is finite.

 By \cite[4.17.2]{MR2653248} there exists a canonical surjective homomorphism $\mathcal{P}_a \to \Pc_{\widehat H_\CE,a_\CE}$ whose kernel is an affine group scheme of finite type over $\Spec(k)$. This morphism induces a surjective homomorphism $\pi_0(\mathcal{P}_a) \to \pi_0(\mathcal{P}_{\widehat H_\CE,a_\CE})$ whose kernel is finite. This implies the claim.
\end{proof}

Now we fix a point $\infty \in X(k)$ and consider the space $\widetilde{\BM}_G$ given by Definition \ref{TildeDef}. 
\begin{lemma} \label{BlaComp}
  Let $\CE=(\kappa,\rho_\kappa,\rho_\kappa \to \rho)$ be a coendoscopic datum for $G$ over $k$ and $\infty_{\rho_\kappa}$ be a point in $\rho_\kappa(k)$ above $\infty$. The point $\infty_{\rho_\kappa}$ induces identifications $G_\infty \cong \BG$ and $H_{\CE,\infty} \cong \BH_\kappa$ compatible with the pinnings on the groups. The induced isomorphism $\mathfrak{t}_{\CE,\infty} \cong \mathbf{t} \cong \mathfrak{t}_\infty$ makes the following diagram commute:
  \begin{equation*}
    \xymatrix{
      \mathfrak{t}_{\CE,D,\infty} \ar[r]^\cong \ar[d] & \mathfrak{t}_{D,\infty} \ar[d] \\
      \mathfrak{c}_{\CE,\infty} \ar[r]^{\nu_\CE} & \mathfrak{c}_\infty
}
  \end{equation*}
\end{lemma}
\begin{proof}
  This follows by a direct verification from the construction of $\nu_\CE$.
\end{proof}

\begin{construction} \label{TildeMuCons}
  Let $\CE=(\kappa,\rho_\kappa,\rho_\kappa \to \rho)$ be a coendoscopic datum for $\BG$ over $k$ and let $\infty_{\rho_\kappa}$ be a point in $\rho_\kappa(k)$ above $\infty$. 

Inside $\A$ we have the open subset $\A^\infty$ of those elements which are regular semisimple at $\infty$. Then, as in Definition \ref{defi:tilde}, the space $\tilde \A$ given by the following pullback diagram:
\begin{equation*}
  \xymatrix{
    \widetilde \A \ar[r] \ar[d] & \mathfrak{t}^{rs}_{D,\infty} \ar[d] \\
    \A^\infty \ar[r] & \mathfrak{c}_{D,\infty}
  }
\end{equation*}

By Construction \ref{muCons} we have a canonical morphism $\nu_\CE \colon \A_{\CE} \to \A$. Let $\A_\CE^{G-\infty}$ be the preimage of $\A^\infty$ in $\A_\CE$ and $\widetilde \A_\CE^{G-\infty}$ the preimage of $\A_\CE^{G-\infty}$ in $\widetilde \A_\CE$. 

As in Lemma \ref{BlaComp}, the chosen point $\infty_{\rho_\kappa}$ induces an isomorphism $\mathfrak{t}_{\CE,\infty} \cong \mathfrak{t}_\infty$ compatible with $\nu_\CE$. This isomorphism induces a morphism $\tilde \nu_\CE \colon \widetilde \A^{G-\infty}_\CE \to \widetilde \A$ compatible with $\nu_\CE \colon \A_\CE \to \A$. 

We let $\widetilde{\A}^{\ani}$ be the preimage of $\A^{\ani} \cap \A^\infty$ in $\widetilde\A$ and $\widetilde\A_\CE^{G-\infty,\ani}$ the preimage of $\A_\CE^{G-\infty} \cap \A_\CE^{\ani}$ in $\widetilde \A_\CE^{G-\infty}$. For the various exponents $*$ occurring above, we will denote by $\widetilde \BM^* \to \widetilde \A$ (resp. $\widetilde \BM^*_{H_\CE} \to \widetilde \A^*_{\CE}$) the pullback of $\BM \to \A$ (resp. $\BM_{H_\CE} \to \A_{H_\CE}$) to $\widetilde \A^*$ (resp. $\widetilde \A^*_{\CE}$). Then using Lemma \ref{AniComp} by pulling back the morphism $\mu_\CE$ we obtain the following commutative diagram:
\begin{equation*}
\xymatrix{
  \widetilde\BM^{G-\infty,\ani}_{H_\CE} \ar[r]^{\tilde \mu_\CE} \ar[d] & I_{\hat\mu} \widetilde \BM^{\ani}_G\ar[d] \\
\widetilde\A^{G-\infty,\ani}_{\CE} \ar[r]  & \widetilde \A^{\infty,\ani}  
}
\end{equation*}
\end{construction}
In Theorem \ref{IMTheorem} below, we will show that the morphisms $\tilde\mu_\CE$ just constructed for varying $\CE$ completely describe the $k$-valued points of $I_{\hat\mu}\widetilde{\BM}^{\ani}_G$ (at least for $k$ sufficiently big). The proof of this result will rest on the following construction, which associates to every section of $I_{\hat\mu} \widetilde{\BM}^{\ani}_G$ a certain reduction $\tilde F$.
\begin{construction}  \label{ReductionCons} 
Consider a point $\infty_\rho \in \rho_\infty(k)$, a homomorphism $\kappa \colon \hat\mu \to \BT$ over $\Spec(k)$ and a scheme $S$ over $\Spec(k)$. Consider an object $(E,\theta,\tilde\infty)$ of $\widetilde\BM^{\infty,\ani}(S)$ together with a homomorphism $\gamma \colon \hat\mu \to \Aut(E,\theta)$ whose image under the homomorphism $\Aut(E,\theta) \to \BT$ from Construction \ref{cons:auto_torus} is equal to $\kappa$. 


Under the description of $[\mathfrak{g}_D/G]$ given by Construction \ref{TwistedQuotient}, the Higgs bundle $(E,\theta)$ corresponds to a $\BG\rtimes (\pi_0(\kappa) \times \BG_m)$-torsor $\tilde E$ over $X \times S$ together with an equivariant morphism $\tilde\theta\colon \tilde E \to \mathbf{g} \times \rho_\kappa\times D$. In this picture the automorphism $\gamma$ corresponds to an automorphism of the torsor $\tilde E$ which stabilises $\tilde \theta$. 

Fix an \'etale covering $X'$ of $X \times S$ over which $\tilde E$ has a section. By also fixing such a section we can identify $\gamma$ with a homomorphism $\hat\mu_{X'} \to (\BG\rtimes (\pi_0(\kappa) \times \BG_m))_{X'}$. Changing the section amounts to conjugating $\gamma$. Since the image of $\gamma$ stabilizes $\tilde \theta$ and since the group $\pi_0(\kappa)\times \BG_m$ acts without fixed points on $\mathbf{g} \times \rho_\kappa \times D$, the homomorphism $\gamma$ factors through $\BG \subset \BG\rtimes (\pi_0(\kappa) \times \BG_m)$. First we show:
  \begin{lemma}
    The homomorphisms $\gamma,\kappa \colon \hat\mu_{X'} \to \BG_{X'}$ are conjugate \'etale-locally over $X'$.
  \end{lemma}
  \begin{proof}  
 First we claim that \'etale-locally on $X'$, the homomorphism $\gamma$ factors through maximal torus of $\BG$: Consider the centraliser $C_\BG(\gamma)$ of $\gamma$ in $\BG_{X'}$. For example by \cite[Lemma 2.2.4]{MR3362641} this is representable by a smooth group scheme of $\BG_{X'}$, and there exists an open and closed subgroup scheme $C_\BG(\gamma)^\circ$ which fibrewise is given by the identity component. Then $C_\BG(\gamma)^\circ$ is a reductive subgroup scheme of $\BG_{X'}$ through which $\gamma$ factors, as this holds in every geometric fibre. Then any maximal torus of $C_\BG(\gamma)^\circ$, which exists after replacing $X'$ by a suitable covering, is a maximal torus of $\BG_{X'}$ through which $\gamma$ factors.

Hence, by the \'etale-local conjugacy of maximal tori of $\BG$, after enlarging $X'$ and conjugating we may assume that $\gamma$ factors through $\BT_{X'}$. 

Let $r\geq 1$ be such that both $\gamma$ and $\kappa$ factor through a homomorphism $\mu_r \to \BT_{X'}[r]$. After suitably enlarging $X'$ we may assume that $\BT_{X'}[r]$ is a constant group scheme.


 By construction, the homomorphisms $\gamma$ and $\kappa$ are conjugate above $\infty$. Let $X'_1$ be a connected component of $X'$ whose image in $X\times S$ intersects $\infty$. Since $X'_1$ is connected and $\BT[r]$ is constant on $X'$, the fact that the sections $\gamma$ and $\kappa$ of $\BT[r](X'_1)$ are conjugate in some point of $X'_1$ implies that they are conjugate over all of $X'_1$. 

Now let $X'_2$ be any other connected component of $X'$ and let $X'_{12}\defeq X'_1 \times_{X\times S} X'_2$ which has open dense image in both $X'_1$ and $X'_2$. By the construction of $\gamma$, the images of $\gamma|_{X'_1}$ and $\gamma|_{X'_2}$ in $\BT[r](X'_{12})$ are \'etale-locally conjugate. Hence the conjugacy of $\gamma$ and $\kappa$ over $X'_1$ implies that that $\gamma$ and $\kappa$ are conjugate over some geometric point of $X'_2$. Hence they are conjugate over $X'_2$ by the same argument as above. By varying $X'_2$ over all of $X'$ this proves the lemma.
  \end{proof}

Now let $\tilde F \subset \tilde E$ be the subsheaf consisting of those sections of $\tilde E$ on which the actions of $\hat\mu$ through $\gamma$ and $\kappa$ coincide. By the Lemma just proved this is a $(\BG \rtimes \Out(\BG))_\kappa \times \BG_m$-subtorsor of $\tilde E$. Let $\rho_\kappa$ be the $\pi_0(\kappa)$-torsor $\pi_0(F)$ over $X \times S$. The composition 
\begin{equation*}
  \tilde F \to \tilde E \to \mathbf{g} \times \rho \times D \to \rho
\end{equation*}
induces an $o_\BG$-equivariant morphism $\rho_\kappa \to \rho$ over $X \times S$. 

 Next we consider the fibre product
\begin{equation*}
  \xymatrix{
    \hat E \ar[r] \ar[d] & \tilde E \ar[d] \\
    \mathbf{g} \times \rho_\kappa \times D \ar[r] & \mathbf{g} \times \rho \times D.
}
\end{equation*}
This diagram is naturally equivariant under the fibre product
\begin{equation*}
  \xymatrix{
    \BG \rtimes (\pi_0(\kappa) \times \BG_m) \ar[r] \ar[d] & \BG \rtimes (\Out(\BG) \times \BG_m) \ar[d]\\
    \BG \rtimes (\pi_0(\kappa) \times \BG_m) \ar[r] & \BG \rtimes (\Out(\BG) \times \BG_m) ,
  }
\end{equation*}
which makes $\hat E$ into a $\BG \rtimes (\pi_0(\kappa) \times \BG_m)$-torsor. The quotient morphism $\tilde F \to \rho_\kappa$ induces a closed embedding $\tilde F \into \hat E$. Using Proposition \ref{ikProps} one sees that via this embedding $\tilde F$ is a $\BH_\kappa \rtimes (\pi_0(\kappa) \times \BG_m)$-subtorsor of $\hat E$, where this group is embedded into $\BG \rtimes (\pi_0(\kappa \times \BG_m)$ via $\iota_\kappa \times \id_{\BG_m}$.

By the definition of $\tilde F$, on sections of $\tilde F$ the actions on $\tilde E$ of $\hat\mu$ through $\gamma$ and $\kappa$ coincide. Since the image of $\gamma$ fixes $\tilde\theta$, it follows that the composition
\begin{equation*}
  \tilde F \to \tilde E \to \mathbf{g}\times \rho \times D \to \mathbf{g}
\end{equation*}
factors through the fixed points of $\kappa$ on $\mathbf{g}$, that is through $\mathbf{h}_\kappa$. Thus we obtain a $\BH_\kappa \rtimes (\pi_0(\kappa) \times \BG_m)$-equivariant morphism $\tilde F \to \mathbf{h}_\kappa \times \rho_\kappa \times D$. 

\end{construction}

The following follows by combining Constructions \ref{ReductionCons} and \ref{muCons}: 
\begin{proposition} \label{kCase}
  In case $S=\Spec(k)$, the datum $(\kappa,\rho_\kappa,\rho_\kappa \to \rho)$ from Construction \ref{ReductionCons} is a coendoscopic datum $\CE$ for $G$ over $k$ and the sheaf $\tilde F$ corresponds under Construction \ref{TwistedQuotient} to a $H_\CE$-Higgs bundle $F$ over $\Spec(k)$ which maps to the point $((E,\theta),\gamma)$ of $I\BM(k)$ under $\mu_\CE$. 
\end{proposition}

\begin{definition}
  A coendoscopic datum $\CE$ for $G$ over $k$ occurs on $I_{\hat\mu} \widetilde\BM^{\ani}_G(k)$ if there exists a point on $I_{\hat\mu} \widetilde\BM^{\ani}_G(k)$ for which the associated coendoscopic datum from Proposition \ref{kCase} is equal to $\CE$.
\end{definition}

\begin{theorem} \label{IMTheorem}
  For a point $\infty_\rho \in \rho_\infty(k)$ and points $\infty_{\rho_\kappa} \in \rho_{\kappa,\infty}(k)$ mapping to $\infty_\rho$ for every coendoscopic datum $\CE$ for $G$ over $k$ occurring on $I_{\hat\mu}\widetilde\BM^{\ani}_G(k)$, the resulting morphism of groupoids
\begin{equation*}
  \bigsqcup_{\CE}\tilde \mu_\CE \colon  \bigsqcup_{\CE} \widetilde{\BM}^{G-\infty,\ani}_{H_\CE}(k) \to I_{\hat\mu}\widetilde{\BM}^{\ani}_G(k)
\end{equation*}
is an equivalence.
\end{theorem}

\begin{rmk} \label{SplitingRmk}
  Using Construction \ref{cons:auto_torus}, Construction \ref{ReductionCons} and the fact that $I_{\hat\mu}\widetilde{\BM}^{\ani}_G$ is of finite type over $\Spec(k)$ one sees that there exist only finitely many coendoscopic data for $G$ over $\bar k$ which occur on $I_{\hat\mu}\widetilde\BM^{\ani}_{G,\bar k}$. Hence in case $k$ is finite one may always choose points $\infty_\rho$ and $\infty_{\rho_\kappa}$ as in Theorem \ref{IMTheorem} after possibly replacing $k$ by a finite extension.
\end{rmk}
\begin{proof}[Proof of Theorem \ref{IMTheorem}]

We construct an inverse morphism: Consider a point $((E,\theta),\tilde \infty,\gamma)$ of $I_{\hat\mu}\widetilde\BM^{\ani}_G(k)$. By composing $\gamma \colon \hat\mu \to \Aut(E,\theta)$ with the homomorphism $\Aut(E,\theta) \to \BT$ from Construction \ref{cons:auto_torus} we obtain a homomorphism $\kappa \colon \hat\mu \to \BT$ over $\Spec(k)$. Then Construction \ref{ReductionCons} gives us a coendoscopic datum $\CE$ for $\BG$ over $k$ and $H_\CE$-Higgs bundle $F$ mapping to $((E,\Theta),\gamma)$ under $\mu_\CE$. By Lemma \ref{AniComp}, the $H_\CE$-Higgs bundle $F$ lies above $\A_\CE^{\ani}$. 

It remains to lift $F$ to $\widetilde{\BM}^{G-\infty,\ani}_{H_\CE}$. Consider the morphism $\tilde \infty_\kappa\colon \Spec(k) \to \mathfrak{t}_{\CE,\infty}$ obtained by composing $\tilde\infty\colon \Spec(k) \to \mathfrak{t}_{D,\infty}$ with the isomorphism $\mathfrak{t}_{D,\infty} \cong \mathbf{t}_{D,\infty} \cong \mathfrak{t}_{\CE,D,\infty}$ as in Lemma \ref{BlaComp}. We claim that endowing $F$ with the morphism $\tilde \infty_\kappa$ gives a preimage $\underline F$ of $\underline E$ in $\widetilde{\BM}^{G-\infty,\ani}_{H_\CE}$. That is we claim the following:

\begin{lemma}
The left part of the following diagram commutes:
\begin{equation*}
  \xymatrix{
    X \ar[r] & [\mathfrak{h}_D/H] \ar[r]^{\mu_\CE} \ar[d] & [\mathfrak{g}_D/G]  \ar[d]\\
    & \mathfrak{h}_D/H \ar[r]^{\nu_\CE} & \mathfrak{g}_D/G \\
    \Spec(k)  \ar[uu]^\infty \ar[r]^{\tilde\infty_\kappa} & \mathfrak{t}_{\CE,D,\infty}^{\text{rs}} \ar[ru] \ar[u] &
}
\end{equation*}
\end{lemma}
\begin{proof}
  Let $x\in \tilde F(\bar k) \subset \tilde E(\bar k)$ and $\bar x$ its image in $\mathbf{h}_\kappa \subset \mathbf{g}$. By the commutativity of \eqref{SomeDiag} the image of $\tilde\infty_\kappa$ under the isomorphism $\mathfrak{t}_{\CE,\infty}\cong \mathbf{t}$ given by $\infty_{\rho_\kappa}$ and the point $\bar x$ of $\mathbf{g}_{\bar k}$ are conjugate by an element $g$ of $\BG(\bar k)$. Furthermore the construction of $\kappa$ shows that such an element $g$ must stabilise $\kappa$ so that $g \in (\BG \rtimes \Out(\BG))_\kappa$. This shows that the images of $x$ and $\tilde\infty_\kappa$ in $\mathbf{h}_{\CE,D}/H_\CE$ coincide, which is what we wanted.
\end{proof}
So we have constructed a preimage of $\underline E$ under $\bigsqcup_{\CE}\tilde \mu_\CE$. One checks that this construction is functorial. Then to show that $\bigsqcup_{\CE}\tilde \mu_\CE$ gives an equivalence on $k$-valued points one checks that if one starts with a coendoscopic datum $\CE$ for $\BG$ over $k$ and a point $\underline F$ of $\widetilde \BM^{G-\infty,\ani}_{H_\CE}(k)$ and applies the above construction to $\tilde\mu_\CE(\underline F)$, one gets back the same endoscopic datum $\CE$ and the same point $\underline F$.
\end{proof}

For a coendoscopic datum $\CE = (\kappa,\rho_{\kappa},\rho_{\kappa} \to \rho)$, the condition that the torsor $(\rho_{\kappa})_{\infty} \to \Spec k$ splits, is of arithmetic nature. It is always satisfied over an algebraic closure $\bar{k}$, respectively after replacing $k$ by an appropriate finite field extension. It is therefore possible that for a non-algebraically closed field $k$, this splitting exists only for some coendoscopic data $\CE$. For this reason we record the following variant of the result above. It will be handy to introduce the following notation:
$$\left(I_{\hat\mu}\widetilde{\BM}^{\ani}_G(k)\right)_{\kappa} = I_{\hat\mu}\widetilde{\BM}^{\ani}_G(k) \cap \bigsqcup_{\CE,\text{ type }=\kappa}\widetilde{\BM}^{G-\infty,\ani}_{H_\CE}(\bar{k}).$$

For a fixed $k$-rational point $a \in \wac^{\ani}_G(k)$, and a fixed $\kappa$ there is at most one coendoscopic datum $\CE_a$ of type $\kappa$ which occurs in the \emph{Hitchin fibre} $(I_{\hat\mu}\Mc_G)_{a}$. This follows from \cite[Proposition 6.3.3]{MR2653248}.

\begin{corollary}\label{kappa_IMTheorem}
Assume that the torsor $\infty_{\rho_{\CE_a}}(k)$ is non-empty, and choose a point $\infty_{\rho_\kappa} \in \rho_{\kappa,\infty}(k)$ mapping to $\infty_\rho$. The morphism of groupoids
\begin{equation*}
 \mu_{\CE_a} \colon  \widetilde{\BM}^{G-\infty,\ani}_{H_{\CE_a},a}(k) \to \left(I_{\hat\mu}\widetilde{\BM}^{\ani}_G(k)\right)_{a,\kappa}
\end{equation*}
is an equivalence.
\end{corollary} 



\subsection{The twisted inertia stack of the rigidification}
In Definition \ref{defi:higgs}(c) we introduced the rigidification $\widetilde{\Mc}^{\ani}_G$ of the stack $\widetilde{\BM}^{\ani}_G$ by its generic stabiliser $Z(X,G)$. Here we explain how to modify the statement of Theorem \ref{IMTheorem} to obtain a description of the twisted inertia stack $I_{\hat\mu} \widetilde{\mathcal{M}}^{\ani}_G$.

\begin{construction}\label{grig}
   Let $\CE=(\kappa,\rho_\kappa,\rho_\kappa \to \rho)$ be a coendoscopic datum for $\BG$ over $k$ and let $\infty_{\rho_\kappa}$ be a point in $\rho_\kappa(k)$ above $\infty$. 

Since the actions of $\Out(\BG)$ on $\BG$ and $\Out(\BH)$ on $\BH_\kappa$ preserve $Z(\BG)$ and $Z(\BH_\kappa)$ we have $Z(G)=Z(\BG) \times^{\pi_0(\kappa)} \rho_\kappa$ and $Z(H_\CE)=Z(\BH_\kappa)\times ^{\pi_0(\kappa)} \rho_\kappa$. Furthermore, property (iii) of Construction \ref{ikCons} implies that the inclusion $Z(\BG) \into Z(\BH_\kappa)$ is equivariant with respect to the action of $\pi_0(\kappa)$ on $Z(\BG)$ through $o_\BG$ and on $Z(\BH_\kappa)$ through $o_{\BH}$. Hence by twisting this inclusion with $\rho_\kappa$ we obtain a natural inclusion $Z(G) \into Z(H_\CE)$ over $X$. This induces an inclusion $Z(X,G) \into Z(X,H_\CE)$ of the Weil restrictions of these group schemes to $\Spec(k)$ (c.f. Definition \ref{defi:higgs}).

The group scheme $Z(X,H_\CE)$ acts naturally on every section of $\widetilde{\BM}^{G-\infty,\ani}_{H_\CE}$. Hence through the above inclusion $Z(X,G) \into Z(X,H_\CE)$ the same is true for the group $Z(X,G)$. 
\begin{definition}\label{defi:Grig}
We let $\widetilde{\mathcal{M}}^{G-\infty,\ani,G-\rig}_{H_\CE}$ be the rigidification (c.f. \cite[Theorem 5.1.5]{abramovich2003twisted}) of $\widetilde{\BM}^{G-\infty,\ani}_{H_\CE}$ by this action of $Z(X,G)$.
\end{definition}

The morphism $$\tilde\mu_\CE\colon \widetilde{\BM}^{G-\infty,\ani}_{H_\CE} \to I_{\hat\mu} \widetilde{\BM}^{\ani}_G$$ given by Construction \ref{TildeMuCons} is equivariant with respect to the action of $Z(X,G)$ on both these stacks. Hence it induces a morphism
\begin{equation*}
  \bar\mu_\CE\colon \widetilde{\mathcal{M}}^{G-\infty,\ani,G-\rig}_{H_\CE} \to I_{\hat\mu} \widetilde{\mathcal{M}}^{\ani}_G.
\end{equation*}
\end{construction}

\begin{construction}\label{const:coend_action}
Consider the set of pairs $(\CE,\infty_{\rho_\kappa})$ consisting of an endoscopy datum $\CE=(\kappa,\rho_\kappa,\rho_\kappa \to \rho)$ for $G$ over $k$ together with a point $\infty_{\rho_\kappa} \in \rho_{\kappa}(k)$ above $\infty$. There is a natural free action of the group $\Hom(\widehat{\mu},Z(X,\Gb))(k)$ on this set as follows: 

Let $(\CE,\infty_{\rho_\kappa})$ be such a pair and $\lambda\colon \widehat\mu \to Z(X,G)$ a homomorphism over $k$. The pair $(\rho_\kappa,\infty_{\rho_\kappa})$ amounts to a group homomorphism $\pi_1^{\text{\'et}}(X,\infty) \to \pi_0(\kappa)$. Since the $\pi_0(\kappa)$-torsor $\rho_{\kappa}$ induces the $\Out(\Gb)$-torsor $\rho$, we obtain
$$\Theta \defeq \mathsf{image}(\pi_1^{\text{\'e}t}(X,\infty) \to \pi_0(\kappa) \to \Out(\Gb)) = \mathsf{image}(\pi_1^{\text{\'et}}(X,\infty) \to \Out(\Gb)) .$$
We define $\pi_0(\kappa)_{\Theta}$ to be the pullback $\pi_0(\kappa) \times_{\Out(\Gb)} \Theta$. By the observation above, the $\pi_0(\kappa)$-torsor $\rho_{\kappa}$ has a canonical reduction to a $\pi_0(\kappa)_{\Theta}$-torsor $\rho_{\kappa,\Theta}$ which is equipped with a $k$-rational point $\infty_{\rho_{\kappa,\Theta}}$ above $\infty$. For this reason, the datum $(\CE,\infty_{\rho_\kappa})$ is equivalent to the datum $((\kappa,\rho_{\kappa,\Theta},\rho_{\kappa,\Theta}\to \rho),\infty_{\rho_{\kappa,\Theta}})$.

Now let $\kappa'\defeq \lambda \cdot \kappa$. Using the fact that $\lambda \in \Hom(\widehat{\mu},Z(X,G))(k) = \Hom(\widehat{\mu},Z(\Gb))(k)^{\Theta}$ we observe $\pi_0(\lambda\cdot{} \kappa)_{\Theta} = \pi_0(\kappa)_{\Theta}$. Hence we can consider the pair $((\kappa', \rho_{\kappa,\Theta},\rho_{\kappa,\Theta} \to \rho),\infty_{\rho_{\kappa,\Theta}})$ which by the above corresponds to a coendoscopic datum $\CE'\defeq (\kappa',\rho_{\kappa'},\rho_{\kappa'}\to \rho)$ together with a $k$-rational point $\infty_{\rho_{\kappa'}}$ above $\infty$. This defines the action of $\lambda$.


By construction, one has $H_\CE \simeq H_{\lambda \cdot{} \CE}$, and therefore an isomorphism $\widetilde{\mathcal{M}}^{G-\infty,\ani,G-\rig}_{H_\CE} \simeq \widetilde{\mathcal{M}}^{G-\infty,\ani,G-\rig}_{H_{\lambda \cdot{} \CE}}$.
\end{construction}

\begin{definition} \label{CompChoiceDef}
  We call a choice of points $\infty_\rho \in \rho_\infty(k)$ and $\infty_{\rho_\kappa} \in \rho_{\kappa,\infty}(k)$ for all coendoscopic data $\CE=(\kappa,\rho_\kappa,\rho_\kappa\to \rho)$ for $G$ over $k$ occurring on $I_{\hat\mu}\widetilde\BM^{\ani}_G(k)$ as in Theorem \ref{IMTheorem} \emph{compatible} if for every such $\CE$ and every $\lambda\colon \widehat{\mu} \to Z(G,X)$ the image of $(\CE,\infty_{\rho_\kappa})$ under $\lambda$ is of the form $(\CE',\infty_{\rho_{\kappa'}})$.
\end{definition}
\begin{rmk} \label{CompChoiceRmk}
  A compatible choice of points as in Definition \ref{CompChoiceDef} exists as soon as all the torsors $\rho_{\kappa,\infty}$ appearing there are trivial. This follows from the fact that for a pair $(\CE,\infty_{\rho_\kappa})$ as in Construction \ref{const:coend_action} and a non-trivial homomorphism $\lambda \colon \widehat{\mu} \to Z(X,G)$ sending $(\CE,\infty_{\rho_\kappa})$ to $(\CE',\infty_{\rho_{\kappa'}})$, the coendoscopic datum $\CE'$ is necessarily different from $\CE$.
\end{rmk}
Now Theorem \ref{IMTheorem} implies the following:
\begin{theorem} \label{IMTheoremTwisted}
  For a compatible choice of points $\infty_\rho \in \rho_\infty(k)$ and $\infty_{\rho_\kappa} \in \rho_{\kappa,\infty}(k)$ mapping to $\infty_\rho$ for every coendoscopic datum $\CE$ for $G$ over $k$ occurring on $I_{\hat\mu}\widetilde\BM^{\ani}_G(k)$, the resulting morphism of groupoids
\begin{equation*}
  \bigsqcup_{\{(\CE,\infty_{\rho_\kappa})\}/\Hom(\widehat\mu,Z(X,G))(k)}\bar \mu_\CE \colon  \bigsqcup_{\{(\CE,\infty_{\rho_\kappa})\}/\Hom(\widehat\mu,Z(X,G))(k)} \widetilde{\mathcal{M}}^{G-\infty,\ani,G-\rig}_{H_\CE}(k) \to I_{\hat\mu}\widetilde{\mathcal{M}}^{\ani}_G(k)
\end{equation*}
is an equivalence.
\end{theorem}

As observed in Corollary \ref{kappa_IMTheorem}, there's a slightly more refined statement by considering only coendoscopic data of type $\kappa$ and a fixed Hitchin fibre. For a $k$-rational point $a \in \wac^{\ani}_G(k)$, and a fixed $\kappa$ there is at most one coendoscopic datum $\CE_a$ of type $\kappa$ which occurs in the \emph{Hitchin fibre} $(I_{\hat\mu}\Mc_G)_{a}$. This follows from \cite[Proposition 6.3.3]{MR2653248}.

\begin{corollary}\label{kappa_IMTheoremTwisted}
Assume that the torsor $\infty_{\rho_{\CE_a}}(k)$ is non-empty, and choose a point $\infty_{\rho_\kappa} \in \rho_{\kappa,\infty}(k)$ mapping to $\infty_\rho$. The resulting morphism of groupoids
\begin{equation*}
 \mu_{\CE_a} \colon  \widetilde{\mathcal{M}}^{G-\infty,\ani,G-\rig}_{H_{\CE_a}}(k) \to \left(I_{\hat\mu}\widetilde{\mathcal{M}}^{\ani}_G(k)\right)_{a,\kappa}
\end{equation*}
is an equivalence.
\end{corollary} 

We also have the following corollary. 

\begin{corollary}\label{cor:codimension2}
The stack $\Mc_G^{\ani}$ is an algebraic space up to codimension $2$.
\end{corollary}

\begin{proof}
It suffices to prove this for $\widetilde{\Mc}_G^{\ani}$, since $\widetilde{\Mc}_G^{\ani} = \Mc_G^{\ani} \times_{\A^{\ani}} \widetilde{\A}^{\ani}$, and $\widetilde{\A}^{\ani} \to \A^{\ani}$ is \'etale. We remark that the map $\widetilde{\A} \to \A^{\ani}$ is not surjective for a fixed $\infty \in \C$. However we may vary the point to produce an \'etale cover of $\A^{\ani}$.

According to Theorem \ref{IMTheorem}, the closed points of the (twisted) inertia stack of $\widetilde{\Mc}_G^{\ani}$ is a disjoint union of even-dimensional stacks (this follows from \cite[4.13.4]{MR2653248} and our assumption that $\deg D$ is even). This implies the assertion, as the inertia stack measures the locus where an algebraic stack ceases to be an algebraic space.
\end{proof}

\subsection{Functoriality of the abstract Prym}
\begin{construction} \label{JFunct}
    Let $\CE=(\kappa,\rho_\kappa,\rho_\kappa \to \rho)$ be a coendoscopic datum for $\BG$ over $k$. Let $\mathfrak{c}_\CE^{\text{$G$-rs}} \subset \mathfrak{c}_\CE$ be the preimage of $\mathfrak{c}^\text{rs} \subset \mathfrak{c}$ under the morphism $\nu_\CE\colon \mathfrak{c}_\CE \to \mathfrak{c}$ from Construction \ref{muCons}.  Consider the group scheme $J$ over $\mathfrak{c}$ of regular centralizers associated to $G$ by \cite[Lemma 2.1.1]{MR2653248} as well as the analogous group scheme $J_\CE$ associated to $H_\CE$. By an argument completely analogous to the one given in \cite[Proposition 2.5.1]{MR2653248} we construct a canonical homomorphism
    \begin{equation*}
      \nu_\CE^* J \to J_{\CE}
    \end{equation*}
over $\mathfrak{c}_\CE$ which is an isomorphism over $\mathfrak{c}_\CE^{\text{$G$-rs}}$ as follows: 

We begin by considering the group schemes $J^1$ and $J^1_\CE$ associated to $G$ and $H_\CE$ as in \cite[2.4]{MR2653248}. As in the proof of \cite[Proposition 2.5.1]{MR2653248}, the group schemes can be described by
\begin{equation*}
  J^1=\prod_{\rho_\kappa \times \mathbf{t}/ \mathfrak{c}}(\rho_\kappa \times \mathbf{t} \times \BT)^{\BW \rtimes \pi_0(\kappa)}
\end{equation*}
and
\begin{equation*}
  J^1_\CE=\prod_{\rho_\kappa \times \mathbf{t}/ \mathfrak{c}_\CE}(\rho_\kappa \times \mathbf{t} \times \BT)^{\BW_{\BH_\kappa} \rtimes \pi_0(\kappa)}.
\end{equation*}
The morphism $\iota_\kappa\colon \BH_\kappa \rtimes \pi_0(\kappa)  \to \BG \rtimes \pi_0(\kappa)$ from Construction \ref{ikCons} induces a morphism $\BW \rtimes \pi_0(\kappa) \to \BW \rtimes \pi_0(\kappa)$ with respect to which the identity of $\rho_\kappa \times \mathbf{t} \times \BT$ is equivariant. Hence this morphism induces a morphism
\begin{equation*}
  \nu_\CE^*J^1 \to J^1_\CE.
\end{equation*}

By \cite[Proposition 2.4.7]{MR2653248} there are canonical open immersions $J \into J^1$ and $J_\CE \into J^1_\CE$. We want to show that under these embeddings the above morphism $\nu_\CE^* J^1 \to J^1_\CE$ restricts to a morphism $\nu_\CE^*J \to J_\CE$. To show this, as in the proof of \cite[Proposition 2.5.1]{MR2653248} it suffices to show that for every geometric point $a_\CE \in \mathfrak{c}_\CE(\bar k)$ with image $a=\nu_\CE(a_\CE) \in \mathfrak{c}(\bar k)$ the induced homomorphism 
\begin{equation*}
  \pi_0(J^1_a) \to \pi_0(J^1_{\CE,a_\CE})
\end{equation*}
sends $\pi_0(J_a)$ to $\pi_0(J_{\CE,a_\CE})$. Again as in \emph{loc.cit.} this follows from the description of $J$ inside $J^1$ given by \cite[Definition 2.4.5]{MR2653248} and the fact that the root system of $\BH_\kappa$ is contained in the root system of $\BG$.

Finally it follows from \cite[Lemma 1.9.2]{MR2653248} that over $\mathfrak{c}_\CE^{\text{$G$-rs}}$ the morphism $\nu_\CE^*J^1 \to J^1_\CE$, and hence the morphism $\nu_\CE^* J \to J_\CE$, is an isomorphism.

Now we consider the abstract Pryms $\Pb$ and $\Pb_\CE$ over $\A$ and $\A_\CE$ associated to $G$ and $H_\CE$. By definition, for a scheme $S$ over $\A$, the category $\Pb(S)$ is the category of torsors under the pullback of $J$ under the morphism $X \times S \to X \times \A \to \mathfrak{c}$. Hence the above morphism $\nu_\CE^* J \to J_\CE$ induces a canonical morphism
\begin{equation} \label{PFunc}
  \nu_\CE^* \Pb \to \Pb_\CE
\end{equation}
over $\A_\CE$.
\end{construction}

\begin{proposition} \label{PFuncComp}
  Let $\CE=(\kappa,\rho_\kappa,\rho_\kappa \to \rho)$ be a coendoscopic datum for $\BG$ over $k$. The morphisms $\BM_{H_\CE} \to \BM_G$ from Construction \ref{muCons} and the morphism $\nu_\CE^* \Pb \to \Pb_\CE$ are compatible in the following sense: Let $S$ be a scheme over $\A_\CE$ together with objects $F \in \BM_{H_\CE}(S)$ and $M \in \Pb(S)$. If $E$ denotes the image of $F$ in $\BM_G(S)$ and $N$ the image of $M$ in $\Pb_\CE(S)$, then there is a canonical isomorphism between the image of $F \otimes N$ in $\BM_G(S)$ and $E \otimes M$. 
\end{proposition}
\begin{proof}
  Let $[\chi]\colon [\mathfrak{g}_D/G] \to \mathfrak{c}_D$ and $[\chi_\CE]\colon [\mathfrak{h}_{\CE,D}/H_\CE] \to\mathfrak{c}_{\CE,D}$ be the canonical morphisms. By \cite[4.3.2]{MR2653248}, the action of $\Pb$ on $\BM_G$ is defined via the canonical morphism of group stacks $[\chi]^*J \to I[\mathfrak{g}_D/G]$ given by \cite[Lemma 2.2.1]{MR2653248}. Using this, and the analogous description of the action of $\Pb_\CE$ on $\BM_{H_\CE}$, one sees that to prove the claim it suffices to prove that the diagram
  \begin{equation} \label{ISquare}
    \xymatrix{
      [\chi_\CE]^* \nu_\CE^* J \ar[r] \ar[d] & [\chi_\CE]^* J_\CE \ar[d] \\
      \mu_\CE^* I[\mathfrak{g}_D/G]  & I [\mathfrak{h}_{\CE,D}/H_\CE], \ar[l] \\
}
  \end{equation}
of group stacks over $[\mathfrak{h}_{\CE,D}/H_\CE]$ in which the bottom morphism is the one induced by the morphism $\mu_\CE\colon [\mathfrak{h}_{\CE,D}/H_\CE] \to [\mathfrak{g}_D/G]$ from Construction \ref{muCons}, is commutative.

The commutativity of \eqref{ISquare} can be shown as follows: Let $\mathfrak{h}_{\CE,D}^\text{$G$-rs} \subset \mathfrak{h}_{\CE,D}$ be the preimage of the open subset $\mathfrak{c}_{\CE}^\text{$G$-rs}$ from Construction \ref{JFunct}. This is an $H_\CE$-invariant open subset of $\mathfrak{h}_{\CE,D}$ which non-empty in every geometric fibre and hence Zariski dense. Since the group scheme $J$ is smooth, this implies that $[\chi_\CE]^* \nu_\CE^* J|_{[\mathfrak{h}_{\CE,D}^\text{$G$-rs}/H_\CE]}$ is Zariski dense in $[\chi_\CE]^* \nu_\CE^* J$. Hence it suffices to check the commutativity of \eqref{ISquare} over $[\mathfrak{h}_{\CE,D}^\text{$G$-rs}/H_\CE]$.

Over $[\mathfrak{h}_{\CE,D}^\text{$G$-rs}/H_\CE]$, all the groups appearing in \eqref{ISquare} are tori which are canonically equivalent to the torus of regular semisimple centralizers on $[\mathfrak{h}_{\CE,D}^\text{$G$-rs}/H_\CE]$. Using this one can verify the desired commutativity.
\end{proof}

\begin{construction}\label{const:funct}
Let $\CE=(\kappa,\rho_\kappa,\rho_\kappa \to \rho)$ be a coendoscopic datum for $\BG$ over $k$ and let $\infty_{\rho_\kappa} \in \rho_\kappa(k)$ be a point above $\infty$.
  \begin{enumerate}[(i)]
  \item We let $\widetilde{\Pb}^{G-\infty,\ani}_\CE$ be the pullback of $\Pb_\CE$ along $\widetilde{\A}^{G-\infty,\ani}_\CE \to \A$.
  \item As in Construction \ref{grig}, the group $Z(X,G)$ acts on every section of $\widetilde{\Pb}^{\G-\infty,\ani}_\CE$ through the natural homomorphism $Z(X,G) \to Z(X,H_\CE)$. We let $\widetilde{\Pc}^{G-\infty,\ani,G-\rig}_\CE$ be the rigidification of $\widetilde{\Pb}^{\G-\infty,\ani}_\CE$ by this action.
  \item The action of $\Pb_\CE$ on $\Mb_\CE$ induces an action of $\widetilde{\Pc}^{G-\infty,\ani,G-\rig}_\CE$ on $\widetilde{\Mc}^{G-\infty,\ani,G-\rig}_\CE$.
  \item The morphism $\nu^*\Pb \to \Pb_\CE$ from \eqref{PFunc} is compatible with the actions of $Z(X,G)$. Hence it induces a morphism
    \begin{equation}
      \label{PFunc2}
      \tilde\nu_\CE^* \widetilde{\Pc} \to \widetilde{\Pc}^{G-\infty,\ani,G-\rig}_\CE
    \end{equation}
    over $\widetilde{\A}_\CE^{G-\infty,\ani,G-\rig}$.
  \item The morphism \eqref{PFunc2} satisfies a compatibility property with the morphism 
    \begin{equation*}
      \bar\mu_\CE\colon \widetilde{\mathcal{M}}^{G-\infty,\ani,G-\rig}_{H_\CE} \to I_{\hat\mu} \widetilde{\mathcal{M}}^{\ani}_G.
    \end{equation*}
    analogous to the one given by Proposition \ref{PFuncComp}.
  \end{enumerate}
\end{construction}
\begin{lemma} \label{PFunctSurj}
  Let $\CE=(\kappa,\rho_\kappa,\rho_\kappa \to \rho)$ be a coendoscopic datum for $\BG$ over $k$ and let $\infty_{\rho_\kappa} \in \rho_\kappa(k)$ be a point above $\infty$. For every point $a \in \widetilde\A_\CE(\bar k)$, the fibre $\widetilde{\Pc}_a \to \widetilde{\Pc}^{G-\infty,\ani,G-\rig}_{\CE,a}$ of the morphism from point (iv) above is surjective.
\end{lemma}
\begin{proof}
  As in \cite[4.17.2]{MR2653248} this follows from the fact that the morphism $ J_a \to J_{\CE,a}$ from Construction \ref{JFunct} is generically an isomorphism.
\end{proof}

\section{The proof of Geometric Stabilisation}\label{proofgst}

In this section we finally prove the Geometric Stabilisation Theorem in Corollary \ref{cor:geometric_endoscopy}. The proof will be a consequence of Theorem \ref{mainidentity}, which is an equality between character sums on the twisted inertia stacks $I_{\widehat{\mu}}\Mc_G$ and $I_{\widehat{\mu}}\Mc_{\widehat{G}}$.

Let $k$ be a finite field and write $F=k((x))$ and $\Oo = k[[x]]$. We fix objects $X_k$, $D_k$, $\BG_k$, $\rho_k$, $G_k$, etc. satisfying Situation \ref{situation:Higgs_general} over $R=k$. We assume the existence of a point $\infty \in X(k)$ and fix such a choice. 
Furthermore, we assume that our base field $k$ is sufficiently big, such that the following holds:

\begin{situation}\label{situation:split}
We assume that the torsor $\rho_{\infty}$ splits over $\Spec k$. 
\end{situation}
We choose a point $\infty_\rho \in \rho_\infty(k)$. This choice induces in particular an isomorphism $T_{\infty} \simeq \Tb$.

We denote by $X$, $D$, $\BG$, $\rho$, $G$, etc. the base changes of the objects just introduced to $\Oo$. These satisfy the conditions of Situation \ref{situation:Higgs_general} over $R=\Oo$ and we will denote by $\A$, $\Mc_G$, $\Pc_G$, etc. the associated objects over $\Oo$ from Section \ref{HiggsBundles}.

We assume that the degree $d$ of $D$ is even and satisfies $d \geq 2g-2$. By \cite[4.14.1]{MR2653248} this implies smoothness of $\widetilde{\Mc}_G^{\ani}$. We will also assume that $\A^\Diamond \subset \A$ is non-empty. By \cite[4.7.1]{MR2653248} a sufficient condition for this is $d > 2g$.

By construction, the space $\widetilde{\Mc}_G^{\ani}$ is constant over $\Oo$ (i.e. it is isomorphic to the base change of its special fibre to $\Oo$). By the above assumptions \cite[Corollary 4.11.3]{MR2653248} it is a smooth tame DM-stack. Hence by applying \cite[Theorem 4.4 \& Proposition 5.2]{MR2483938} to the special fibre of $\widetilde{\Mc}_G^{\ani}$ it follows that $\widetilde{\Mc}_G^{\ani}$ is Zariski-locally a finite \'etale quotient stack. This will enable us to apply Theorem \ref{thm:volume} to this stack.

Throughout this section we work over the base $\widetilde{\A}^{\ani}$. To avoid proliferation of superscripts we will generally drop the superscript ${}^{\ani}$. For a coendoscopic datum $\CE$ for $G$ we will also abbreviate the space $\widetilde{A}^{G-\infty}_\CE$ from Construction \ref{TildeMuCons} as $\widetilde{A}^G$ and the space $\widetilde{\mathcal{M}}^{G-\infty,\ani,G-\rig}_{H_\CE}$ from Definition \ref{defi:Grig} as $\widetilde{\mathcal{M}}^G_{H_\CE}$. For a stack $\widetilde{\CM}_*^?$, we will denote by $\widetilde{M}_*^?$ its coarse moduli space.

We will consider the type $\kappa$ of a coendoscopic datum for $G$ over $k$ interchangeably as an element of one of the following groups:
\begin{equation*}
  \Hom_k(\widehat{\mu},\Tb)=\colim_{n \in \Nb'} \Hom_k(\mu_n,\Tb)=\colim_{n \in \Nb'} \Xb_*(\Tb) \otimes \mathbb{Z} /n \mathbb{Z} = \Xb_*(\Tb) \otimes (\mathbb{Q}/ \mathbb{Z})' = \Xb^*(\widehat{\Tb}) \otimes (\mathbb{Q}/\mathbb{Z})'.
\end{equation*}

We let $a$ be a point in $\widetilde{\A}(k)$.

In order to apply our results on the twisted inertia stack of $\wmc_G$ from Section \ref{secinst} we will need to assume that for certain coendoscopic data $(\kappa,\rho_\kappa,\rho_\kappa \to \rho)$ the torsor $\rho_{\kappa,\infty}$ splits. The strongest assumption of this kind would be the following:
\begin{situation} \label{situation:strongsplit}
  The torsors $\rho_{\kappa,\infty}$ for all endoscopic data $\CE$ for $G$ (resp. $\hat G$) over $k$ occurring on $I_{\hat\mu}\widetilde\BM^{\ani}_G(k)$ (resp. $I_{\hat\mu}\widetilde\BM^{\ani}_{\hat G}(k)$), split. 
\end{situation}
By Remark \ref{SplitingRmk}, the assumption of Situation \ref{situation:strongsplit} can always be achieved by replacing $k$ by a finite extension. If we put ourselves in this situation, we always fix a compatible choice of splittings $\infty_\rho$ and $\infty_{\rho_\kappa}$ of these torsors as in Definition \ref{CompChoiceDef}. This is possible by Remark \ref{CompChoiceRmk}.

However, for the proof of the Fundamental Lemma the assumptions of Situation \ref{situation:strongsplit} are too strong. Hence in the Geometric Stabilisation Theorem, following Ng\^o we will only assume that the torsor $\rho_{\kappa,\infty}$ for one specific coendoscopy datum splits.
\subsection{The set-up}\label{set-up}
Recall Ng\^o's result on the existence of a surjective morphism 
\begin{equation}\label{ngomap}
  \widetilde{\A} \times \Xb^*(\widehat{\Tb}) \twoheadrightarrow \pi_0(\Pc_G),
\end{equation} 
see Proposition \ref{prop:surjection}. 

We denote by $\blacksquare \subset \pi_0(\Pc_G)$ the image of the subgroup $\Xb^*(\widehat{\Tb}/Z(X,\widehat{G}))$. Let $\Pc_G^{\blacksquare}$ be the base change
$$\Pc_G^{\blacksquare} = \Pc_G \times_{\pi_0(\Pc_G)} \blacksquare.$$

We fix a square root $D'$ of $D$ and denote the associated Kostant section by $\epsilon = [\epsilon]^{D'}: \wac \to \wmc_G$ (see \ref{prop:kostant}). Let $\wmc_G^\blacksquare \subset \wmc_G$ be the closure inside $\widetilde{\Mc}_G$ of the $\wpc_G^{\blacksquare}$-translate of the Kostant section $\epsilon: \wac \to \wmc_G$. 
%

\begin{rmk}
The generic fibre of $\wpc_G$ is an extension of an abelian variety by a finite group scheme. The group scheme $\wpc_{G}^\blacksquare$ restricted to $\A^\Diamond$ satisfies $\pi_0(\Pc_G^{\blacksquare}) = 0$ (see \cite[Corollaire 4.10.4]{MR2653248}). In particular, its restriction to $\A^{\Diamond}$ is a family of abelian varieties. This is the main reason for introducing the ${}^\blacksquare$-spaces.
\end{rmk}

For the remainder of this subsection let $\CE=(\kappa,\rho_\kappa,\rho_\kappa \to \rho)$ be a coendoscopic datum for $G$ over $k$ together with a point $\infty_{\rho_{\kappa}} \in \rho_{\kappa,\infty}(k)$. Similarly to the above we define an open substack $\wmc_{H_{\Ec}}^\blacklozenge \subset \wmc^{G}_{H_{\Ec}}$: 

\begin{definition}
 We define $\wmc_{H_{\Ec}}^\blacklozenge \subset \wmc^{G}_{H_{\Ec}}$ to be the maximal open substack (actually a union of connected components), for which the diagram
\[
\xymatrix{
 \wmc_{H_{\Ec}}^\blacklozenge \ar[r] \ar[d] & I_{\widehat{\mu}}\widetilde{\Mc}_G^{\blacksquare} \ar[d] \\
 \wmc_{H_{\Ec}}^G \ar[r]^{\bar{\mu}_\CE} & I_{\widehat{\mu}}\widetilde{\Mc}_G
}
\]
commutes.
\end{definition}

Further we remark that the choice of square root $D'$ induces a Kostant section $\widetilde{\A}_\CE \to \widetilde{\BM}_{H_\CE}$. By pulling back  $\wmc_{H_{\Ec}}^{\blacklozenge}$ along the induced map $\Pc_{H_{\Ec}} \to \wmc_{H_{\Ec}}$ we define $\wpc_{H_{\Ec}}^{\blacklozenge}$.

\begin{rmk}\label{remarky}
One has a natural inclusion $\Pc_{H_{\Ec}}^{\blacklozenge} \supset \Pc_{H_{\Ec}}^{\blacksquare}$, and the quotient agrees with $(Z(X,\widehat{H}_{\Ec})/Z(X,\widehat{G}))^{\vee}$. This follows from
\[
\xymatrix{
\Xb^*(\widehat{\Tb}/Z(X,\widehat{H}_{\Ec})) \ar[r] \ar@{^(->}[d] & \pi_0(\Pc_{H_{\Ec}}^{\blacksquare}) \ar@{^(->}[d] \\
\Xb^*(\widehat{\Tb}/Z(X,\widehat{G})) \ar[r] & \pi_0(\Pc_{H_{\Ec}}^{\blacklozenge}) 
}
\]
and the fact that connected components of $\widetilde{\Mc}_{H_{\Ec}}$ are parametrised by $Z(X,\widehat{H}_{\Ec})^{\vee}$, and similarly connected components of $\widetilde{\Mc}_G$ by $Z(X,\widehat{G})^{\vee}$ (this follows from \cite[Corollaire 4.10.4]{MR2653248} and the density result \cite[Proposition 4.16.1]{MR2653248}).
\end{rmk}


%

\subsection{Unramified twists of Hitchin systems}

For a base scheme $S$ (the cases $\Spec k$ and $\Spec \Oo$ suffice), an element $t \in H^1_{\text{\'et}}(S,\wpc_{G,a}^{\blacksquare})$ corresponds to a $\wpc_{G,a}^{\blacksquare}$-torsor $T$, which is well-defined up to isomorphism. We define the twist to be the Deligne-Mumford stack
$$\wmc_{G,a}^{\blacksquare,t} = \wmc_{G,a}^{\blacksquare} \times_S^{\wpc_{G,a}^{\blacksquare}} T = [(\wmc_{G,a}^{\blacksquare} \times T) / {\wpc_{G,a}^{\blacksquare}}].$$
By construction, $\wmc_{G,a}^{\blacksquare,t} \to S$ is \'etale-locally equivalent to $\wmc_{G,a}^{\blacksquare} \to S$.

By Lang's Theorem and Proposition \ref{prop:surjection} we have $H^1(k,\wpc_{G,a}^\blacksquare) = H^1(k,\pi_0(\wpc_{G,a}^\blacksquare)) \cong \pi_0(\wpc_{G,a}^\blacksquare)$ and hence we can define for every $t \in \Xb^*(\widehat{\Tb}/Z(X,\widehat{G}))$ twists $\wmc_{G,a}^{\blacksquare,t}$ of $\wmc_{G,a}^{\blacksquare}$ by $t$, where we use that $t$ induces an element in $\pi_0(\wpc_{G,a}^\blacksquare)$, using the map \eqref{ngomap}. Analogously we define twists $\wmc_{G,a}^t$. The relevance of these twists is demonstrated by the following lemma:

\begin{lemma} \label{Fourier}
The Fourier transform of the function 
$$\qquad \qquad \ \pi_0(\wpc_{G,a}^{\blacksquare}) \to \bar{\Qb}_{\ell}, \; \;t \mapsto \#\wmc_{G,a}^{\blacksquare,t}(k) = \Tr(\Fr, H_c^*(\wmc_{G,a}^{\blacksquare},\bar{\Qb}_{\ell}))$$ is the function
$$\pi_0(\wpc_{G,a}^{\blacksquare})^* \to \bar{\Qb}_{\ell},\; \; \chi \mapsto \Tr(\Fr, H^*_c(\wmc_{G,a}^{\blacksquare},\bar{\Qb}_{\ell})^{\chi}).$$
In particular, we have that the average value  $\frac{1}{|\pi_0(\wpc_{G,a}^{\blacksquare})|} \sum_{t} \#\wmc_{G,a}^{\blacksquare,t}(k)$ equals $\Tr(\Fr, H_c^*(\wmc_{G,a}^{\blacksquare},\bar{\Qb}_{\ell})^{\id})$, the stable part of the point-count.
\end{lemma}

\begin{proof}
By construction, the base change $\wmc_{G,a}^{\blacksquare,t} \times_k \bar{k}$ is isomorphic to $\wmc_{G,a}^{\blacksquare}$. From this we get an isomorphism of $\ell$-adic vector spaces
$$H_c^*(\wmc_{G,a}^{\blacksquare},\bar{\Qb}_{\ell}) \simeq H_c^*(\wmc_{G,a}^{\blacksquare,t},\bar{\Qb}_{\ell}).$$
With respect to this identification, the action of the geometric Frobenius $\Fr^t$ on the right hand side, is sent to
$$\Fr^t=t\cdot{}\Fr,$$
where we use that $t \in \pi_0(\wpc_{G,a}^{\blacksquare})$ acts on the cohomology of $\wmc_{G,a}^{\blacksquare}$ by homotopy invariance. One has
$$\Tr(\Fr_t, H_c^*(\wmc_{G,a}^{\blacksquare})^{\chi}) = \chi(t)\cdot{}\Tr(\Fr, H_c^*(\wmc_{G,a}^{\blacksquare})^{\chi}).$$
Which implies
$$\frac{1}{|\pi_0(\wpc_{G,a}^{\blacksquare})|} \sum_{t} \chi^{-1}(t) \#\wmc_{G,a}^{\blacksquare,t}(k) = \Tr(\Fr, H_c^*(\wmc_{G,a}^{\blacksquare})^{\chi}).$$
This concludes the argument.
\end{proof}

Let now $\Ec = (\kappa,\rho_\kappa,\rho_\kappa \to \rho)$ be a coendoscopy datum for $G$ over $k$ together with a point $\infty_{\rho_\kappa} \in \rho_{\kappa,\infty}(k)$.
\begin{lemma}\label{lemma:coendo_surjection}
There exists a surjective map $\pi_0(\wpc_{G}^\blacksquare) \times_{\widetilde{\A}} \widetilde{\A}_{H_{\Ec}} \twoheadrightarrow  \pi_0(\wpc_{H_{\Ec}}^{\blacklozenge})$, rendering the diagram
\[
\xymatrix{
\Xb^*(\widehat{\Tb}/Z(X,\widehat{G})) \times \wac_{H_{\Ec}} \ar@{->>}[r]  &  \pi_0( \wpc_{{G}}^{\blacksquare}) \times_{\wac_G} \wac_{H_{\Ec}} \ar@{->>}[d]\\
\Xb^*(\widehat{\Tb}/Z(X,\widehat{H}_{\Ec})) \times \wac_{H_{\Ec}} \ar@{->}[r] \ar@{^(->}[u]  & \pi_0( \wpc_{{H}_{\Ec}}^{\blacklozenge})
}
\]
commutative. 
\end{lemma}

\begin{proof}
In Construction \ref{const:funct} we define a natural morphism 
\begin{equation}\label{eqn:endo}\wpc_{{G}} \times_{\wac_{G}} \wac_{H_{\Ec}} \to \wpc_{{H}_{\Ec}}.
\end{equation} 
We claim that this morphism sends $\wpc_G^\blacksquare \times_{\wac_G} \wac_{H_\Ec}$ to $\wpc_{{H}_{\Ec}}^{\blacklozenge}$. Since we are dealing with open subspaces, it suffices to check this on $\bar k$-valued points. Hence we may also restrict to the fibre over a point $a \in \widetilde \A_{\Ec}(\bar k)$. By Lemma \ref{PFunctSurj} the morphism $\wpc_{G,a} \to \wpc_{H_\Ec,a}$ is surjective. Using this and the compatibility given by Construction \ref{const:funct} (v) the desired fact follows from the definition of the $\blacksquare$- and $\blacklozenge$-spaces.

So we obtain a morphism
\begin{equation}
  \wpc^{\blacksquare}_{{G}} \times_{\wac_{G}} \wac_{H_{\Ec}} \to \wpc_{{H}_{\Ec}}^\blacklozenge.
\end{equation} 
The induced map between sheaves of connected components produces the commutative diagram above.

It remains to show that the map $\pi_0( \wpc_{{G}}^{\blacksquare}) \times_{\wac_G} \wac_{H_{\Ec}} \to  \pi_0( \wpc_{{H}_{\Ec}}^{\blacklozenge})$ is surjective. This may again be checked on geometric fibres, in which case it follows from the surjectivity of $\wpc_{G,a} \to \wpc_{H_\Ec,a}$ used above.
\end{proof}



The morphism $ \pi_0(\wpc_{G,a}^\blacksquare) \to  \pi_0(\wpc_{H_{\Ec},a}^\blacklozenge)$ allows us to define twists $\wmc_{{H}_{\Ec},a}^{\blacklozenge,t}$ which are compatible with twists $\wmc_{{G},a}^{\blacksquare,t}$ in the following sense:

\begin{lemma}

For every $a \in \wac_\CE(k)$ and every $t \in \pi_0(\wpc_{G,a}^{\blacksquare})$ there is a natural equivalence
\[    \left(I_{\widehat{\mu}}\Mc_{G,a}^{\blacksquare,t}(k)\right)_{\kappa} \cong \Mc_{H_{\Ec},a}^{\blacklozenge,t}(k).
\]
\end{lemma}

\begin{proof}
For every $t \in \pi_0(\wpc_{G,a})$ there exists a $\lambda \in \Xb_*(\Tb/Z(X,G))$, such that the map $\Xb_*(\Tb) \to \pi_0(\wpc_G)$ of Proposition \ref{prop:surjection} sends $\lambda \mapsto t$. We denote by $\lambda_H$ the induced element of $\Xb_*(\Tb/Z(X,H_{\Ec}))$.

The Abel-Jacobi map of Construction \ref{const:AJ} yields rational points $\AJ_{G}(\lambda)\in \wpc_{G,a}(k)$ and $\AJ_{{H}_{\Ec}}(\lambda_H) \in \wpc_{{H}_{\Ec}}(k)$ which are compatible with respect to the natural map $\wpc_{G,a} \to \wpc_{{H}_{\Ec},a}$. Via the natural action of these groups we obtain automorphisms $\AJ_G(\lambda)$ of $\widetilde{\Mc}_{G,a}$ and $\AJ_{H_{\Ec}}(\lambda_H)$ of $\widetilde{\Mc}_{H_{\Ec},a}$.These fit into a commutative diagram
\[
\xymatrix{
\widetilde{\Mc}^{\blacklozenge}_{{H}_{\Ec},a} \ar[r]^-{\AJ_{H_{\Ec}}(\lambda_H)} \ar[d] & \widetilde{\Mc}^{\blacklozenge}_{{H}_{\Ec},a} \ar[d] \\
I_{\widehat{\mu}}\widetilde{\Mc}_{G,a} \ar[r]^-{\AJ_{G}(\lambda)} & I_{\widehat{\mu}}\widetilde{\Mc}_{G,a}.
}
\]
This yields a morphism $\widetilde{\Mc}_{H_{\Ec},a}^{\blacklozenge,\lambda_H} \to I_{\widehat{\mu}}\widetilde{\Mc}_{G,a}^{\blacklozenge,\lambda}$.

The claim that 
$$\Mc_{H_{\Ec},a}^{\blacklozenge,t} (k) \to \left(I_{\widehat{\mu}}\wmc_{G,a}^{\blacksquare,t}(k)\right)_{\kappa}$$
is a bijection follows from the statement above for the trivial torsor $t =0$, the fact $\wmc_{G,a}^{\blacksquare,t} \times_{k} \bar{k} \simeq \wmc_{G,a}^{\blacksquare} \times_{k} \bar{k}$, and Galois descent: with respect to the identification above, a $\bar{k}$-rational point $y \in I_{\widehat{\mu}}\wmc_{G,a}^{\blacksquare}(\bar{k})$ corresponds to a $k$-rational of the twist $I_{\widehat{\mu}}\wmc_{G,a}^{\blacksquare,t}$ if and only if it satisfies 
$$\phi(y) = \AJ_G(\lambda)^{-1}y,$$
where $\phi$ denotes the Frobenius of $\bar{k}/k$.
This is the case if and only if the corresponding $\bar{k}$-rational point of $\widetilde{\Mc}_{H_{\Ec},a}^{\blacklozenge,G-\rig} (\bar{k})$ satisfies $\phi(y) = \AJ_G(\lambda)^{-1}y.$ Hence, if and only if $y$ corresponds to a $k$-rational point of $\widetilde{\Mc}_{H_{\Ec},a}^{\blacklozenge,t} (k)$.
\end{proof}

\subsection{A comparison of $p$-adic integrals} \label{PAdicComparison}

For a groupoid $\Xc$ we denote by $[\Xc]$ its set of isomorphism classes. For every character $s\in \Xb^*(\Tb/Z(X,G))$ we define the function 

\[ \chi_s: [I_{\widehat{\mu}}\wmc_{G,a}^{\blacksquare,t}(k)] \to \BQ / \BZ \]
by

\[  \chi_{s}| _{[I_{\widehat{\mu}}\wmc_{G,a}^{\blacksquare,t}(k)_\kappa] } \equiv s(\kappa)\]
for every $\kappa \in \Xb^*(\widehat \Tb)\otimes (\mathbb{Q}/\mathbb{Z})'$.

As before we write $\Oo^h_a$ for the henselisation of $\Oo_{\wac}$ at $a$ and $U_a=\Spec(\Oo^h_a)$. The $\Oo$-rational points of $U_a$ can be identified with the $\Oo$-rational points of $\wac$ whose special fibre equals $a$. We then define 
\[U_a(\Oo)^\flat  = U_a(\Oo) \cap \wac^\Diamond(F).\]
For a $U_a$-space $Y$ we write $Y(\Oo)^{\flat}$ to denote the fibre product of sets $Y(\Oo) \times_{U_a(\Oo)} U_a(\Oo)^\flat$.

\begin{rmk}\label{rmk:zeroset}
We have $U_a(\Oo)^\flat = U_a(\Oo) \setminus (U_a \setminus U_a^{\Diamond})(\Oo_F)$, and $(U_a \setminus U_a^{\Diamond})(\Oo_F)$ is a set of measure zero (with respect to Weil's canonical measure). This follows from Proposition \ref{intools}\eqref{measurezero}.
\end{rmk}

By virtue of $H^1(\Oo_a^h,\wpc_{G,U_a}^\blacksquare) \cong \pi_0(\wpc_{G,a}^\blacksquare)$, we get twists $\wmc_{G,U_a}^{\blacksquare,t}$ and $I_{\widehat{\mu}}\wm_{G,U_a}^{\blacksquare,t}$ for every $t \in \pi_0(\wpc_{G,a}^\blacksquare)$. Recall that we have a $p$-adic manifold $\wm_{G,U_a}^{\blacksquare,t}(\Oo)^\natural \subset \wm_{G,U_a}^{\blacksquare,t}(\Oo)$ and furthermore we have a specialisation morphism (see Definition \ref{defi:e})

\begin{equation}\label{prest}  e\colon \widetilde{M}_{G,U_a}^{\blacksquare,t}(\Oo)^\natural \to    [I_{\widehat{\mu}}\wmc_{G,a}^{\blacksquare,t}(k)]. 
\end{equation}

By construction, we have an inclusion $\wm_{G,U_a}^{\blacksquare,t}(\Oo)^\flat \subset \wm_{G,U_a}^{\blacksquare,t}(\Oo)^\natural$.

The composition of $\chi_s$ with the specialisation map yields a function 
$$\widetilde{M}_{G,U_a}^{\blacksquare,t}(\Oo)^\natural \to \Qb/\Zb \subset \Cb^{\times}$$ which we also denote by $\chi_s$.


In Definition \ref{defi:weight} we defined the weight function
$$w\colon [I_{\widehat{\mu}}\widetilde{\Mc}_G(k)] \to \Qb$$
which is locally constant on $I_{\widehat{\mu}}\widetilde{\Mc}_G(k)$. We claim that with respect to the decomposition
\begin{equation*}
  [I_{\widehat{\mu}}\widetilde{\Mc}_G(k)] \simeq \bigsqcup_{\kappa} [I_{\widehat{\mu}}\widetilde{\Mc}_G(k)_\kappa]
\end{equation*}
the weight function depends only $\kappa$:

\begin{lemma}\label{lemma:shift_constant}
There exists a function $\CF\colon \Xb_*(\Tb) \otimes (\Qb/\Zb)' \to \Qb$, such that for every $\kappa \in \Xb_*(\Tb)\otimes (\Qb/\Zb)'$  we have $w|_{[I_{\widehat{\mu}}\widetilde{\Mc}_G(k)_\kappa]} = \CF(\kappa)$.
\end{lemma}

\begin{proof}
For the proof of the claim we may replace $k$ by a finite extension. Hence we may put ourselves into Situation \ref{situation:strongsplit}. Then under the description of $I_{\widehat{\mu}}\widetilde{\Mc}_G(k)_\kappa$ given by Corollary \ref{kappa_IMTheorem} we want to show that for a coendoscopic datum $\CE$ for $G$ over $k$ of type $\kappa$ the restriction of the weight function to $\wmc_{H_\CE}^G(k)$ is constant with value $\CF(\kappa)$.

The weight function $w\colon I_{\widehat{\mu}}\widetilde{\Mc}_G(k) \to \Qb$ is locally constant. Therefore for $a \in \widetilde{\A}_{\Ec}(k)$ it induces a function $w\colon \pi_0(\widetilde{\Mc}_{H_{\Ec},a}^G) \to \Qb$. In order to show that this function is constant, it suffices to to prove that the Prym $\widetilde{\Pc}_{G,a}$ (which acts through automorphisms of $\widetilde{\Mc}_G$), acts transitively on $\pi_0(\widetilde{\Mc}^G_{H_{\Ec},a})$. We have seen in Lemma \ref{lemma:coendo_surjection} that the natural map $\pi_0(\widetilde{\Pc}_{G,a}) \to \pi_0(\widetilde{\Pc}_{H_{\Ec},a})$ is surjective.
\end{proof}

We can now formulate our main identity, which is a relation between point counts of Hitchin fibres of all coendoscopic groups at once. Here the point count $\# \Mc(k)$ of a smooth Deligne-Mumford stack $\Mc$ over $k$ is defined as the weighted count of the groupoid $\Mc(k)$:
\[ \# \Mc(k) = \sum_{[x] \in [ \Mc(k)]}  \frac{1}{|\Aut(x)|}.  \]
In order to relate this to the point counts in \eqref{ladicpc} we will implicitly use the trace formula \cite{Su12}
 \[\# \Mc(k) = \sum_n (-1)^n \mathrm{tr}\left(\text{Frob}_k,H_c^n(\Mc, \overline{\Qb}_\ell)\right).\]

\begin{theorem}\label{mainidentity} For all $a \in \wac(k)$, $t \in \Xb^*(\widehat{\Tb}/Z(X,\widehat{G}))$ and $s \in   \Xb^*(\Tb/Z(X,G))$ we have
\begin{equation}\label{mainsum} 
\sum_{y \in [I_{\widehat{\mu}} \wmc_{G,a}^{t,\blacksquare}(k)]} \frac{\chi_s(y)\cdot q^{-w(y)}}{|\Aut_{I_{\widehat{\mu}} \wmc_{G,a}^{t,\blacksquare}(k)}(y)|} = \sum_{z \in [I_{\widehat{\mu}} \wmc_{\widehat{G},a}^{s,\blacksquare}(k)]} \frac{\chi_t(z)\cdot q^{-w(z)}}{|\Aut_{I_{\widehat{\mu}} \wmc_{\widehat{G},a}^{t,\blacksquare}(k)}(z)|}.
\end{equation}
\end{theorem}

We will deduce this theorem from $p$-adic integration.The proof of \ref{mainidentity} can be found at the end of this subsection, as we have to establish a few auxiliary results first. 

If we are in Situation \ref{situation:strongsplit}, Theorem \ref{mainidentity} can be reformulated as follows using the description of the twisted inertia stack given in Theorem \ref{IMTheoremTwisted}:
\begin{corollary}\label{cor:mainidentity} 
For all $a \in \wac(k)$, $t \in \Xb^*(\widehat{\Tb}/Z(X,\widehat{G}))$ and $s \in  \Xb^*(\Tb/Z(X,G))$ we have
\begin{equation}\label{cor:mainsum} 
  \sum_{\{(\CE,\infty_{\rho_\kappa})\}/\Hom(\widehat\mu,Z(X,G))(k)}  s( \kappa) q^{-\CF(\kappa)} \# \wmc_{H_{\Ec},a}^{\blacklozenge,t}(k)  =  \sum_{\{(\CE',\infty_{\rho_\kappa})\}/\Hom(\widehat\mu,Z(X,\widehat G))(k)}  t(\nu) q^{-\CF(\nu)} \# \wmc_{{H}_{\nu,\rho'},a}^{\blacklozenge,s}(k),
\end{equation}
where $\CE$ ranges over the coendoscopy data for $G$ over $k$ occurring on $I_{\hat\mu}\widetilde\BM^{\ani}_G(k)$ and $\CE'$ ranges over the coendoscopy data for $\hat G$ over $k$ occurring on $I_{\hat\mu}\widetilde\BM^{\ani}_{\hat G}(k)$.
\end{corollary}

Using the specialisation map recalled above, we have according to Theorem \ref{thm:volume}
$$\vol_G\left({e^{-1}(y)}\right) = \frac{q^{-w(y)}}{|\Aut_{ I_{\widehat{\mu}}\widetilde{\Mc}_{G,a}^{\blacksquare,t}}(y)|}$$
for every $y \in I_{\widehat{\mu}}\widetilde{\Mc}_{G,a}^{\blacksquare,t}(k)$, where we also write $\vol_G$ instead of $\vol_{\widetilde{\Mc}_{G,U_a}^{\blacksquare,t}}$. Theorem \ref{mainidentity} is therefore equivalent to the identity

\begin{equation}\label{uauaua} 
  \int_{\widetilde{M}_{G,U_a}^{\blacksquare,t}(\Oo)} \chi_s \cdot \vol_G =  \int_{\widetilde{M}_{\widehat{G},U_a}^{\blacksquare,s}(\Oo)} \chi_t \cdot \vol_{\widehat{G}}.
\end{equation}
 We will prove this identity by fibrewise integration relative to the Hitchin base. 
\begin{lemma}\label{lemma:relative}
Let $\eta_0$ be a generating section of $\Omega_{\widetilde{\A}}^{\top}$.
\begin{enumerate}[(a)]
\item The line bundle $\Omega^{\top}_{\widetilde{\Mc}^{\text{reg},t}_{G}/\widetilde{\A}}$ on $\widetilde{\Mc}^{\text{reg},t}_{G}$ is trivial.
\item Let $\eta_1 $ be a generating section of $\Omega^{\top}_{\widetilde{\Mc}^{\text{reg},t}_{G}/\widetilde{\A}}$ on $\widetilde{\Mc}^{\text{reg},t}_{G}$. There exists a unique generating section $\omega \in \Omega^{\top}_{\widetilde{\Mc}_G^t}$, such that $\omega|_{\widetilde{\Mc}^{\text{reg},t}_G} = \eta_0 \wedge \eta_1$.
\item For an integrable function $g\colon \widetilde{M}_G^t(\Oo) \to \Cb$, we have
$$\int_{\widetilde{M}_G^t(\Oo)} g \cdot{} \vol_{\widetilde{\Mc}_G^t} = \int_{\widetilde{\A}(\Oo)} |\eta_0| \int_{\widetilde{M}^t_{G,b}(F)} g\cdot{} |\eta_1|.$$
\end{enumerate}
\end{lemma}

\begin{proof}
To simplify notation we drop the superscript ${}^t$ in the rest of the argument. 

Claim (a) follows from \cite[Proposition 4.16.1]{MR2653248} since $\widetilde{\Mc}^{\text{reg}}_G$ is a $\wpc^{\text{reg}}_G$-torsor.

Over the open subset $\widetilde{\Mc}^{\text{reg}} \subset \widetilde{\Mc}$ the Hitchin map $f\colon \widetilde{\Mc}_G \to \widetilde{\A}$ is smooth. Therefore we have a short exact sequence
$$0 \to f^*\Omega_{\widetilde{\A}}^1 \to \Omega_{\widetilde{\Mc}^{\text{reg}}_G}^1 \to \Omega_{\widetilde{\Mc}^{\text{reg}}_G/\widetilde{\A}}^1 \to 0.$$
Passing to top exterior powers we obtain an isomorphism of line bundles 
\begin{equation}\label{eqn:det}
\Omega_{\widetilde{\Mc}^{\text{reg}}_G}^{\top} \simeq \Omega_{\widetilde{\Mc}^{\text{reg}}_G/\widetilde{\A}}^{\top} \otimes \Omega_{\widetilde{\A}}^{\top}.
\end{equation}
On the level of sections this isomorphism sends $\omega_0 \otimes \omega_1 \mapsto \omega_0 \wedge \omega_1$.
Assertion (b) follows from Hartogs's extension property for the line bundle $\Omega^{\top}_{\widetilde{\Mc}_{G}/\widetilde{\A}}$ applied to the open subset $\widetilde{\Mc}^{\text{reg}} \subset \widetilde{\Mc}$. We remark that this is possible since its complement has codimension at least $2$ by \cite[Proposition 4.16.1]{MR2653248}.


We now turn to the proof of (c). Since $\widetilde{M}_G^{\Diamond} \to \widetilde{\A}^\Diamond$ is a smooth morphism of algebraic varieties we can apply \cite[Theorem 7.6.1]{MR1743467} to conclude the identity over $\widetilde{M}_G(\Oo) \cap \widetilde{M}^{\Diamond}_G(F)$:
$$\int_{\widetilde{M}_G(\Oo)\cap \widetilde{M}_G^{\Diamond}(F)} g \cdot{} \vol_{\widetilde{\Mc}_G}= \int_{\widetilde{M}_G(\Oo) \cap \widetilde{M}_G^{\Diamond}(F)} g \cdot{} |\omega|= \int_{\widetilde{M}_G(\Oo) \cap \widetilde{M}_G^{\Diamond}(F)} g \cdot{} |\eta_0 \wedge \eta_1|  = \int_{\widetilde{\A}(\Oo)} d\mu_{\widetilde{\A}} \int_{\widetilde{M}_{G,b}^{\Diamond}(F)} g\cdot{} |\eta_1|,$$
where for the first equality we used Lemma \ref{defcanmes}. By \ref{intools}(a) the complement of $\widetilde{M}_G(\Oo) \cap \widetilde{M}_G^{\Diamond}(F)$ is a set of measure zero, and therefore this implies the claim.
\end{proof}

Now fix $\eta_0$ and $\eta_1$ as in Lemma \ref{lemma:relative} for $\widetilde{M}_{\widehat{G}}^{s}$. Since $\widetilde{M}_{\widehat{G}}^{\text{reg},s}$ is a $\wpc^{\text{reg}}_{\widehat{G}}$-torsor over $\widetilde{\A}^{\text{reg}}$, the section $\eta_1$ induces a generating section $\eta_1'$ of $\Omega^{\top}_{\wpc^{\text{reg}}_{G}/\widetilde{\A}}$, see \cite[Lemma 6.13]{gwz}. By abuse of notation we write $\rho^*\eta_1$ for the generating section of $\Omega^{\top}_{\widetilde{\Mc}^{\text{reg},t}_{G}/\widetilde{\A}}$ obtained from the pullback of $\eta'_1$ under the isogeny $\rho: \wpc_G \to \wpc_{\widehat{G}}$ of Construction \ref{const:ngo_isogeny}.

By virtue of Lemma \ref{lemma:relative} above, \eqref{uauaua} follows if we show for all $b \in U_a(\Oo)^\flat$ the equality
\begin{equation*} \int_{ \widetilde{M}_{G,b}^{\blacksquare,t}(F) } \chi_s \cdot |\rho^*\eta_1| =  \int_{ \widetilde{M}_{\widehat{G},b}^{\blacksquare,s}(F) } \chi_t \cdot |\eta_1|, \end{equation*}
as by Proposition \ref{intools}(a) the complement of the subset $U_a(\Oo)^{\flat}\subset U_a(\Oo)$ is a set of measure $0$ (see also Remark \ref{rmk:zeroset}). We deduce this equality from the following two lemmas, which we will prove in the next subsection.


\begin{lemma}\label{lemma:0}
Let $b \in U_a(\Oo)^{\flat}$, such that the corresponding fibre $\widetilde{M}^{\blacksquare,t}_{G,b}$ does not have an $F$-rational point. Then $\chi_t|_{\widetilde{M}^{\blacksquare,s}_{\widehat{G},b}(F)}$ is a non-trivial character (up to translation), and hence 
$$\int_{ \widetilde{M}_{\widehat{G},b}^{\blacksquare,s}(F) } \chi_t \cdot{} |\eta_1| = 0.$$
\end{lemma}

These statements also play an important role in our article \cite{gwz} devoted to the proof of the Hausel--Thaddeus conjecture, see Proposition 6.16 in \emph{loc. cit}. The lemma below is a generalisation of \cite[Lemma 6.15]{gwz}. In order to keep this article self-contained we include the arguments in Subsection \ref{tate}.

\begin{lemma}\label{lemma:1}
Let $b \in U_a(\Oo)^{\flat}$, such that the corresponding fibres $\widetilde{M}^{\blacksquare,t}_{G,b}$ and $\widetilde{M}^{\blacksquare,s}_{\widehat{G},b}$ have $F$-rational points. Then we have $\chi_t|_{\widetilde{M}^{\blacksquare,s}_{\widehat{G},b}} = 1$ and $\chi_s|_{\widetilde{M}^{\blacksquare,t}_{G,b}} = 1$, and an equality of integrals
$$\int_{ \widetilde{M}_{G,b}^{\blacksquare,t}(F) }  |\rho^* \eta_1|=  \int_{ \widetilde{M}_{\widehat{G},b}^{\blacksquare,s}(F) } |\eta_1|. $$ 
\end{lemma}

\begin{proof}[Proof of Theorem \ref{mainidentity}] We just summarise the discussion above. For every $b \in U_a(\Oo)^{\flat} = U_a(\Oo) \cap \widetilde{\A}(\Oo)^{\flat}$ we have to prove the equality
\begin{equation}\label{eqn:key} \int_{ \widetilde{M}_{G,b}^{\blacksquare,t}(F) } \chi_s \cdot |\rho^* \eta_1| =  \int_{ \widetilde{M}_{\widehat{G},b}^{\blacksquare,s}(F) } \chi_t \cdot |\eta_1|. \end{equation}
If $\widetilde{M}_{G,b}^{\blacksquare,t}(F) = \emptyset$, then by Lemma \ref{lemma:0} the right hand side of \eqref{eqn:key} is $0$ as well. The same argument works in the case $\widetilde{M}_{\widehat{G},b}^{\blacksquare,s}(F) = \emptyset$, so we're left with the case when both Hitchin fibres have a rational point, which is covered by Lemma \ref{lemma:1}
\end{proof} 

\subsection{Character sums}\label{character}

The first consequence of our comparison of $p$-adic integrals for moduli spaces of Higgs bundles is a special case of (the geometric analogue of) Waldspurger's non-standard Fundamental Lemma. At first we introduce some notation. We fix an identification of $\Cb$ with $\bar{\Qb}_{\ell}$. 

Assume that a finite group $\Gamma$ acts on $H^*_c(Y_{\bar{k}},\bar{\Qb}_{\ell})$ where $Y/k$ is a $k$-variety, and that this action commutes with the action of $\Gal(k)$ on \'etale cohomology. For a character $\kappa$ of $\Gamma$ we write $\#^{\kappa}Y(k)$ to denote $\Tr(Fr,H^*_{c}(Y,\bar{\Qb}_{\ell})^{\kappa})$, where $V^{\kappa}$ denotes the $\kappa$-isotypical component of a $\Gamma$-representation $V$. We also use the shorthand $\#^{stab} = \#^{\id}$ for the identity character. The proof of the following lemma is an easy exercise left to the reader.

\begin{lemma}\label{lemma:easystab}
Assume that $Y^{\blacksquare} \subset Y$ is a union of connected components for which $H^*_c(Y^\blacksquare,\bar{\Qb}_{\ell})$ is acted on by a subgroup $\Gamma^{\blacksquare} \subset \Gamma$, in such a way that $\Gamma\cdot{} (H^*_c(Y^\blacksquare,\bar{\Qb}_{\ell})) = H^*_c(Y,\bar{\Qb}_{\ell})$, and $\dim (H^*_c(Y^\blacksquare,\bar{\Qb}_{\ell})) = \frac{|\Gamma|}{|\Gamma^{\blacksquare}|} \cdot{} \dim (H^*_c(Y,\bar{\Qb}_{\ell}))$.

Then $\Ind_{\Gamma^{\blacksquare}}^{\Gamma}H^*_c(Y^\blacksquare,\bar{\Qb}_{\ell}) = H^*_c(Y,\bar{\Qb}_{\ell})$, and hence we have
for every character $\kappa \in \Gamma^{*}$ an equality
$$\#^{\kappa|_{\Gamma^{\blacksquare}}}Y^{\blacksquare}(k) = \#^{\kappa}Y(k).$$
\end{lemma}

\begin{corollary}\label{nonstd} 
For every $a \in \widetilde \A(k)$ we have 
\[\#^{stab} \wmc_{G,a}(k)  =  \#^{stab} \wmc_{\widehat{G},a}(k).\]
\end{corollary}
\begin{proof} 
By virtue of Lemma \ref{lemma:easystab} for $\wmc_{G,a}^{\blacksquare} \subset \wmc_{G,a}$, we have
$$\#^{stab}\wmc_{G,a}^{\blacksquare}(k) = \#^{stab}\wmc_{G,a}(k).$$

We will obtain the identity above by summing up the identities of \eqref{mainsum} for all possible values of $t \in \pi_0(\wpc_{G,a}^{\blacksquare})$ and $s \in \pi_0(\wpc_{\widehat{G},a}^{\blacksquare})$.
On the left hand side we have
$$\sum_{s,t} \text{LHS of \eqref{mainsum}} = \sum_{\kappa} \sum_{s,t}s(\kappa) q^{-\mathcal{F}(\kappa)} \hash \left(I_{\widehat{\hat\mu}}\widetilde{\Mc}^{\blacksquare}_{G,a}(k)\right)_{\kappa} = q^{-\dim \wmc_{G}}\sum_{s,t} \hash \widetilde{\Mc}_{G,a}^{\blacksquare,t}(k),$$
where we use for the second equality sign that $s \mapsto s(\kappa)$ is a non-trivial character for $\kappa \neq 1$ and $\CF(1) = \dim \wmc_G$. Using Lemma \ref{Fourier} we find
$$\sum_{s,t} \hash \widetilde{\Mc}_{G,a}^{\blacksquare,t}(k) = |\pi_0(\Pc_{G,a}^{\blacksquare})|\sum_s \hash^{stab} \widetilde{\Mc}_{G,a}^{\blacksquare}(k)  = |\pi_0(\Pc_{G,a}^{\blacksquare})||\pi_0(\Pc_{\widehat{G}, a}^{\blacksquare})| \hash^{stab} \widetilde{\Mc}_{G,a}^{\blacksquare}(k).$$
The same computation for the right hand side of \eqref{mainsum} yields 
$$\sum_{s,t} \text{RHS of \eqref{mainsum}} = q^{-\dim \wmc_{\widehat{G}}} |\pi_0(\Pc_{G,a}^{\blacksquare})||\pi_0(\Pc_{\widehat{G}, a}^{\blacksquare})| \hash^{stab} \widetilde{\Mc}_{\widehat{G},a}^{\blacksquare}(k).$$
Therefore we conclude $\hash^{stab} \widetilde{\Mc}_{G,a}^{\blacksquare}(k)=\hash^{stab} \widetilde{\Mc}_{\widehat{G},a}^{\blacksquare}(k)$.
\end{proof}

This result on point-counts implies an assertion on the topology of the Hitchin map over $\bar k$: Using Chebotarev density and the fact that $(R\widetilde{f}_*\bar{\Qb}_{\ell})_{\bar k}$ is pure (since $\widetilde{f}$ is proper and $\widetilde{\Mc}$ and $\widetilde{\A}$ are smooth) and semi-simple (by virtue of the Decomposition Theorem) one can deduce from this assertion that the complexes of constructible sheaves $(R\widetilde{f}_{G,*}\bar{\Qb}_{\ell})^{stab}_{\bar k}$ and $(R\widetilde{f}_{\widehat{G},*}\bar{\Qb}_{\ell})^{stab}_{\bar k}$ are equivalent. However, for establishing Waldspurger's non-standard form of the Fundamental Lemma, the above version in terms of point-counts suffices.



The next statement is equivalent to Ng\^o's Geometric Stabilisation Theorem \ref{gmst}. Using purity of the cohomology of Hitchin fibres, it allows one to reconstruct the dimension of the $\lambda$-isotypical component $H^i(\wmc_{G,a},\bar{\Qb}_{\ell})^{\lambda}$, where $\lambda \in \Hom(\Xb_*(\Tb),(\Qb/\Zb)')$ is a character of finite order.
Since the action of $\Xb_*(\Tb)$ on the cohomology of a Hitchin fibre $\wmc_{G,a}$ factors (by definition) through $\pi_0(\wpc_{G,a})$, we see that the $\lambda$-isotypical component $H^i(\wmc_{G,a},\bar{\Qb}_{\ell})^{\lambda}$ is non-zero, if and only if $\lambda$ factors through $\pi_0(\wpc_{G,a})$, and therefore defines an element of $(\pi_0(\wpc_{G,a}))^*$.
Since $\Hom(\Xb_*(\Tb),(\Qb/\Zb)') = \Hom(\widehat{\mu},\widehat{\Tb})$ we may consider coendoscopic data $\Ec$ for $\widehat{G}$ of type $\lambda$.

\begin{corollary}\label{cor:geometric_endoscopy}  
Let $\CE=(\lambda,\rho_\lambda,\rho_\lambda \to \rho)$ be a coendoscopic datum for $\widehat G$ over $k$ for which the torsor $\rho_{\lambda,\infty}$ splits. We also use the notation $\lambda \in \pi_0(\wpc_{G,a})^*$ to denote the image of $\lambda$ in $\pi_0(\wpc_{G,a})^*$.
For every point $a \in \widetilde{\A}_\CE(k)$ we have
\[ \#^{\lambda} \wmc_{G,a}(k) =q^{\dim \wmc_G -\CF(\lambda)}\#^{stab} \wmc_{\widehat{H}_\CE,a}(k). \]
Here $\widehat{H}_\CE$ is the Langlands dual group of ${H}_\CE$, that is, an endoscopy group of $G$. 
\end{corollary}

\begin{proof} 
This time we apply $\sum_{s,t} \lambda^{-1}(t)$ to both sides of \eqref{mainsum}. We get 
$$\sum_{s,t} \lambda^{-1}(t) \cdot{} \text{LHS of \eqref{mainsum}} =  \sum_{\kappa} \sum_{s,t}   \lambda^{-1}(t) s( \kappa) q^{-\CF(\kappa)} \# \left(I\wmc_{G,a}^{\blacksquare,t}(k)\right)_{\kappa} = \sum_{s,t} \lambda^{-1}(t) q^{-\dim \wmc_G} \# \wmc_{G,a}^{\blacksquare,t}(k),$$
where we use that $s \mapsto s(\kappa)$ is a non-trivial character for $\kappa \neq 1$. The right hand side above simplifies further to
$$\sum_{s,t} \lambda^{-1}(t) \cdot{} \text{LHS of \eqref{mainsum}} = q^{-\dim \wmc_G}|\pi_0(\widetilde{\Pc}_{G,a}^{\blacksquare})||\pi_0(\widetilde{\Pc}^{\blacksquare}_{\widehat{G},a})| \cdot{} \hash^{\lambda} \widetilde{\Mc}_{G,a}^{\blacksquare}.$$
Next we evaluate what happens to the right hand side of \eqref{mainsum}. By virtue of \cite[Proposition 6.3.3]{MR2653248}, the coendoscopy datum $\CE$ for $\widehat G$ is the only such datum of type $\lambda$ such that $a \in A_\CE(k)$. Using this we find
$$\sum_{s,t} \lambda^{-1}(t) \cdot{} \text{RHS of \eqref{mainsum}} = \sum_{\nu}  \sum_{s,t} \lambda^{-1}(t) t( \nu) q^{-\CF(\nu)} \# \left(I\wmc_{\widehat{G},a}^{\blacksquare,s}(k)\right)_{\nu} = \sum_{s,t}q^{-\CF(\lambda)} \# \wmc_{{H}_\CE,a}^{\blacklozenge,s}(k),$$
where we use that $t \mapsto \lambda(t)^{-1} t(\nu)$ is a non-trivial character for $\lambda \neq \nu$, and the last equality sign uses the equivalence
$$\wmc_{{H}_\CE,a}^{\blacklozenge,s}(k) \simeq \left(I\wmc_{\widehat{G},a}^{\blacksquare,s}(k)\right)_{\lambda}$$
of Corollary \ref{kappa_IMTheoremTwisted}.
Since we have a surjective morphism $\pi_0(\widetilde{\Pc}_{G,a}^{\blacksquare}) \to \pi_0(\widetilde{\Pc}_{H_{\Ec},a}^{\blacklozenge})$ (see Lemma \ref{lemma:coendo_surjection}), using Lemma \ref{Fourier} this expression simplifies to
$$\sum_{s,t} \lambda^{-1}(t) \cdot{} \text{RHS of \eqref{mainsum}} =  q^{-\CF(\lambda)}|\pi_0(\widetilde{\Pc}_{G,a}^{\blacksquare})|\cdot{}|\pi_0(\widetilde{\Pc}_{\widehat{G},a}^{\blacksquare})|\cdot{} \#^{stab} \wmc_{{H}_\CE,a}^{\blacklozenge}(k).$$
Using the fact that $\wmc_{H_{\Ec}}^{G}$ is a gerbe over $\wmc_{H_{\Ec}}$ banded by $Z(X,H_{\Ec})/Z(X,G)$, Remark \ref{rmk:cohomology} and Lemma \ref{lemma:easystab} imply $ \hash^{stab} \widetilde{\Mc}_{H_{\Ec},a}^{\blacklozenge}(k)=\hash^{stab} \widetilde{\Mc}_{H_{\Ec},a}(k)$. Lemma \ref{lemma:easystab} and Corollary \ref{nonstd} applied to $H_{\CE}$ yield $\#^{\lambda} \wmc_{G,a}(k) =q^{\dim \wmc_G-\CF(\lambda)}\#^{stab} \wmc_{\widehat{H}_\CE,a}(k)$.
\end{proof}
As before, one can base change to $\bar{k}$ and show by applying Chebotarev density and the Decomposition Theorem that 
the complexes of constructible sheaves $(R\widetilde{f}_{G,*}\bar{\Qb}_{\ell})^{\lambda}_{\bar k}$ and $(R\widetilde{f}_{\widehat{H}_{\CE},*}\bar{\Qb}_{\ell})^{stab}_{\bar k}$ are equivalent. For the proof of the Fundamental Lemma (see \cite[Section 8]{MR2653248}) this equality of complexes is not needed. It suffices to prove the equality of point-counts
$$\#^{\lambda} \wmc_{G,a}(k) =q^{\dim \wmc_G -\CF(\lambda)}\#^{stab} \wmc_{\widehat{H}_\CE,a}(k).$$
The arguments given above establish this equality without relying on the theory of perverse sheaves. Apart from $p$-adic integration, the main ingredient (to be discussed subsequently) is \emph{Tate duality}.

\begin{rmk}
The assumption of Corollary \ref{cor:geometric_endoscopy} is precisely the one used in \cite[8.6.3]{MR2653248} to deduce the Fundamental Lemma from Geometric Stabilisation. In \emph{loc. cit.} the assumption is formulated for the endoscopy group $\widehat{H}_{\CE}$, but this is equivalent to our assumption for the Langlands dual group $H_{\CE}$ stated in Theorem \ref{cor:geometric_endoscopy}.
\end{rmk}

In order to conclude our proof of the Geometric Stabilisation Theorem we have to prove the following identity for the constant $\mathcal{F}(\kappa)$ appearing in Corollary \ref{cor:geometric_endoscopy}. For a co-endoscopy group $H_{\CE}$ of  $G$ of type $\kappa$ 
we set $r^G_{H_{\CE}}(D)=\frac{1}{2}(\dim \widetilde{\Mc}_{G} - \dim \widetilde{\Mc}_{H_{\CE}} )$.

\begin{lemma} We have 
\[  \dim \wmc_G - \mathcal{F}(\kappa)= r^G_{H_{\CE}}(D). \]
\end{lemma}

\begin{proof} Let $a \in \wac_{\CE}(k) \subset \wac(k)$. The fibre $\widetilde{\Mc}_{G,a}$ is a projective $k$-scheme with set of irreducible components $\pi_0(\wpc_{G,a})$, see \cite[4.16.3]{MR2653248}. Furthermore all irreducible components have the same dimension $d_G $ by \cite[4.16.1]{MR2653248}. Therefore we see that $H^{2d}(\widetilde{\Mc}_{G,a},\bar{\Qb}_{\ell}) \simeq \bar{\Qb}_{\ell}^{\pi_0(\wpc_{G,a})}(d_G)$. The finite group $\pi_0(\wpc_{G,a})$ acts via the regular representation and therefore we conclude that 
$$H^{2d}(\widetilde{\Mc}_{G,a},\bar{\Qb}_{\ell})^{\kappa} \simeq \bar{\Qb}_{\ell}(d_G)$$
and hence obtain the asymptotic estimate $\#^{\kappa} \widetilde{\Mc}_{G,a}(\BF_q) \sim q^{d_G}$ as $q \to \infty$ for finite overfields $\BF_q$ of $k$.

The same argument shows that we have $\#^{stab}\widetilde{\Mc}_{H_{\CE}}(\mathbb{F}_q) \sim q^{d_{H_{\CE}}}$. By virtue of Corollary \ref{cor:geometric_endoscopy} we have $\#^{\kappa}\widetilde{\Mc}_{G,a}(\mathbb{F}_q) = q^{\dim \wmc_G - \mathcal{F}(\kappa)} \#^{stab} \widetilde{\Mc}_{H_{\CE},a}(\mathbb{F}_q)$. The asymptotic considerations therefore show that $\dim \wmc_G - \mathcal{F}(\kappa) = d_G - d_{H_{\CE}} = r^G_{H_{\CE}}(D)$.
\end{proof}

This concludes the reduction of Ng\^o's Geometric Stabilisation Theorem \ref{gmst} to Theorem \ref{mainidentity}. The next section will be devoted to proving this result.

\subsection{Tate duality}\label{tate}

Given a Beilinson $1$-motive $\Pc$ over base scheme $S$ (see Definition \ref{beil}), there exists an interesting correspondence between $\mathsf{Ext}^2_S(\Pc,\G_m)$ and torsors over $\Pc^{\vee}$ (later also referred to as \emph{twists} of $\Pc^{\vee}$). We briefly recall this correspondence.

A torsor $\Tt$ over $\Pc^{\vee}$ gives rise to a short exact sequence
\begin{equation}\label{eqn:tor}0 \to \Pc^{\vee} \to \bigsqcup_{n \in \Zb} \Tt^{\otimes n} \to \Zb \to 0.\end{equation}
Dualising it, we obtain a short exact sequence of abelian group stacks
\begin{equation}\label{eqn:ext}0 \to B\G_m \to \Gg \to \Pc \to 0,\end{equation}
where we used $\Pc^{\vee \vee} \simeq \Pc$. This short exact sequence represents an element of $\mathsf{Ext}^2_S(\Pc,\G_m)$. Vice versa, one dualises \eqref{eqn:ext} to obtain a short exact sequence as in \eqref{eqn:tor}. It is clear that this defines a bijection between isomorphism classes (and in fact an equivalence of groupoids).

In the case that $\Pc = A$ is an abelian variety and the base $S = \Spec F$ where $F$ is a local non-archimedean field, this construction gives rise to the Tate pairing
$$A(F) \times H^1(F,A^{\vee}) \to \Qb/\Zb$$
as follows: a torsor $t \in H^1(F,A^{\vee})$ corresponds to $\alpha_t \in \mathsf{Ext}^2_F(A,\G_m)$ which induces a gerbe $\alpha_t \in Br(A)$. One defines the value of the Tate pairing at $(x,t) \in A(F) \times H^1(F,A^{\vee})$ to be the Hasse invariant of the gerbe $x^*\alpha_t$ on $F$. According to a theorem of Tate, this is a perfect pairing (see \cite[I.3.4]{MilneADT}). We refer to this as \emph{Tate Duality}.

This subsection is devoted to the study of a technical aspect of this story for (regular) Hitchin fibres. It can be skipped on a first reading of this text.

\begin{goal}\begin{enumerate}[(1)]
\item We will construct a $\widehat{\Tb}/Z(X,\widehat{G})$-gerbe $\alpha$ on $\widetilde{\Mc}_{\widehat{G}}^{s}$, such that for every $x \in \widetilde{M}_{\widehat{G}}^{s}(\Oo)^{\natural}$ we have for $\lambda \in \Xb^*(\widehat{\Tb}/Z(X,\widehat{G})) \subset \Xb_*(\Tb) \to \Xb_*(\Tb/Z(X,G))$ that the Hasse invariant of $\lambda(x_F^*\alpha) \in Br(F)$ equals $\chi_{\lambda}(e(x))$. See Lemma \ref{lemma:chi_gerbe}.
\item Subsequently we will verify that $\lambda(\alpha)$ corresponds to the twist $\widetilde{\Mc}_G^{\blacksquare,t}$ with respect to Tate duality. See Lemma \ref{lemma:chi_gerbe}.
\end{enumerate}
\end{goal}

Recall from Construction \ref{construction:ngo_higgs} that for a $k$-scheme $S$, an $S$-family of $G$-Higgs bundles on $X_k$ corresponds to a section of the map $[\mathfrak{g}_D/G] \to X \times S$. Let $\infty \in X$ be our marked point on the curve $X$, we fix an isomorphism $(\mathfrak{g}_D)_{\infty} \simeq \mathbf{g} = \mathsf{Lie}(G)$ and obtain an evaluation morphism $\ev_{\infty}\colon \mathbb{M}_G \to [\mathbf{g}/\Gb]$.

\begin{construction}\label{const:T-torsor}
There is a $\Tb$-torsor on $\widetilde{\mathbb{M}}_G(X_k)$, induced by the morphism $\widetilde{\mathbb{M}}_G(X_k) \to B\Tb$ of the diagram below
\[
\xymatrix{
\widetilde{\mathbb{M}}_G(X_k) \ar[r] \ar[d] & [\mathbf{t}^{rs}/\Tb] \ar[d] \ar[r]^{\simeq} & \mathbf{t}^{rs} \times B_k\Tb \ar[r] & B_k\Tb \\
\mathbb{M}_G(X_k) \ar[r]^{\ev_{\infty}} & [\mathbf{g}/\Gb]. & &
}
\] 
By virtue of Definition \ref{defi:tilde} we have $\widetilde{\A} = \A \times_{\mathbf{c}} \mathbf{t}^{rs}$. Since $\widetilde{\mathbb{M}} = \mathbb{M} \times_{\A} \widetilde{\A}$ we obtain map $\widetilde{\mathbb{M}}_G \to [\mathbf{t}^{rs}/\Tb] = [\mathbf{g}^{rs}/\Gb] \times_{\mathbf{c}} \mathbf{t}^{rs}$.
\end{construction}

\begin{definition}\label{defi:T-torsor}
Assume that we are in Situation \ref{situation:split}. The $\Tb$-torsor on $\widetilde{\mathbb{M}}_G(X_k)$ represented by the morphism $\widetilde{\mathbb{M}}_G(X) \to B_k\Tb$ of Construction \ref{const:T-torsor} will be denoted by $\mathcal{T}_G \to \widetilde{\mathbb{M}}$. By quotienting out the centre $Z(X,G) \subset \Tb$ we obtain a $\Tb/Z(X,G)$-torsor on $\widetilde{M}_G(X_k)$ which we also denote by $\mathcal{T}_G$.
\end{definition}

In fact, a similar construction is used in Ng\^o's description of automorphism groups of Higgs bundles which we recalled in \ref{cons:auto_torus}. 

\begin{rmk}\label{rmk:auto}
For $(E,\theta,\widetilde{\infty}) \in \widetilde{\mathbb{M}}_G^{\ani}(X_k)$ we have that the map $\Aut(E,\theta) \to \Tb$ induced by $\widetilde{\mathbb{M}}_G^{\ani}(X_k) \to B_k\Tb$ is the embedding described in Lemma \ref{cons:auto_torus}.
\end{rmk}

In the following definition we make use of the fact that a morphism of abelian group stacks $\mathcal{H} \to BA$ gives rise to an extension of abelian group stacks 
$$1 \to A \to \widetilde{\mathcal{H}} \to \mathcal{H} \to 1,$$
where $\widetilde{\mathcal{H}}$ is defined as the fibre product $\mathcal{H} \times_{BA} \{1\}$. This is standard, which amounts to the well-known description of morphisms in the derived category of an abelian category $\C$ in terms of Ext-groups 
$$\Hom_{D(\C)}(B,A[1]) \simeq \mathsf{Ext}_{\C}(B,A).$$
In the case of abelian group stacks, it can be verified directly from the definitions of $BA$ as classifying $A$-torsors. We refer the reader to \cite[2.A4]{MR1988970} for a similar construction for groupoids (without the abelian assumption).
\begin{definition}\label{defi:T-torsor2}
We assume that we are in Situation \ref{situation:split}. Let $\Pc_G \to \Bun_{T}(\Cc/\A)$ be the morphism which forgets the $W$-equivariant and $+$-structures (see Definition \ref{defi:donagi-gaitsgory} and Theorem \ref{thm:dg}). The marked point $\widetilde{\infty} \in \Cc^{sm}$ yields a morphism $\widetilde{\Pc}_G \to \Bun_{\Tb}(\Cc/\A) \to B_k\Tb$ of abelian group stacks. We denote the corresponding extension of abelian group stacks by 
$$1 \to \Tb \to \mathcal{T}_{\Pc} \to \widetilde{\Pc}_G \to 1.$$
\end{definition}

\begin{rmk}
The action of the Prym $\widehat{\Pc}_G$ on $\widehat{\mathbb{M}}_G$ lifts to an action of $\mathcal{T}_{\Pc}$ on the $\Tb$-torsor $\mathcal{T}_G$ over $\widetilde{\mathbb{M}}_G$. 
\end{rmk}

This remark has an important consequence, it implies that the restriction of $\mathcal{T}_G|_{\widehat{\Mc}_{G,a}}$ is translation invariant.

\begin{corollary}
Let us be in Situation \ref{situation:split} and $y \in \Pc_{G,a}(k)$. We denote by $\psi_y\colon \widetilde{\Mc}_{G,a} \to \widetilde{\Mc}_{G,a}$ the induced automorphism given by acting through $y$. Then $\psi_y^*\mathcal{T}_G$ is non-canonically isomorphic to $\mathcal{T}_G$.
\end{corollary}

This implies that $\mathcal{T}_G$ descends automatically to a $\Tb/Z(X,G)$-torsor on the unramified twist $\widetilde{\Mc}_{G,a}^{t}$. Indeed, the obstruction to the existence of a Galois descent datum is given by a $2$-cocycle with values in $H^2_{\text{\'et}}(k,\Tb/Z(X,G)) = Br(k) \otimes \Xb_*(\Tb/Z(X,G))$. Since the Brauer group of a finite field vanishes, so does the obstruction and a Galois-descent datum exists. Furthermore, it is unique (up to a unique isomorphism) since $H^1_{\text{\'et}}(k,\Tb/Z(X,G)) = H^1_{\text{\'et}}(k,\G_m) \otimes \Xb_*(\Tb/Z(X,G)) = 0$.

\begin{definition}
We denote the $\Tb/Z(X,G)$-torsor on $\widetilde{\Mc}_{G,a}^t$ defined above by $\mathcal{T}_G^t$.
\end{definition}

Recall from Theorem \ref{thm:duality_quasi-split} that we have an equivalence of abelian group stacks over $\A^{\Diamond}$ 
$$\Hhom(\Pc_G^{\Diamond},B\G_m) \simeq \Pc_{\widehat{G}}^{\Diamond}.$$
The classifying map $\widetilde{\Pc}_G^{\Diamond}\to B\Tb$ of the torsor $\Tt_G$ therefore yields a section of 
\begin{equation}\label{eqn:dual}
  \widetilde{\Pc}_{\widehat{G}}^{\Diamond} \otimes \Xb_*(\Tb) \simeq \Hhom(\Xb_*(\widehat{\Tb}),\widetilde{\Pc}_{\widehat{G}}^{\Diamond}).
\end{equation}

\begin{lemma}\label{lemma:dual_torsor}
The map $\Xb_*(\widehat{\Tb}) \to \widetilde{\Pc}^{\Diamond}_{\widehat{G}}$ of \eqref{eqn:dual} agrees with the map \eqref{ajconst}.
\end{lemma}

\begin{proof}
The isomorphism $(\Pc_G^{\Diamond})^{\vee} \simeq \Pc^{\Diamond}_{\widehat{G}}$ is induced by pullback along the Abel-Jacobi map
\begin{equation*}
  \AJ_{{G}}^{\Diamond} \colon \Cc^\Diamond \times_X \Xb_*({T})\to \Pc_{{G}}^{\Diamond}
\end{equation*} 
from Construction \ref{const:AJ}.

We recall the definition of the Abel-Jacobi map in slightly different wording, and refer the reader to Construction \ref{const:AJ} for more details. The Abel-Jacobi map is induced by a strongly $W$-equivariant $T$-torsor on $\Cc^{\Diamond} \times_{\A^{\Diamond}} \Cc^{\Diamond} \times_X \Xb_*({T})$ whose fibre over a local section $\lambda\colon U \to \Xb_*({T})$ is given by 
$W$-equivariantisation of $\Oo(\lambda \Delta),$ where $\Delta$ denotes the diagonal divisor. \'Etale-locally on $U$ the underlying $T$-torsor can be described as
$$\bigotimes_{w \in W} \Oo(\lambda \Delta)^w,$$
it is endowed with a tautological $+$-structure.

By definition of the ${\Tb}$-torsor $\mathcal{T}_{\Pc}/\widetilde{\Pc}_{{G}}$,  we have that $(\AJ_{{G}}^\Diamond)^*\mathcal{T}_{\Pc}$ is given by restricting the torsor above along 
$$\id_{\widetilde{\Cc}} \times \widetilde{\infty} \times \id_{\Xb_*(T)} \colon \widetilde{\Cc}^{\Diamond}   \times_X \Xb_*({T}) \hookrightarrow \widetilde{\Cc}^{\Diamond} \times_{\A^{\Diamond}} \widetilde{\Cc}^{\Diamond} \times_X \Xb_*({T}).$$
Hence the pullback $(\AJ_{{G}}^{\Diamond})^*\Tt_{\Pc}$ is given by $W$-equivariantisation of $\Oo(\lambda \widetilde{\infty})$ with its tautological $+$-structure, where this time we denote by $\widetilde{\infty} \subset \widetilde{\Cc}$ the divisor given by the image of the section $\widetilde{\infty} \colon \widetilde{\A} \to \widetilde{\C}$.

For every \'etale-local section $\nu$ of $\Xb^*(T)$ we obtain therefore a $\G_m$-torsor $\nu((\AJ_{{G}}^{\Diamond})^*\Tt_{\Pc})$ which is given by the $W$-equivariantisation of $\Oo(\langle \lambda, \nu \rangle \widetilde{\infty})$ with the tautological $+$-structure. This yields a $\G_m$-torsor $\langle (\AJ_{{G}}^{\Diamond})^*\mathcal{T}_{\Pc} \rangle$ on 
$$\widetilde{\Cc}^{\Diamond}   \times_X \Xb_*({T}) \times_X \Xb^*(T).$$
By virtue of the isomorphism $\Xb_*(T) \simeq \Xb^*(\widehat{T})$, we can view it as induced by a strongly $W$-equivariant $\widehat{T}$-torsor on
$$\widetilde{\Cc}^{\Diamond}   \times_X \Xb_*(\widehat{T})$$
with $+$-structure, given by $W$-equivariantisation of $\Oo(\nu \widetilde{\infty})$ for every \'etale-local section $\nu$ of $\Xb_*(\widehat{T}) = \Xb^*(\Tb)$. This corresponds precisely to the section of $\AJ_{\widehat{G}}(\widetilde{\infty})\colon \widetilde{A} \times \Xb_*(\widehat{T}) \to \widetilde{\Pc}_{\widehat{G}}$ of Definition \eqref{ajconst} and hence concludes the proof.
\end{proof}

We have seen in Definition \ref{defi:unramified_gerbe} that for varieties over finite fields (or ring of integers of local fields), every torsor $\mathcal{T}$ of finite order gives rise to a gerbe. The reason for this is that $\pi_1^{\text{\'et}}(\Spec k) \cong \widehat{\mathbb{Z}}$ is topologically generated by the Frobenius automorphism. By restricting the $\Tb/Z(X,G)$-torsor to a Hitchin fibre $\M_{G,a}$ we show first that we obtain a torsor of finite order, and hence can apply the construction.

\begin{lemma}
The restriction of the $\Tb/Z(X,G)$-torsor $\mathcal{T}_G$ to a Hitchin fibre $\widetilde{\Mc}_{G,a}$ is of finite order.
\end{lemma}

\begin{proof}
We will deduce this assertion from the following claim. We denote by $\Pic_{\Tb/Z(X,G)}$ the relative (unsheafified) Picard functor.
\begin{claim}
There exists a morphism $\widetilde{\Pc}_{\widehat{G}} \otimes {\Xb_*(\Tb/Z(X,G))} \to \Pic_{\Tb/Z(X,G)}(\widetilde{\Mc}_G/\widetilde{\A})$, such that for $a \in \widetilde{A}(k)$ the torsor $\mathcal{T}_G|_{\widetilde{\Mc}_{G,a}}$ agrees with the image of a $k$-rational point of $\widetilde{\Pc}_{\widehat{G},a}$ with respect to this map. 
\end{claim}
The lemma above follows from the claim, since $\pi_0(\widetilde{\Pc}_{\widehat{G},a})$ is a torsion group, and since the neutral connected component of $\widetilde{\Pc}_{\widehat{G},a}$ is an abelian variety and hence has only finitely many $k$-rational points. It therefore suffices to prove the claim.

Since $\Tb/Z(X,G)$ is a torus, every $\Tb/Z(X,G)$-torsor on an open subspace $U \subset Y$ of a smooth algebraic space $Y$ extends to $Y$. This extension is unique up to tensor product by $\Oo_Y(E)$, where $E$ is a divisor on $Y$ set-theoretically supported on the complement of $U$. Since $\widetilde{\Mc}_G$ is a smooth DM-stack which is an algebraic space up to codimension $2$ by Corollary \ref{cor:codimension2}, we conclude that the same property holds for the open immersion 
$$\widetilde{\Mc}^{\Diamond}_G \subset \widetilde{\Mc}_G.$$
Furthermore, the Hitchin map $\widetilde{f}\colon \widetilde{\Mc}_G \to \widetilde{\A}$ is flat (\cite[Corollaire 4.16.4]{MR2653248}) and therefore we obtain an isomorphism
$$\Pic_{\Tb/Z(X,G)}(\widetilde{\Mc}_G/\widetilde{\A})(V) \simeq \Pic_{\Tb/Z(X,G)}(\widetilde{\Mc}^{\Diamond}_G/\widetilde{\A}^{\Diamond})(V^{\Diamond})$$
for every smooth morphism $V \to \widetilde{\A}$.
In particular we can apply this to $V = \widetilde{\Pc}_{\widehat{G}} \otimes \Xb_*(\Tb/Z(X,G))$, and use the above isomorphism to uniquely extend the $\mathbb{T}/Z(X,G)$-torsor on 
$$\widetilde{\Mc}^{\Diamond}_G \times_{\widetilde{\A}} \widetilde{\Pc}_{\widehat{G}} \otimes \Xb_*(\Tb/Z(X,G))$$
which is induced by the Poincar\'e line bundle on $\widetilde{\Mc}^{\Diamond}_G \times_{\widetilde{\A}} \widetilde{\Pc}_{\widehat{G}}$. 

We therefore obtain a $\Tb/Z(X,G)$-torsor on the product $\widetilde{\Mc}_G \times \widetilde{\Pc}_{\widehat{G}} \otimes \Xb_*(\Tb/Z(X,G))$ and hence a morphism $$\widetilde{\Pc}_{\widehat{G}} \otimes {\Xb_*(\Tb/Z(X,G))} \to \Pic_{\Tb/Z(X,G)}(\widetilde{\Mc}_G/\widetilde{\A}).$$
It follows from Lemma \ref{lemma:dual_torsor} that $\mathcal{T}_G$ is the image of a $k$-rational point.
\end{proof}

\begin{definition}\label{defi:alpha'}
Let $\alpha'_a \in H^2_{\text{\'et}}(\widetilde{\mathcal{M}}^t_{G,{a}},\Tb/Z(X,G))$ be the gerbe on $\widetilde{\Mc}_{G,a}$ which is induced by the element of $H^1_{\text{\'et}}(k,H^1_{\text{\'et}}(\widetilde{\mathcal{M}}^t_{G,{a}},\Tb/Z(X,G)))$ which sends the topological generator $1 \in \widehat{\Zb}\cong \Gal(\bar k/k)$ to $[\mathcal{T}^t_{G}] \in H^1_{\text{\'et}}(\widetilde{\mathcal{M}}^t_{G,{a}},\Tb/Z(X,G))$. We decompose $\alpha'_a = \alpha_a \alpha''_a$ where $\alpha_a$ is of order coprime to $p$ and the order of $\alpha''_a$ is a $p$-power.
\end{definition}

Recall that $U_a$ denotes $\Spec \Oo^h_{\A,a}$. For every integer $N$ coprime to $p$ we have a natural isomorphism $H^2_{\text{\'et}}(\widetilde{\Mc}_{G,U_a},\mu_N) \simeq H^2_{\text{\'et}}(\widetilde{\Mc}_{G,a},\mu_N)$ (as a consequence of proper base change \cite[Expos\'e XVI]{SGA43}).

\begin{definition}\label{defi:alpha}
We define $\alpha_{U_a} \in H^2_{\text{\'et}}(\widetilde{\Mc}^t_{G,U_a},\mu_N) \simeq H^2_{\text{\'et}}(\widetilde{\Mc}^t_{G,a},\mu_N)$  as the element corresponding to $\alpha_a$ of Definition \ref{defi:alpha'}.
\end{definition}

\begin{lemma}\label{lemma:chi_gerbe}
Let $x \in \widetilde{M}_{G,U_a}(\Oo)^{\natural}$ and $\lambda \in \Xb^*(\Tb/Z(X,G))$. The Hasse invariant of $\lambda(x_F^*\alpha_{U_a}) \in Br(F)$ equals $\chi_{\lambda}(e(x))$. 
\end{lemma}

\begin{proof}
This follows directly from Corollary \ref{cor:hasse} and the claim that the function 
$$f_{\lambda(\Tc_G)}\colon I_{\widehat{\mu}}\widetilde{\Mc}_G(k) \to \Cb$$ 
induced by the $\G_m$-torsor $\lambda(\Tc_G)/\widetilde{\Mc}_G$ agrees with $\chi_{\lambda}$. To see this, recall the definition of $f_{\lambda(\Tc_G)}$: a $k$-rational point $y$ of $I_{\widehat{\mu}}\widetilde{\Mc}_G$ gives rise to a map $B\widehat{\mu} \to \widetilde{\Mc}_G$. Post-composing this morphism with the morphism $\widetilde{\Mc}_G \to B\G_m$ corresponding to $\lambda(\Tc_G)$, we obtain a morphism of stacks $B\widehat{\mu} \to B\G_m$. The latter corresponds to a group homomorphism $\widehat{\mu} \to \G_m$, corresponding to $\lambda(\kappa)$ of $(\Qb/\Zb)'$, where $\kappa$ is the element corresponding to $y \in I_{\widehat{\mu}}\widetilde{\Mc}_G$ according to Theorem \ref{IMTheoremTwisted} (see also Remark \ref{rmk:auto}).
\end{proof}


\begin{proof}[Proof of \ref{lemma:0}]
Recall what we would like to show: let $b \in U_a(\Oo)^{\flat}$, such that the corresponding fibre $\widetilde{M}^{\blacksquare,t}_{G,b}$ does not have an $F$-rational point. Then $\chi_t|_{\widetilde{\Mc}^{\blacksquare,s}_{\widehat{G},b}(F)}$ is a non-trivial character (up to translation).

There are two cases, if $\Mc^{\blacksquare,s}_{\widehat{G},b}$ does not have an $F$-rational point there is nothing to show.
If an $F$-rational point exists, then the twist $\Mc^{\blacksquare,s}_{\widehat{G},b}$ can be identified with $\Mc^{\blacksquare}_{\widehat{G},b}$, respectively the Prym $\Pc_{G,b}^{\blacksquare}$. 

It follows from Lemma \ref{lemma:dual_torsor} that $\alpha_{U_a} \in \Ext_{\text{\'et}}^2(\widetilde{\Pc}_{G,b}^{\Diamond,\blacksquare},\Xb_*(\Tb/Z(X,G)) \otimes \G_m) = \Ext^1_{\text{\'et}}(\Xb_*(\widehat{\Tb}/Z(X,G)), \widetilde{\Pc}_{\widehat{G},b}^{\Diamond,\blacksquare})$ corresponds to the unramified twists $\widetilde{\Mc}^{\blacksquare,t}_{\widehat{G},b}$. Recall the correspondence between twists and gerbes on abelian varieties described at the beginning of Subsection \ref{tate}.

According to Tate Duality \cite[I.3.4]{MilneADT} we have a perfect pairing
$$\widetilde{\Pc}^{\blacksquare}_{{G},b}(F) \times H^1_{\text{\'et}}(F,\widetilde{\Pc}^{\blacksquare}_{\widehat{G},b}) \to \Qb/\Zb,$$ which sends $(y,t)$ to $\inv^*(y^*\alpha_t)$, where $\alpha_t \in \Ext^2_{\text{\'et}}(\widetilde{\Pc}^{\blacksquare}_{\widehat{G},b},\G_m)$ denotes the corresponding gerbe. In particular, for every non-trivial $t \in H^1_{\text{\'et}}(F,\widetilde{\Pc}^{\blacksquare}_{\widehat{G},b})$, the Hasse invariant of the gerbe $\alpha_t|_{\widetilde{\Pc}^{\blacksquare}_{{G},b}}$ is a non-trivial character.

Lemma \ref{lemma:chi_gerbe} now provides the finishing touch: the Hasse invariant of $\alpha_{U_a}|_{\widetilde{\Pc}^{\blacksquare}_{{G},b}}$ is given by the function $\chi_t$, which proves exactly what we wanted.
\end{proof}

\begin{proof}[Proof of Lemma \ref{lemma:1}]
The proof of the previous lemma shows that the function $\chi_t|_{\widetilde{\Mc}^{\blacksquare,s}_{\widehat{G},b}(F)}$ is constant and equal to $1$.
The isogeny of Construction \ref{const:ngo_isogeny} induces an isogeny of abelian varieties $\Pc_G^{\Diamond,\blacksquare} \to \Pc_{\widehat{G}}^{\Diamond,\blacksquare}$ of degree coprime to the characteristic of $k$. Lemma \ref{lemma:1} is now a consequence of Proposition \ref{abvol}. 
\end{proof}


\appendix

\section{Appendix}

\subsection{Co-descent}

As a preparation to extend the duality result of Donagi--Pantev and Chen--Zhu to outer twists of reductive group schemes $G$, we develop a \emph{co-descent} description of the Prym $\Pb$ in this subsection. As a warm up we discuss descent for Prym varieties. Subsequently we freely use standard terminology for simplicial objects, as explained in \cite[Tag 0169]{stacks-project}. We also remark that we will depict simplicial objects as semi-simplicial objects (ignoring degeneracy maps). 

\begin{definition}\label{defi:descent}
Let $Y \xrightarrow{f} X$ be a finite \'etale morphism of smooth proper curves. We fix a quasi-split reductive group scheme $G/X$, and by abuse of notation also denote by $G/Y$ its pullback to $Y$.
\begin{enumerate}[(a)]
\item We denote by $\Res_{Y/X}\colon \Pb_{G,X} \to \Pb_{G,Y}$ the natural map induced by pullback of $J$-torsors along the map $f$.
\item By forming base changes of the map $f$ along itself iteratively, we obtain an augmented simplicial object $Y_{\bullet}$ in the category of finite \'etale schemes over $X$:
\[
\xymatrix{
X & Y \ar[l] & Y\times_X Y \ar@<-0.7ex>[l] \ar@<0.7ex>[l] & Y \times_X Y \times_X Y \ar@<-1.4ex>[l] \ar[l] \ar@<1.4ex>[l] & \cdots \ar@<-0.7ex>[l] \ar@<0.7ex>[l] \ar@<-2.3ex>[l] \ar@<2.3ex>[l]  
}
\]
\item Pullback of $J$-torsors yields an augmented co-simplical object in strict Picard stacks
\[
\xymatrix{
\Pb_{X} & \Pb_{Y} \ar@{<-}[l]_{\Res} & \Pb_{Y\times_X Y} \ar@{<-}@<-0.7ex>[l]_-{\Res} \ar@{<-}@<0.7ex>[l]^-{\Res} & \Pb_{Y \times_X Y \times_X Y} \ar@{<-}@<-1.4ex>[l]_-{\Res} \ar@{<-}[l] \ar@{<-}@<1.4ex>[l]^-{\Res} & \cdots \ar@{<-}@<-0.7ex>[l] \ar@{<-}@<0.7ex>[l] \ar@{<-}@<-2.3ex>[l]_-{\Res} \ar@{<-}@<2.3ex>[l]^-{\Res}  .
}
\]
\end{enumerate}
\end{definition}

Faithfully flat descent theory applied to the surjective \'etale morphism $f\colon Y \to X$ above yields the following statement.

\begin{lemma}[Descent]\label{lemma:descent}
The augmented co-simplicial diagram of Definition \ref{defi:descent}(c) induces an equivalence $\Res\colon \Pb_X {\simeq} |\Pb_{Y_{\bullet}}|$.
\end{lemma}

In order to treat co-descent we have to define \emph{norm maps} for \'etale sheaves of abelian groups on schemes. This will give rise to morphisms $\Nm_{Y/X}\colon \Pb_Y \to \Pb_X$, and to a simplicial object analogous to the one of Definition \ref{defi:descent}(b).
We begin with an elementary lemma about finite \'etale morphisms.

\begin{lemma}\label{lemma:trivialisation}
Let $Y \xrightarrow{f} X$ be a surjective finite \'etale morphism of schemes of finite presentation over $\mathbb{Z}$. Then there exists a finite \'etale morphism $g\colon Z \to X$, such that the base change $Y \times_X Z \to Z$ is a trivial finite \'etale morphism. That is, there exists a finite set $F$, such that $Y \times_X Z \to Z$ is isomorphic to the morphism $\bigsqcup_{i \in F} Z \to Z$.
\end{lemma}

\begin{proof}
Without loss of generality we may assume that $X$ is non-empty and connected. Let $d$ be the degree of the map $f$. We prove the assertion by induction on $d$.

The base case $d = 1$ is true, as $f$ is then an isomorphism and we can choose $g=\id_X$. We assume that the assertion is already known to be true for finite \'etale morphisms of degree $d-1$. The diagonal $Y \to Y \times_X Y$ is open, since $f$ is \'etale \cite[Tag 02GE]{stacks-project}, and closed as $f$ is finite (and thus separated). Therefore, the diagonal morphism corresponds to the inclusion of a (union of) connected component $Y \hookrightarrow Y \times_X Y$. We obtain a finite \'etale map $U= Y \times_X Y \setminus Y \to Y$ of degree $d-1$.

By the inductive hypothesis there exists $g'\colon Z' \to Y$, a finite \'etale morphism, such that $U \times_Y Z' \to Z'$ is a trivial finite \'etale morphism. We define $g\colon Z=Z' \times_X Y \to X$ and observe that $Y \times_X Z \to Z$ is a trivial finite \'etale morphism.
\end{proof}

\begin{definition}\label{defi:norm}
Let $J$ a sheaf of abelian groups on the small \'etale site of $X$. Let $Y \xrightarrow{f} X$ be a surjective finite \'etale morphism. We denote by $J_Y$ the sheaf on the small \'etale site of $X$ given by $f_* f^{-1}J$.
\begin{enumerate}[(a)]
\item If $f$ is trivial, that is, there exists a finite set $F$, such that $f$ is isomorphic to the map $\bigsqcup_{i \in F} X \to X$, then we define $\Nm_{Y/X}\colon J_Y \to J_X$ to be the multiplication map
$$J^F \to J$$
which sends local sections $(s_i)_{i \in F}$ to $\prod_{i \in F} s_i$.
\item We say that a morphism of \'etale sheaves $\mathsf{nm}\colon J_Y \to J_X$ is a norm map, if for every map $g$ as in Lemma \ref{lemma:trivialisation} we have $g^*\mathsf{nm} = \Nm_{Y \times_X Z / Z}$.
\end{enumerate}
\end{definition}

Norm maps as in Definition \ref{defi:norm} always exist, and furthermore are unique.

\begin{construction}[Norm maps]
Let $f\colon Y \to X$, $J$ and $J_Y$ be as in Definition \ref{defi:norm}. We define a norm map $\Nm_{Y/X} \colon J_Y \to J$ as follows. Let $g\colon Z \to X$ be a finite \'etale morphism, such that $Y \times_X Z \to Z$ is a trivial \'etale morphism. The existence of $g$ is guaranteed by Lemma \ref{lemma:trivialisation}.

Definition \ref{defi:norm}(a) defines a norm map $\Nm_{Y \times_X Z/Z}\colon g^{-1} J_Y \to g^{-1}J_X$. Let $h\colon U \to X$ be an \'etale open of $X$ and $s \in J_Y(U)$ a local section.

The section $t=g^*s \in J_Y(U \times_X Z)$ descends by construction to $U$, that is, with respect to the two projections 
$$p_1,p_2\colon(U \times_X Z) \times_U (U \times_X Z) \to U$$
we have $p_1^* t = p_2^* t$. Furthermore, by construction of $g$ and since $J$ is a sheaf, we have a canonical isomorphism $J_Y(U) \simeq J_X(U)^F$. With respect to this identification we write $t=(s_i)_{i \in F} \in J_X(U \times_X Z)^F$. We have $p_1^* s_i = p_2^* s_i$ for all $i \in F$. This shows that $\Nm_{Y \times_X U \times_X Z/U \times_X Z}(t) = \prod_{i \in F} s_i$ also satisfies $p_1^*\Nm(t) = p_2^*\Nm(t)$. Therefore it descends to a well-defined section $\Nm_{Y/X}(s) \in J(U)$. 
We leave the verifications that the resulting map $\Nm_{Y/X}$ is a morphism of \'etale sheaves in abelian groups, and is independent of the choice of $g$, to the reader.
\end{construction}

The simplicial object $Y_{\bullet}$ in finite \'etale schemes over $X$ yields an augmented simplicial object in \'etale sheaves on the small \'etale site of $X$:
\begin{equation}\label{diag:J}
\xymatrix{
J \ar@{<-}[r]^{\Nm} & J_Y \ar@{<-}@<-0.7ex>[r] \ar@{<-}@<0.7ex>[r] & J_{Y\times_X Y} \ar@{<-}@<-1.4ex>[r] \ar[r] \ar@{<-}@<1.4ex>[r] & J_{Y \times_X Y \times_X Y} \ar@{<-}@{<-}@<-0.7ex>[r] \ar@{<-}@{<-}@<0.7ex>[r]  \ar@{<-}@<-2.3ex>[r] \ar@{<-}@<2.3ex>[r] & \cdots 
}
\end{equation}
The first incarnation of co-descent is the assertion that the augmentation of this diagram is a limit.

\begin{lemma}[Local co-descent]\label{lemma:co-descent}
Diagram \eqref{diag:J} is a limit diagram in the category of sheaves of abelian groups on the small \'etale site of $X$, that is, $J \simeq |J_{Y_{\bullet}}|$. Moreover, this holds in the derived $\infty$-category \'etale sheaves in abelian groups.
\end{lemma}

\begin{proof}
We only prove the second assertion (in the derived $\infty$-category) as it implies the first since $\Sh_{\text{\'et}}(X)$ is a full subcategory of $D_{\text{\'et}}(X)$. By virtue of Lemma \ref{lemma:trivialisation} there exists a finite \'etale morphism $Z \to X$, such that $Y \times_X Z \to Z$ is a trivial \'etale covering. Since the assertion of the lemma can be verified \'etale locally, we may assume without loss of generality that $Y \to X$ is of the shape $\bigsqcup_{i \in F} X \to X$, where $F$ is a finite set.

Since $J_{\bigsqcup_{i \in F}X} \simeq J_X^F$, we have to verify that $J_X \xleftarrow{\Nm_{Y/X}} |J_{X}^{\times F_{\bullet}}|$ is an isomorphism. Here, we denote by $F_{\bullet}$ the simplicial set which agrees with $\times_{0=1}^n F$ in level $n$. By the Dold--Kan correspondence (see \cite[Tag 019D]{stacks-project}) the geometric realisation of this simplicial abelian group is quasi-isomorphic to the chain complex
$$J_X^F \leftarrow J_X^{F \times F} \leftarrow J_X^{F \times F \times F} \leftarrow \cdots$$
with differential $\partial_n\colon J_X^{\times_{i=0}^n F} \to J_X^{\times_{i=0}^{n-1} F}$ given by the alternating sum of the face maps ($i = 0,\dots,n$)
$$\partial_{n,i}\colon J_X^{F^{n+1}} \to J_X^{F^n},$$
which sends $(g\colon F^n+1 \to J_X) \in J_X^{F^{n+1}}$ to 
$$\partial_{n,i}g\colon (c_0,\dots,c_{n-1}) \mapsto \prod_{c \in F} g(c_0,\dots,c_{i-1},c,c_{i},\dots,c_{n-1}).$$
A direct computation shows that the homology groups of this complex are $0$ in positive degrees and isomorphic to $J_X$ in degree $0$. This concludes the proof of the lemma.
\end{proof}

Strict Picard stacks naturally from a $2$-category (to be precise, a $(2,1)$-category). The corresponding ``homotopy category", obtained by identifying $2$-isomorphism $1$-morphisms, is well-understood: in SGA IV.3 \cite[XVIII.1.4]{SGA43} Deligne proved that the homotopy category of strict Picard stacks is equivalent to the full subcategory of the derived category of sheaves of abelian groups
$$D_{[-1,0]}(\Sh) \subset D(\Sh)$$
consisting of those complexes supported in degrees $[-1,0]$. The inverse to this embedding works as follows: to a length $1$ complex $[V_{-1} \to V_0]$ of sheaves of abelian groups one associates the quotient stack $[V_0/V_{-1}]$. It is clear that the quotient stack inherits a strict monoidal structure from this presentation.

Deligne's embedding can be enhanced to capture the full $2$-categorical structure of strict Picard stacks. It is well-known that the $2$-category of strict Picard stacks embeds as a subcategory into the derived $\infty$-category 
$$\D_{[-1,0]}(\Sh) \subset \D(\Sh)$$
of complexes of sheaves of abelian groups. This result is part of the mathematical folklore and can be deduced from the Dold-Kan correspondence \cite[Tag 019G]{stacks-project}.

\begin{rmk}\label{rmk:RGamma}
With respect to Deligne's equivalence we have $\Pb \simeq R\Gamma(X_S,J[1])$.
\end{rmk}

\begin{definition}{Norm map}
Let $f\colon Y \to X$ be a finite \'etale morphism which is surjective. The norm map $\Nm_{Y/X}\colon J_Y \to J_X$ induces a morphism of strict Picard stacks $\Nm_{Y/X} \colon \Pb_Y \to \Pb_X$.
\end{definition}

The simplicial scheme $Y_{\bullet}$ of Definition \ref{defi:descent}(b) yields a co-simplicial diagram of strict Picard stacks
\begin{equation}\label{diag:Nm}
\xymatrix{
\Pb_{X} & \Pb_{Y} \ar[l]_{\Nm} & \Pb_{Y\times_X Y} \ar@<-0.7ex>[l] \ar@<0.7ex>[l] & \Pb_{Y \times_X Y \times_X Y} \ar@<-1.4ex>[l] \ar[l] \ar@<1.4ex>[l] & \cdots \ar@<-0.7ex>[l] \ar@<0.7ex>[l] \ar@<-2.3ex>[l] \ar@<2.3ex>[l] 
}
\end{equation}

\begin{corollary}[Global co-descent]\label{cor:co-descent}
For a finite \'etale morphism $Y \to X$ of smooth proper curves the diagram \eqref{diag:Nm} is a colimit diagram in the $2$-category of strict Picard stacks, that is, the augmentation yields an equivalence $\Nm_{Y/X}\colon \Pb_G(X) \simeq |\Pb_G(Y_{\bullet})|.$ Furthermore this holds in the derived $\infty$-category $D(\Sh_{\text{big-\'et}}(\Spec k))$.
\end{corollary}

\begin{proof}
As before we only prove the second statement, as it implies the first. For an affine $k$-scheme $S$, the strict Picard groupoid $\Pb_G(S)$ corresponds to the complex $R\pi_*(X_S,J_{X_S}[1])$ in the derived $\infty$-category $D_{\text{\'et}}(S)$, where $\pi\colon X_S \to S$ is the canonical projection. We claim that the functor $R\pi_*\colon D^b_{\text{\'et}(X_S)} \to D^b_{\text{\'et}(S)}$ preserves small colimits. To see this one uses that for exact functors between stable $\infty$-categories, it suffices to show that the induced functor of triangulated categories
$$R\pi_*\colon \Ho(D_{\text{\'et}}(X)) \to \Ho(D_{\text{\'et}}(S))$$
preserves small coproducts (see \cite[Proposition 1.4.4.1(2)]{Lurie:ha}). For bounded complexes this follows from \cite[Tag 0739]{stacks-project}. The general case can be reduced to the bounded case as follows: let $(\F^{\bullet}_i)_{i \in I}$ be a small set of complexes of sheaves on the small \'etale site of $X_S$. We will show that the canonical map of complexes
$$\bigoplus_{i \in I} R\pi_* \F_i^{\bullet} \to R\pi_*\left(\bigoplus_{i \in I}\F_i^{\bullet}\right).$$
is a quasi-isomorphism. Since $\F_i^{\bullet} \simeq \varinjlim_{d \to -\infty} \tau_{\geq d}$, and $\tau_{\geq d}$ commute with small coproducts, and $R\pi_*$ with small limits, we can assume that the $\F^{\bullet}_i$ are uniformly bounded below. That is, there exists $c \in \mathbb{N}$, such that for all $i \in I$ the cohomology sheaves $\Hc^j(\F_i^{\bullet}) = 0$ if $j < c$.

We want to show for all $j \in \mathbb{N}$ that the induced morphism
$$\Hc^j\left(\bigoplus_{i \in I} R\pi_* \F_i^{\bullet}\right) \to \Hc^j\left(R\pi_*(\bigoplus_{i \in I}\F_i^{\bullet})\right)$$
is an isomorphism. Choose $d > j$, then this is equivalent to the induced morphism 
$$\Hc^j\left(\bigoplus_{i \in I} R\pi_*\tau_{\leq d} \F_i^{\bullet}\right) \to \Hc^j\left(R\pi_*(\bigoplus_{i \in I}\tau_{\leq d}\F_i^{\bullet})\right)$$
to be an isomorphism. The complexes $\tau_{\leq d} \F_i^{\bullet}$ are bounded and we have already remarked that $R\pi_*$ commutes with coproducts of bounded complexes.

The fact that $R\pi_*$ preserves colimits and Lemma \ref{lemma:co-descent} imply that we have an equivalence of Picard groupoids
$\Nm_{Y/X}\colon \Pb_G(X)(S) \simeq |\Pb_G(Y_{\bullet})(S)|.$
This concludes the proof.
\end{proof}

\subsection{Langlands duality and the Prym varieties for quasi-split reductive groups}

The following theorem is due to Donagi--Pantev (\cite[Theorem A]{MR2957305}) in characteristic $0$ and due to Chen--Zhu (\cite[Theorem 1.2.1]{chenzhu}) in its more general incarnation.

\begin{theorem}[Donagi--Pantev, Chen--Zhu]\label{thm:duality}
Let $\Gb/k$ be a split reductive group scheme over $k$, such that $\charac(k)$ is $0$ or, $|W|$ is not divisible by $\charac(k)$. We denote by $\widehat{G}$ the Langlands dual group scheme. For $X/k$ a smooth projective curve we obtain an isomorphism $\A_{\Gb} \simeq \A_{\widehat{\Gb}}$, which maps $\A_{\Gb}^{\Diamond}$ to $\A_{\widehat{\Gb}}^{\Diamond}$. With respect to this identification the abelian group stacks $\Pb_{\Gb}$ and $\Pb_{\widehat{\Gb}}$ are dual, that is there is an equivalence
$$\AJ^*\colon \Pb_{\Gb} \simeq \Pb_{\widehat{\Gb}}.$$
\end{theorem}

The following observation plays an important role in extending this theorem to quasi-split reductive groups. We refer the reader to the paragraph preceding \ref{thm:duality_quasi-split} for context and references of the map $\AJ^*$.

\begin{rmk}\label{rmk:dual}
For $f\colon Y \to X$ a finite \'etale map, we have $\Res_{Y/X}^{\vee} = \Nm_{Y/X}\colon \Pb_{G,X}\to \Pb_{G,Y}$, that is, the diagram
\[
\xymatrix{
\Pb_{G,X} \ar[d]_{\Res_{Y/X}} \ar[r]^-{\AJ^*}_{\simeq} & \Pb_{\widehat{G},X}^{\vee} \ar[d]^{\Nm_{Y/X}^{\vee}} \\
\Pb_{G,Y} \ar[r]^-{\AJ^*}_{\simeq} & \Pb_{\widehat{G},Y}^{\vee}
}
\]
commutes. Commutativity of this diagram is a consequence of commutativity of 
\[
\xymatrix{
\Cc_Y \times \Xb_*(\widehat{\Tb}) \ar[r]^-{\AJ_Y} \ar[d]_f & \Pb_{\widehat{G},Y} \ar[d]^{\Nm_{Y/X}} \\
\Cc_X \times \Xb_*(\widehat{\Tb}) \ar[r]^-{\AJ_X} & \Pb_{\widehat{G},X}.
}
\]
Indeed, the upper horizontal map sends a pair $(y,\lambda)$, where $y$ is a point of the spectral cover $\Cc_Y$, and $\lambda \in \Xb^*(\widehat{\Tb})$ to the $\mathbb{W}$-equivariantisation of the $\widehat{\Tb}$-torsor $\Oo_{\Cc_Y}(\lambda y)$. The image of the norm map of this $\widehat{\Tb}$-torsor is the $\mathbb{W}$-equivariantisation of the $\widehat{\Tb}$-torsor $\Oo_{\Cc_X}(\lambda f(y))$ which agrees with $\AJ(f(y),\lambda)$.
\end{rmk}

We combine this remark with the co-descent picture of the previous paragraph \ref{cor:co-descent} in order to extend \ref{thm:duality} to the case of quasi-split reductive group schemes on $X$.

\begin{corollary}\label{cor:duality}
Let $X/k$ be a smooth projective curve over a field $k$, and $G/X$ a quasi-split reductive group scheme, that is, an outer form of a split reductive group $\Gb/k$. For $\charac(k)$ is $0$ or $|W|$ is not divisible by $\charac(k)$, we obtain an isomorphism $\A_{G} \simeq \A_{\widehat{G}}$, which maps $\A_{G}^{\Diamond}$ to $\A_{\widehat{G}}^{\Diamond}$. With respect to this identification the Beilinson $1$-motives $\Pb_{G}$ and $\Pb_{\widehat{G}}$ are dual.
\end{corollary}

\begin{proof}
Let $Y \xrightarrow{f} X$ be a finite \'etale morphism, such that $f^*G$ is isomorphic to the split group $Y$-scheme $\Gb \times Y$. According to Remark \ref{rmk:dual} we have a commutative diagram
\[
\xymatrix{
\Pb_{X}^{\vee} \ar[r]^{\Nm^{\vee}} & \Pb_{Y}^{\vee} \ar[d]^{\AJ^*} \ar@<-0.7ex>[r] \ar@<0.7ex>[r] & \Pb_{Y\times_X Y}^{\vee} \ar[d]^{\AJ^*} \ar@<-1.4ex>[r] \ar[r] \ar@<1.4ex>[r] & \Pb_{Y \times_X Y \times_X Y}^{\vee} \ar[d]^{\AJ^*} \ar@<-0.7ex>[r] \ar@<0.7ex>[r]  \ar@<-2.3ex>[r] \ar@<2.3ex>[r] & \cdots \\
\widehat{\Pb}_{X} \ar[r]^{\Res} & \widehat{\Pb}_{Y} \ar@<-0.7ex>[r] \ar@<0.7ex>[r] & \widehat{\Pb}_{Y \times_X Y} \ar@<-1.4ex>[r] \ar[r] \ar@<1.4ex>[r] & \widehat{\Pb}_{Y \times_X Y \times_X Y} \ar@<-0.7ex>[r] \ar@<0.7ex>[r]  \ar@<-2.3ex>[r] \ar@<2.3ex>[r] & \cdots
}
\]
with each row being an augmented co-simplicial object in strict Picard stacks. The augmentation map of the first row induces an equivalence of $\Pb_{a,X}^{\vee}$ with the geometric realisation of the first row, as can be seen by dualising the co-descent statement of Corollary \ref{cor:co-descent}. The augmentation map of the second row induces an equivalence of $\widehat{\Pb}_{a,X}$ with the geometric realisation, as a consequence of descent theory for $\widehat{\Pb}$ (see Lemma \ref{lemma:descent}).

Since the morphisms $\AJ^*$ are equivalences of strict Picard stacks, we obtain an induced equivalence of geometric realisations $\Pb_{a,X}^{\vee} \simeq \widehat{\Pb}_{a,X}$. 
\end{proof}

\bibliographystyle{amsalpha}
\bibliography{master}
\end{document}